\documentclass[11pt,reqno]{amsart}
\usepackage{amsmath}
\usepackage{amssymb}
\usepackage{amsthm}
\usepackage{amscd, amsfonts, mathrsfs}
\usepackage{amsaddr}

\usepackage{mathtools}

\usepackage{bigints} 

\usepackage{cases}
\usepackage{pdfsync}

\usepackage[dvips]{epsfig}
\usepackage{verbatim} 

\usepackage{epsf}
\usepackage[usenames]{color}
\usepackage[bookmarksnumbered,pdfpagelabels=true,plainpages=false,colorlinks=true,
            linkcolor=black,citecolor=black,urlcolor=blue]{hyperref}

\theoremstyle{plain}
\newtheorem{theorem}{Theorem}[section]
\newtheorem{lemma}[theorem]{Lemma}
\newtheorem{proposition}[theorem]{Proposition}
\newtheorem{corollary}[theorem]{Corollary}

\theoremstyle{definition}
\newtheorem{remark}[theorem]{Remark}
\newtheorem{notation}[theorem]{Notation}
\newtheorem{definition}[theorem]{Definition}

\numberwithin{equation}{section}

\newcommand{\cA}{\mathcal A}
\newcommand{\cB}{\mathcal B}

\newcommand{\cD}{\mathcal D}
\newcommand{\cE}{\mathcal E}
\newcommand{\cF}{\mathcal F}
\newcommand{\cG}{\mathcal G}
\newcommand{\cH}{\mathcal H}
\newcommand{\cI}{\mathcal I}

\newcommand{\cL}{\mathcal L}
\newcommand{\cM}{\mathcal M}
\newcommand{\cN}{\mathcal N}
\newcommand{\cO}{\mathcal O}
\newcommand{\cP}{\mathcal P}
\newcommand{\cQ}{\mathcal Q}

\newcommand{\cS}{\mathcal S}
\newcommand{\cT}{\mathcal T}
\newcommand{\cU}{\mathcal U}

\newcommand{\cW}{\mathcal W}
\newcommand{\cX}{\mathcal X}

\newcommand{\cZ}{\mathcal Z}

\newcommand{\ccA}{\mathscr{A}}
\newcommand{\ccB}{\mathscr{B}}
\newcommand{\ccC}{\mathscr{C}}

\newcommand{\ccE}{\mathscr{E}}

\newcommand{\al}{\alpha}
\newcommand{\be}{\beta}
\newcommand{\ga}{\gamma}
\newcommand{\Ga}{\Gamma}
\newcommand{\de}{\delta}

\newcommand{\la}{\lambda}
\newcommand{\La}{\Lambda}
\newcommand{\si}{\sigma}

\newcommand{\om}{\omega}
\newcommand{\Om}{\Omega}

\newcommand{\RR}{\mathbb R}

\newcommand{\rar}{\rightarrow}

\newcommand{\ve}{\varepsilon}

\newcommand{\id}{\operatorname{id}}
\newcommand{\dive}{\operatorname{div}}
\newcommand{\curl}{\operatorname{curl}}
\newcommand{\p}{\parallel}

\newcommand{\psQ}{\mathcal{Q}^{(3)}(\partial^2 f, \partial^3)}
\newcommand{\psq}{\mathcal{Q}^{(2)}(\partial^2 f, \partial^2)}
\newcommand{\pres}{\pi} 
\newcommand{\Emb}{\text{Emb}}
\newcommand{\Rin}{\overline{R}} 
\newcommand{\IntDom}{\underset{0 \leq t \leq T}{\bigcup} \{t\} \times \Om(t)}

\newcommand{\mss}{\hspace{0.2cm}}

\title[Surface tension]{The free boundary Euler equations with large surface tension}

\author[Disconzi]{ \vspace{-0.5cm} Marcelo M. Disconzi}
\address{ \vspace{-0.5cm} Department of Mathematics\\
Vanderbilt University\\ Nashville, TN, USA}
\email{marcelo.disconzi@vanderbilt.edu}
\thanks{Marcelo M. Disconzi is partially supported by NSF grant
1305705.}

\author[Ebin]{ \vspace{-0.5cm} David G. Ebin}
\address{ \vspace{-0.5cm} Department of Mathematics\\
Stony Brook University\\ Stony Brook, NY, USA}
\email{ebin@math.sunysb.edu}

\begin{document}

\begin{abstract}
We study the free boundary Euler equations with surface tension 
in three spatial dimensions, showing that the equations are well-posed if 
the coefficient of surface tension is positive.   Then we prove that
under natural assumptions, the solutions
of the free boundary motion converge to solutions of the Euler equations 
in a domain with fixed boundary when the coefficient of surface tension tends to infinity.
\end{abstract}
\keywords{large surface tension, free boundary, Euler equations, incompressible fluid}

\maketitle

\vspace{-1.2cm}

\tableofcontents

\section{Introduction \label{intro}}

Consider the initial value problem for the motion of an incompressible inviscid fluid 
with free boundary whose equations of motion are 
given in Lagrangian coordinates by (see below for the equations in Eulerian
coordinates)
\begin{subnumcases}{\label{free_boundary_full}}
 \ddot{\eta}  = - \nabla p \circ \eta   & in $ \Om$, \label{basic_fluid_motion_full} \\
  \operatorname{div} (u) = 0  & in  $\eta(\Om)$, \label{equation_p_full} \\
 \left. p \right|_{\partial \eta(\Om)} = \kappa \cA   & on  $\partial \eta(\Om)$, \label{bry_p_full} \\
 \eta(0) = \id,~\dot{\eta}(0)=u_0,
\end{subnumcases}
\noindent where
 $\Om$ is a domain in $\RR^n$; $\eta(t,\cdot)$ is, for each $t$, a volume 
preserving embedding $\eta(t):\Om \rar \RR^n$ representing the fluid 
motion, with $t$ thought of as the time variable ($\eta(t, x)$ is the
 position at time $t$ of the fluid 
particle that at time
zero was at $x$); $``\,\dot{~}\,"$ denotes 
derivative with respect to $t$; $\Om(t) = \eta(t)(\Om)$;
$u: \Om(t) \rar \RR^n$ is a 
vector field on $\Om(t)$ defined by $u = \dot{\eta} \circ \eta^{-1}$
(it represents the fluid velocity); 
$\cA$ is the mean curvature of the boundary of the domain $\Om(t)$; $p$ is a real valued function
on $\Om(t)$
called the pressure; finally,
$\kappa$ is a non-negative constant known as the coefficient of surface tension.
$\id$ denotes the identity map, $u_0$ is a given
divergence free vector field on $\Om$, 
 and 
$\operatorname{div}$ means divergence.
The unknowns are the fluid motion
$\eta$ and the pressure $p$, but notice that the system (\ref{free_boundary_full}) is coupled in a non-trivial fashion in the
sense that the other quantities appearing in (\ref{free_boundary_full}), namely $u$, $\cA$, and $\Om(t)$, depend 
explicitly or implicitly on $\eta$ and $p$.

With suitable assumptions, we shall prove 
the following result, concerning the 
existence of solutions to (\ref{free_boundary_full}) and
the behavior of solutions when the coefficient of  surface tension is large, i.e., in the limit 
$\kappa \rar \infty$. A precise statement is given in theorem
 \ref{main_theorem} below. 

\vskip 0.25cm
\textbf{Theorem (Main Result, see theorem \ref{main_theorem} for precise statements).}  
{\it Under appropriate conditions on the initial condition $u_0$ and on 
$\partial \Om$, we have:

1) If $\kappa > 0$, then
(\ref{free_boundary_full}) is well posed.

2) Consider a family 
of initial conditions $u_{0\kappa}$ parametrized 
by the coefficient of surface tension that converges,
when $\kappa \rar \infty$,  to 
a divergence free and tangent to the boundary vector field $\vartheta_0$.
Then, the corresponding solutions $\eta_\kappa$
 to (\ref{free_boundary_full}) 
converge to the solution of the incompressible Euler equations on the fixed domain $\Om$, given by
\begin{subnumcases}{\label{Euler}}
 \ddot{\zeta}  = - \nabla \pres \circ \zeta,
\label{Euler_edo} \\
 \dive(\dot{\zeta}\circ \zeta^{-1}) = 0,
\label{Euler_div}
\label{Euler_incompressible}  \\
\zeta(0) = \id,~\dot{\zeta}(0) = \vartheta_0.
\label{Euler_initial_conditions}
\end{subnumcases}
\noindent Here, $\zeta(t, \cdot)$ is, for each $t$, a volume preserving 
diffeomorphism  $\zeta(t): \Om \rar \Om$.
}
\vskip 0.25cm
\begin{remark}
It is well known that the pressure $\pres$ in the incompressible Euler equations is not an independent
quantity, since it is completely determined by the velocity vector field $\vartheta = \dot{\zeta}\circ\zeta^{-1}$ (see, e.g., \cite{EM}).  
\end{remark}

We remind the reader that in Eulerian 
coordinates, equations (\ref{free_boundary_full}) and (\ref{Euler}) take,
respectively, the following forms:
\begin{gather}
\begin{cases}
\frac{\partial u}{\partial t} + \nabla_u u  = - \nabla p  
  & \text{ in } \IntDom , \\
 \dive(u) = 0  & \text{ in  }  \Om(t), \\
  p  = \kappa \cA   & \text{ on } \partial \Om(t),  \\
 u(0) = u_0,
\end{cases}
\end{gather}
and
\begin{gather}
\begin{cases}
\frac{\partial \vartheta}{\partial t} + \nabla_\vartheta \vartheta 
 = -\nabla \pres
  & \text{ in }  [0,T]\times\Om, \\
 \dive(\vartheta) = 0  & \text{ in }  \Om,  \\
   \langle \vartheta, \nu \rangle = 0   & \text{ on }  \partial \Om,  \\
 \vartheta(0) = \vartheta_0,
 \end{cases}
\end{gather}
where in the free boundary case, $p$, $\kappa$, $\cA$ and $\pi$, 
are as before, $u$ is the velocity field with $u_0$ as its initial value.
In the fixed boundary case $u$ is replaced by $\vartheta$ which has an initial value
$\vartheta_0.$ The other symbols have the same meaning,  except $\nu$ which is
the unit outer normal
to $\partial \Om$.  In this case, our theorem states that $u$ converges to $\vartheta$ as $\kappa$ goes to infinity.

In Eulerian coordinates, the free boundary Euler equations also carry a boundary condition
stating that the normal speed of the moving boundary equals
to $\langle u, N \rangle$,
where $N$ is the unit outer normal to $\partial \Om(t)$. This is the same as saying 
that the vector field $\partial_t + \nabla_u$ is tangent to $\IntDom$.

In order to state the main result, we need to introduce some definitions.
Given manifolds $M$ and $N$, denote by $H^s(M,N)$
the space of maps of Sobolev class $s$ between $M$ and $N$; that is, maps with derivatives up to order $s$ in $L^2$. 
For $s > \frac{n}{2} +2$ define
\begin{gather}
 \cE_\mu^s(\Om)= \cE_\mu^s  = \Big \{ \eta \in H^s(\Om,\RR^n) ~ \Big | ~ J(\eta) = 1, \eta^{-1} 
\text{ exists and belongs to }  H^s( \eta(\Om), \Om) \Big \},
\nonumber
\end{gather}
\noindent where $J$ is the Jacobian. $\cE_\mu^s(\Om)$ is 
therefore the space of $H^s$-volume-preserving embeddings of $\Om$ into $\RR^n$. 
Define also
\begin{gather}
 \cD_\mu^s(\Om)=  \cD_\mu^s = \Big \{ \eta \in H^s(\Om,\RR^n) ~ \Big | ~ J(\eta) = 1, \eta: \Om \rar \Om  
\text{ is bijective and $\eta^{-1}$ belongs to }  H^s \Big \},
\nonumber
\end{gather}
\noindent so that $\cD^s_\mu(\Om)$ is the space of $H^s$-volume-preserving diffeomorphisms of $\Om$.
Notice that $\cD_\mu^s(\Om) \subseteq \cE_\mu^s(\Om)$.

Let $\cB_{\de_0}^{s+2}(\partial \Om)$ be the open ball about zero  of radius $\de_0$ inside $H^{s+2}(\partial \Om)$. We shall prove that if $\de_0$ is sufficiently small, then the map
\begin{align}
& \varphi: \cB_{\de_0}^{s+2}(\partial \Om) \rar H^{s+\frac{5}{2}}(\Om), 
\nonumber \\
& \varphi(h) = f,
\nonumber
\end{align}
where $f$ satisfies 
\begin{subnumcases}{\label{jac_nl}}
J ( \id + \nabla f) = 1 & in $\Om$,  \label{jac_nl_int} \\
f = h & on $\partial \Om$, \label{jac_nl_bry}
\end{subnumcases}
is a well defined $C^1$ map, and $\varphi(\cB_{\de_0}^{s+2}(\partial \Om))$
is a smooth submanifold of $H^{s + \frac{5}{2}}(\Om)$.

 We note that the map $\varphi$ solves a non-linear analog of the Dirichlet problem, that of extending $h$
  from $\partial \Om$ to a  function on $\Om$.  In fact if $\Omega \subseteq \RR^3$ with standard coordinates then (\ref{jac_nl_int}) can be written 
\begin{gather}
\Delta f + f_{xx}f_{yy} + f_{xx}f_{zz} + f_{yy}f_{zz} - f_{xy}^2 - f_{xz}^2 
- f_{yz}^2 + \det (D^2 f) = 0,
\nonumber
\end{gather}
so the difference between (\ref{jac_nl}) and the Dirichlet problem are the non-linear terms, which will be shown to be small. The purpose of (\ref{jac_nl_int}) is to ensure that $\id + \nabla f$ is volume preserving.   

 Using $\varphi$ we then 
construct another map
\begin{align}
\begin{split}
& \Phi: \cD_\mu^s(\Om) \times \varphi(\cB_{\de_0}^{s+2}(\partial \Om)) \rar \cE_\mu^s(\Om), 
 \\
{\rm defined \; by} \:\:\:\;& \Phi(\beta , f) = (\id + \nabla f) \circ \beta.
\end{split}
\label{big_phi}
\end{align}
Thus $\Phi(\beta,f)$ is the composition of two volume preserving maps.

We define $\mathscr{E}_\mu^s(\Om) \subseteq \cE_\mu^s(\Om)$ by
\begin{gather}
\mathscr{E}_\mu^s(\Om)
 = \Phi\big( \cD_\mu^s (\Om) \times \varphi(\cB_{\de_0}^{s+2}(\partial \Om))  \big).
\nonumber
\end{gather}
Notice that  since $\beta \in \cD_\mu^s(\Om)$, we have $\beta(\partial \Om) = \partial \Om$. 
Therefore, solutions $\eta$ to (\ref{free_boundary_full}) that belong
to $\ccE^s_\mu(\Om)$ decompose to a part fixing the boundary and
a boundary displacement, i.e., 
\begin{gather}
\eta = (\id + \nabla f) \circ \beta.
\label{decomp_eta}
\end{gather}
The decomposition (\ref{decomp_eta}) is one of the main ingredients 
of our proof, and establishing that $\eta\rar \zeta$ as $\kappa \rar \infty$
will be done by showing that the boundary displacement $\nabla f$ goes to zero
as $\kappa \rar \infty$.
We shall also show that, under our hypotheses,  $\nabla f$ is in fact $\frac{3}{2}$ degree smoother
than $\eta$ (though $\beta$ is as regular as $\eta$). In this sense,
we work with embeddings which
have smoother boundary values. 
A more detailed discussion for the motivation
for introducing $\ccE^s_\mu(\Om)$ is given in \cite{DE2d}.

We are now ready to state our main result. 
The usual 
decomposition of a vector field $X$ into its gradient part, $QX$,  and
 divergence free and tangent to the boundary part, $PX$,
 which appears in the hypotheses of the theorem, is reviewed in section \ref{auxiliary}. We note that if $X$ is not tangent to the boundary, 
we will have $QX \neq 0$ even if  $\dive(X) =0$. We let $\p \cdot \p_s$ denote the Sobolev norm.
We shall state and prove the theorem only in  three-spatial dimensions, since this is the case
of primary interest. We point out, however, that basically the same proof works in higher dimensions, 
except that the calculations  increase significantly in complexity.

\begin{theorem}
Let  $s > \frac{3}{2} +2$,  $\Om$ be a bounded 
domain in $\RR^3$ with a smooth 
boundary
and $u_0 \in H^{s}(\Om, \RR^3)$ be
a divergence free
vector field. Denote 
by $Q u_0$ the gradient part of $u_0$. 

1) If 
 $\kappa > 0$, then
there exist a $T_\kappa > 0$ and a unique
solution 
$(\eta_\kappa, p_\kappa)$ to (\ref{free_boundary_full}),
with initial condition $u_0.$ The solution satisfies:
\begin{align}
\begin{split}
&\eta_\kappa \in  C^0  ( [0,T_\kappa), \cE_\mu^{s}(\Om)   ), \,
\dot{\eta}_\kappa \in L^\infty( [0,T_\kappa), H^{s}(\Om) ), \, 
\ddot{\eta}_\kappa \in L^\infty( [0,T_\kappa), H^{s-\frac{3}{2}}(\Om) ), \\
& p_\kappa \in L^\infty([0,T_\kappa), H^{s-\frac{1}{2}}(\Om_\kappa(t))),
\, \text{ where } \, \Om_\kappa(t) = \eta_\kappa(t)(\Om).
 \end{split}
 \nonumber
\end{align}
Moreover, 
for each $t \in [0,T_\kappa)$, 
\begin{gather}
\eta_\kappa(t) \; {\rm is \; in} \; \ccE_\mu^{s}(\Om).
\nonumber
\end{gather}

2) Let $\{ u_{0\kappa} \} \subset H^{s}(\Om, \RR^3)$ be a family
of divergence free vector fields parametrized by
the coefficient of surface tension $\kappa$, satisfying $\p Q u_{0\kappa} \p_s 
\leq \frac{C}{\sqrt{\kappa}}$ for some constant $C$, and such 
that $u_{0\kappa}$ converges in $H^s(\Om, \RR^3)$, as $\kappa \rar \infty$, to a divergence free
 vector field $\vartheta_0$ which is  tangent to the boundary. 
Let
$\zeta \in C^1 \big ( [0,T], \cD_\mu^{s}(\Om) \big ) $ be the  
solution to 
(\ref{Euler}) with initial 
condition $\vartheta_0$, defined on some\footnote{Notice
that $[0,T)$ is not a maximal interval of existence since the solution
$\zeta$ exists on the closed interval $[0,T]$.  It can however be arbitrarily close to the maximum.} time interval $[0,T]$.
Assume that the mean curvature of $\partial \Om$ is constant, and 
 let $(\eta_\kappa,p_\kappa)$ be the solution to (\ref{free_boundary_full})
with initial condition $u_{0\kappa}$ and defined 
on a time interval $[0, T_\kappa)$, as stated in part (1) above.
Finally, assume that $[0,T_\kappa)$ is taken as the maximal interval 
of existence for the solution $(\eta_\kappa,p_\kappa)$.
Then, if $T$ is sufficiently small,  we find that 
$T_\kappa \geq T$ for all $\kappa$ sufficiently large, and 
$\eta_\kappa(t) \rar \zeta(t)$ as a continuous 
curve in $\cE^{s}_\mu(\Om)$ as $\kappa \rar \infty$.
Also, $\dot{\eta}_\kappa(t) \rar \dot{\zeta}(t)$ in $H^s(\Om)$ 
as $\kappa \rar \infty$.
\label{main_theorem}
\end{theorem}

\begin{remark} 
We stress that $\zeta$ in the theorem exists and is unique by \cite{EM}. 
\end{remark}

\begin{remark}
Note that since $\eta_\kappa(t) \in \ccE_\mu^{s}(\Om)$, in particular
 $\eta_\kappa$ satisfies decomposition (\ref{decomp_eta}), and 
$\Om(t)$ has a $H^{s+1}$-regular boundary.
\end{remark}

Since existence of solutions to (\ref{free_boundary_full}) has already been established 
in the literature (see \cite{CS,Lin} and comments below), the main result of this paper is the study 
of the singular limit
$\kappa \rar \infty$ and the corresponding convergence of solutions, i.e., 
part (2) of theorem \ref{main_theorem}.
It should be stressed that, in this regard, the \emph{hypothesis of constant 
mean curvature at time zero
cannot be removed}. Indeed, on physical grounds, 
if the mean curvature $\cA_{\partial \Om}$ of $\partial \Om$ 
is not constant, one expects that $\Om(t)$ will develop
high-frequency oscillation for large $\kappa$, but solutions
will not converge in the limit $\kappa \rar \infty$. This is because
the dynamics  $\eta_\kappa$ 
can be thought of as mimicking the behavior of motions with
a strong constraining force, as we explain in section
\ref{geometric} (see also remark \ref{remark_constant_mean_curvature}).
We also notice that, from the point of view of $\kappa \rar \infty$, the assumption that $Q u_0$
be small in part (2) of theorem \ref{main_theorem} is natural, since $Q u_0$ has to be small
if $u_0$ is near $\vartheta_0$.

Despite the fact that, as mentioned above, well-posedness of (\ref{free_boundary_full}) is known in the literature,
our methods are entirely different than previous
works, thus of independent interest.

\begin{remark}
In this paper we treat exclusively the three-dimensional case, but the same proof works in $n=2$. 
In fact, the calculations of section \ref{section_geom_boundary} simplify in two dimensions.
\end{remark}

\begin{remark}
In terms of the more familiar Eulerian coordinates,
theorem \ref{main_theorem} asserts, in particular,
the convergence of $u_\kappa(t) \circ \eta_\kappa(t)$ to 
$\vartheta(t) \circ \zeta(t)$
in $H^s(\Om)$. 
We believe, however, that in 
this problem the statement
in Lagrangian coordinates looks more natural. This is because  $u_\kappa$ and $\vartheta$ are defined in different domains, and therefore only the convergence
of $u_\kappa \circ \eta_\kappa$ to $\vartheta \circ \zeta$, and not
of $u_\kappa$ to $\vartheta$, makes sense. In other words, to state
the convergence in Eulerian coordinates, we still need to invoke 
the flows $\eta_\kappa$ and $\zeta$. 
\end{remark}

\begin{remark}
We  point out that, due to a classical result 
of Aleksandrov \cite{Ale}, if $\cA_{\partial \Om}$ is constant then
$\partial \Om$ is a sphere. This fact, however, is not used directly in our
proof. Had we not known of Aleksandrov's result,
our proof would still follow solely from the 
constancy of $\cA_{\partial \Om}$. In particular,  we expect theorem
\ref{main_theorem} to hold in the more general situation 
where $\partial \Om$  is the boundary of a region inside a Riemannian 
manifold with constant mean curvature with respect to the 
corresponding Riemannian metric.
\end{remark}

Theorem \ref{main_theorem} gives some interesting insight into the structure
of solutions to the free boundary Euler equations: by uniqueness, any solution
to (\ref{free_boundary_full}) satisfying our hypotheses will obey decomposition 
(\ref{decomp_eta}), regardless of the method employed to construct such solutions.

As already pointed out, one wishes to show that $\nabla f_\kappa$ in decomposition (\ref{decomp_eta})
goes to zero when $\kappa \rar \infty$ in order to establish theorem \ref{main_theorem}.
Since in (\ref{decomp_eta}) $\be_\kappa \in \cD_\mu^s(\Om)$, one expects that not only
does $\eta_\kappa \rar \zeta$, but  $\beta_\kappa \rar \zeta$ as well. This is in fact the case:

\begin{corollary}
With the same assumptions and notation of theorem \ref{main_theorem}, consider the
convergence $\eta_\kappa \rar \zeta$. Then we also get
$\beta_\kappa(t) \rar \zeta(t)$ and $\dot{\be}_\kappa(t) \rar \dot{\zeta}(t)$ in $H^s(\Om)$,
 $\nabla f_\kappa(t) \rar 0$ in $H^{s+\frac{3}{2}}(\Om)$, and $\nabla \dot{f}_\kappa(t)
\rar 0$ in $H^s(\Om)$.
\label{corollary_convergence}
\end{corollary}

It is interesting to note that while $\eta_\kappa \rar \zeta$ and $\dot{\eta}_\kappa \rar
\dot{\zeta}$, in general the corresponding pressures do
not converge, even if the initial data are $C^\infty$.
To see this, we first point out that since
$\pi$ is defined up to an additive constant, and thus only $\nabla \pi$ is well-defined,
one can only speak of convergence of $\nabla p_\kappa$ to $\nabla \pi$, and not of 
$p_\kappa$ to $\pi$. Consider the two-dimensional 
 case for simplicity, pick any function $f$ which is constant on $\partial \Om$ and let $u_0 = (f_y, -f_x)$. Then $u_0$ will be divergence free and tangent to the boundary. The pressure for 
 (\ref{Euler}) at time zero will then satisfy:
$$ - \Delta \pi = 2(f_{xy}^2 - f_{xx}f_{yy}),$$
and $\nabla_{\nu}\pi$ will equal zero on $\partial \Omega$.  Thus $\pi$ in general will not be constant on $\partial \Omega$, so one cannot expect that
$\nabla p_\kappa,$ the solution of 
(\ref{free_boundary_full})
 will converge to $\nabla \pi$, as $\kappa \rar \infty$, even at time zero.
 As a consequence, convergence of the second time derivatives, i.e., $\ddot{\eta}_\kappa \rar
 \ddot{\zeta}$, generally  fails
 (see the analogous results in \cite{E2, ED}). However, due to the convergence of the first time
 derivatives, the pressures over any positive time interval converge:
 
\begin{corollary}
Under the same assumptions and notation of theorem \ref{main_theorem}, one has
$\int_{t_a}^{t_b} \nabla p_\kappa \circ \eta_\kappa \rar \int_{t_a}^{t_b} \nabla \pi \circ \zeta$ in $H^s(\Om)$,
for any $ 0 \leq t_a < t_b \leq T$.
\end{corollary} 

The convergence
part of theorem \ref{main_theorem}, namely, part (2), was proven  
by the authors in two
spatial dimensions in \cite{DE2d}. However, compared
to \cite{DE2d}, theorem \ref{main_theorem} is self-contained, 
in the sense that the existence of $\eta_\kappa$ is established before
proving the convergence $\eta_\kappa \rar \zeta$, whereas in \cite{DE2d}
we relied on the existence results in Coutand and Shkoller \cite{CS} in order
to obtain the convergence. 

The mathematical study of equations (\ref{free_boundary_full}) has a long 
history, although for a long time
results only under restrictive conditions had been achieved. In particular,
a great deal of work has been devoted to irrotational flows, in which 
case the free boundary Euler equations reduce to the well-known
water-wave equations.
See \cite{ Amborse, AmbroseMasmoudi, Craig, Lannes, Nalimov, Wu, Yosihara}.
More recent results addressing the question
of global existence can be found in \cite{IT, IonPus, WuWaterWavesGlobal} and references therein.

Not surprisingly, when equations (\ref{free_boundary_full}) are considered in full generality, well-posedness becomes a yet more delicate issue, and most of the results are quite recent. In this regard, 
Ebin has showed that the problem is ill-posed if $\kappa =0$ \cite{E0}, although
Lindblad proved well-posedness for $\kappa = 0$ when the so-called ``Taylor sign condition'' 
holds \cite{Lin,LindNor}; 
see also \cite{ChLin} (the linearized problem was also investigated by Lindblad in \cite{Lin2}).
When $\kappa > 0$, a priori estimates have been obtained by
 Shatah and Zeng \cite{ShatahZeng}, with well-posedness being finally 
established by Coutand and Shkoller \cite{CS, CSB} (see also \cite{Sch}). 
See also \cite{ShatahZeng2}.
Coutand and Shkoller also established the convergence of solutions
in the limit $\kappa \rar 0^+$ in the case that the Taylor sign condition holds.
Other recent results, including the study of the compressible free boundary
Euler equations and singularity formation,
 are \cite{Cas1, Cas2, CS2, CS3, CS4, CSH, CSL, IonFef}. We point out that
the analogous free boundary problem for viscous fluids
was first and extensively studied by Solonnikov \cite{MogSol, Sol1, Sol2, Sol3, 
Sol4, 
Sol5, Sol6, Sol7}, with some more recent
advances found in 
\cite{KPW, PS, Se1, Se2, Se3} and references therein.

Lindblad's result \cite{Lin} is based on a Nash-Moser iteration, while
Coutand and Shkoller  \cite{CS}
obtained existence and convergence when $\kappa \rar 0^+$
by developing a technique they call convolution by layers. The results
here presented, based on the decomposition
defined by (\ref{big_phi}),
 provide yet a third, different
method of proof (valid for $\kappa > 0$).
We point out that,
with exception of the authors' work \cite{DE2d} in two-dimensions, 
the  limit $\kappa \rar \infty$
does not seem to have been investigated in the literature before.

\begin{notation} We reserve $\Om$ for the \emph{fixed} domain,
with $\Om(t)$ being always the domain at time $t$, i.e., $\Om(t) = \eta(t)(\Om)$. Of course,
$\Om(0) = \Om$. In several parts of the paper the subscript $\kappa$ 
will be dropped for the sake of 
notational simplicity. 
\end{notation}

\begin{notation}
$H^s(\Om)$ and $H^s(\partial \Om)$ denote, respectively,
the Sobolev spaces of functions on $\Om$ and $\partial \Om$,
with norms $\p \cdot \p_s$ and $\p \cdot \p_{s,\partial}$.
$H^s(\Om, \RR^n)$ etc are similarly understood, although
when the manifolds are clear from the context, we simply
write  $H^s$ or $H^s(M)$ for $H^s(M,N)$.
Notice that $H^0$ denotes the $L^2$
space, with norm $\p \cdot \p_0$.
$H^s_0$ denotes the Sobolev space modulo constants.  
We use both $\nabla$ and $D$ to denote the derivative. $D_w$ is the directional 
derivative in the direction of $w$, $w$ a vector. 
The letter $C$ will be used to denote several different constants that
appear in the estimates. Sometimes we write $C = C(a,b,\dots)$ to indicate
the dependence of $C$ on $a,b,\dots$.
We use the following abridged notation for partial derivatives:
$\frac{\partial}{\partial x^i} \equiv \partial_i$, $\frac{\partial^2}{\partial x^i \partial x^j} \equiv \partial_{ij}$,
$\frac{\partial^3}{\partial x^i \partial x^j x^k} \equiv \partial_{ijk}$, etc.
\end{notation}

\subsection{Organization of the paper} This paper is organized as follows.
In section \ref{scaling} we make some remarks about the role of surface tension.
In section \ref{geometric} we give a geometric interpretation of our theorem
in terms of curves in the group of volume preserving diffeomorphisms and volume
preserving embeddings. In section \ref{auxiliary}, we state several known results
that will be used and fix some notation. In section \ref{space_smoother_emb},
we carry out the construction of the space $\ccE^s_\mu(\Om)$ as outlined in the introduction.
In section \ref{setting}, we derive a new set of equations that splits the dynamics
into an equation for a function $f$ that controls the boundary motion, and
an equation for a diffeomorphism $\be$ that fixes the boundary setwise. $f$ is determined
by its boundary values. Thus in section \ref{section_geom_boundary}, we derive a
further equation for $\left. f \right|_{\partial \Om}$. Section \ref{section_geom_boundary}
consists solely of a series of calculations necessary to analyze $\left. f \right|_{\partial \Om}$ 
and some readers may want to skip it.
Section \ref{section_f_boundary} is the core of the paper, where the existence 
of $\left. f \right|_{\partial \Om}$ is established and estimates for $f$
in terms of $\frac{1}{\kappa}$ are obtained. Section \ref{section_existence}
establishes the existence of solutions to (\ref{free_boundary_full}) via an
iteration scheme. Section \ref{section_convergence} shows the convergence
$\eta_\kappa \rar \zeta$ in the limit $\kappa \rar\infty$.

\begin{remark}
Throughout sections \ref{space_smoother_emb} to \ref{solution_large_kappa}
we work under the hypotheses of part (2) of theorem \ref{main_theorem}, i.e., 
we assume that $\kappa$ is large, $\partial \Om$ has constant mean curvature, 
and $\p Q u_0 \p_s \leq \frac{C}{\sqrt{\kappa}}$. In section \ref{existence_proof}
we show how to prove the existence result under the general assumptions of part (1)
of theorem \ref{main_theorem}.
\label{remark_large_kappa}
\end{remark}

\subsection{Scaling by length\label{scaling}}

When we speak of large surface tension or large $\kappa,$ we should take into account the size of the domain $\Omega.$ It seems clear that surface tension should have more of an effect in a small domain than in a large one.  To clarify this we examine the effect of scaling the length of the domain.

Let $\lambda$ be a positive scale factor and assume $\eta(t)$ is some motion satisfying (\ref{free_boundary_full}).  Then on the scaled domain $\lambda \Omega$ define $\zeta(t)$ by $\zeta(t)(\lambda x) = \lambda \eta(t)(x).$  Then letting $y = \lambda x$ we find that
$$\ddot{\zeta}(t)(y) = \lambda \ddot{\eta}(t)(x).$$ A routine computation shows that $\zeta$ satisfies (\ref{free_boundary_full}) on $\lambda \Omega$ with $p$ replaced by $q$ where $q$ is defined by $q(y) = \lambda^2 p(x).$ However the mean curvature of $\partial \zeta (\lambda \Omega) = \partial \lambda \eta(\Omega)$ is $(1/\lambda) \cA,$ where $\cA$ is the mean curvature of $\partial \eta (\Omega).$
Thus $q = \lambda^2 p = \lambda^2 \kappa \cA = \lambda^3 \kappa (1/\lambda) \cA$ so the scaled motion has an effective coefficient of surface tension of $\lambda^3 \kappa.$ Hence, when we study the effect of surface tension we really should consider $\kappa$ divided by a typical length cubed or $\kappa$ divided by the volume of the domain.  For a given $\kappa,$ the surface tension will have a much greater effect on a drop of liquid than it will on a large body.  
 
\subsection{A geometric interpretation of theorem \ref{main_theorem}\label{geometric}}

Theorem \ref{main_theorem} not only gives a satisfactory answer to 
the natural 
question of the dependence of solutions on the parameter $\kappa$; it also 
addresses a well motivated problem in Applied Science, namely,
when one can, by considering a sufficiently high surface tension,
neglect the motion of the boundary in favor of the simpler description 
in terms of the equations within a fixed domain.

The physical intuition behind theorem \ref{main_theorem} is very simple, as we now explain.
The system (\ref{free_boundary_full}) can be derived from an action principle with Lagrangian
\begin{align}
 \cL(\eta) = K(\eta) - V(\eta) ,
\label{Lagrangian}
\end{align}
where
\begin{align}
 K(\eta, \dot{\eta}) = \frac{1}{2} \int_{\Om} |\dot{\eta}|^2
\label{kinetic}
\end{align}
is the kinetic energy and 
\begin{align}
 V(\eta) =  \kappa |\partial \Om(t) | - \kappa|\partial \Om(0)| = \kappa \Big (  \operatorname{Area}(\partial \Om(t))
 -  \operatorname{Area}(\partial \Om(0)) \Big )
\label{potential}
\end{align}
is the potential energy\footnote{Many authors consider instead $V(\eta) = \kappa |\partial \Om(t) |$.
As the equations of motion remain unchanged by adding a constant, 
we choose to normalize the potential energy to make $V=0$ at time zero.  Such a normalization
is convenient for our purposes as we are interested in taking $\kappa \rar \infty$, in which 
case, if we did not subtract the contribution at time zero, $V(\eta)$ would diverge to infinity.}, and 
with $\eta : [0,T) \rar \cE_\mu^s(\Om)$.
The energy for the fluid motion (\ref{free_boundary_full}) is given by the sum
of the kinetic and potential energies (\ref{kinetic}) and (\ref{potential}), respectively,
\begin{align}
\begin{split}
 E(t) & = K(\eta,\dot{\eta}) + V(\eta) \\
& = \frac{1}{2}\int_\Om |\dot{\eta}|^2 + k \Big ( |\partial \Om(t)| - |\partial \Om| \Big )  ,
\label{energy}
\end{split}
\end{align}
This energy is conserved, and therefore 
\begin{gather}
E(t)  = \frac{1}{2} \p u_0 \p^2_0 
\label{energy_conserved} 
\end{gather}
where we have used $\dot{\eta} = u \circ \eta$ and $\eta(0) = \id$.

Our theorem \ref{main_theorem} is almost an example of a general theorem on motion with a strong constraining force \cite{E2}.
For the general theorem we are given a Riemannian manifold $M$ and a submanifold $N$.  Also given is a function $V: M \rightarrow \RR$
which has $N$ as a strict local minimum in the sense that $\nabla V = 0$ on $N$ and $D^2 V$ is a positive definite bilinear form on the normal bundle of $N$ in $M$. Then if $\eta_\kappa(t)$ is a motion given by the Lagrangian $\cL(\eta, \dot{\eta}) = \frac{1}{2} \langle \dot{\eta},\dot{\eta} \rangle -\kappa V(\eta)$ where $\langle  \, , \, \rangle$ is the Riemannian metric,
 and if $\zeta(t)$ is a Lagrangian motion in $N$ of $\frac{1}{2} \langle \dot{\zeta},\dot{\zeta} \rangle$ with the same initial conditions as $\eta_\kappa(t)$, the theorem says that $\eta_\kappa(t)$ converges to $\zeta(t)$ as $\kappa \rightarrow \infty$.  Also $\dot{\eta}_k \rightarrow \dot{\zeta}$, but the second derivative in general does not converge.  For our theorem, $M=\cE_\mu^{s}(\Om) ,$ $N=\cD_\mu^{s}(\Om) ,$ $\langle \;,\;\rangle$ is the $L^2$ inner product on tangent vectors and $V(\eta)$ is given by (\ref{potential}).

Our theorem \ref{main_theorem} is not actually an example of the general theorem for two reasons:

a) The $L^2$ inner product on tangent spaces is only a 
weak Riemannian metric\footnote{We recall 
 that a weak Riemannian metric is one which induces, on each tangent space, a weaker topology
 than the one given by the local charts. This is a feature exclusive to infinite dimensional 
 manifolds; see \cite{E_manifold,E1,L} for details.}.  The topology that it induces is weaker than the  $H^{s}$ topology of
$\cE_\mu^{s}(\Om).$

b) The bi-linear form $D^2V$ is only weakly positive definite on each normal space; it gives a  topology weaker than the $H^{s}$ topology.

Thus, while theorem \ref{main_theorem} is not a particular case of 
established results about the behavior of the
Euler-Lagrange equations near a submanifold which minimizes the 
potential energy \cite{E2},
it can be viewed to be in the spirit of
those results. Here, as  in \cite{E2, E4, ED},
the manifold minimizing the potential
energy is $\cD^{s}_\mu(\Om)$; see also \cite{Dis_linear}.

\section{Auxiliary results\label{auxiliary}}

Here we recall some well known facts which will be used throughout 
the paper. For their proof, see e.g. \cite{Adams, BB, Hitch, E1, P}.

\begin{proposition} Let $s > \frac{n}{2} + 2$, $g \in \cE^s_\mu(\Om)$,
$f \in H^s(g(\Om))$. Then $f \circ g \in H^s(\Om)$ and 
\begin{gather}
 \p f \circ g \p_s \leq C \p f \p_s\left( 1 + \p g \p_s^s \right),
 \label{Sobolev_composition}
\end{gather}
where $C =C(n,s,\Om)$.
\end{proposition}

We shall make use of the following well-known bilinear inequality
\begin{gather}
 \p u \, v \p_r \leq \, C \p u \p_r \p v \p_s,
\label{bilinear}
\end{gather}
for $s > \frac{n}{2}$, $s \geq r \geq 0$, where $C =C(n,s,r,\Om)$. 

For future reference, we remark that (\ref{bilinear}) still
holds true in negative norm Sobolev spaces (in a compact domain without boundary). Indeed, 
if $a \in H^{-r_1}$, $r_1 \geq 0$, $b \in H^{r_2}$, $r_2 > \frac{n}{2}$,
then
\begin{align}
\begin{split}
\p a b \p_{-r_1} & = \sup_{\omega \in H^{r_1}}
 \frac{|(ab,\omega)_0|}{\p \omega \p_{r_1}}
= 
 \sup_{\omega \in H^{r_1}} \frac{|(a,b\omega)_0|}{\p \omega \p_{r_1}}
 \\
& \leq  
\sup_{\omega \in H^{r_1}} \frac{\p a \p_{-r_1}
 \p b\omega \p_{r_1}}{\p \omega \p_{r_1}}
 \leq  C \p a \p_{-r_1} \p b \p_{r_2},
\end{split}
\nonumber
\end{align}
after using (\ref{bilinear}) in the last step to estimate 
$\p b\omega \p_{r_1} \leq C \p \omega \p_{r_1} \p b \p_{r_2}$. 

\begin{notation}
Although the dimension $n=3$ is fixed throughout, we sometimes write $n$ instead of
$3$ in order to make it easier to read off  conditions such as $s > \frac{n}{2}$  that 
are necessary for the application of (\ref{bilinear}) and other dimension dependent results.
\end{notation}

Recall also that restriction to the boundary gives rise to 
a bounded linear map,
\begin{gather}
 \p u \p_{s, \partial} \leq C \p u \p_{s + \frac{1}{2}},~~ s > 0,
\label{restriction}
\end{gather}
with $C =C(n,s,\Om)$. Estimates (\ref{bilinear}) and (\ref{restriction})
will be used throughout the paper, so we shall not explicitly refer to 
them in every instance.

The following $\dive-\curl$ estimate is well-known (see, e.g., 
\cite{TaylorPDE1}): 
let $\Om$ be a domain with an $H^r$ boundary, $r \geq 3$,
$v$ be a vector field on $\Om$ such that $v \in H^0(\Om)$, 
$\curl v \in H^{s-1}(\Om)$, $\dive v \in H^{s-1}(\Om)$, and 
$\langle v, \nu \rangle \in H^{s-\frac{1}{2}}(\partial \Om)$, where $\nu$ is the unit vector normal to $\partial \Om.$ Then,
$v \in H^s(\Om)$, and we have the following estimate
\begin{gather}
\p v  \p_s \leq C( \p v \p_0 + 
\p \curl v \p_{s-1} + \p \dive v \p_{s-1} + \p \langle v, \nu \rangle \p_{s-\frac{1}{2},\partial} ).
\label{div-curl-estimate}
\end{gather}
 
Next we recall the decomposition of a vector field into its gradient and divergence free part. Given 
an $H^s$ vector field $\omega$ on $\Om$, define the operator $Q: H^s(\Om, \RR^n) \rar \nabla H^{s+1}(\Om, \RR^n)$
by $Q(\omega) = \nabla g$, where $g$ solves
\begin{gather}
\begin{cases}
 \Delta g = \dive(\omega) & \text{ in }\Om, \\
\frac{ \partial g}{\partial \nu} = \langle \omega, \nu \rangle \text{ on } \partial \Om.
\end{cases}
\end{gather}
Since solutions to the Neumann problem are unique up to additive constants, $\nabla g$ is uniquely determined by $\om$, so $Q$ is well defined.
Define $P: H^s(\Om,\RR^n) \rar \dive^{-1}(0)_\nu$, where $\dive^{-1}(0)_\nu$ denotes divergence free vector fields 
tangent to $\partial \Om$, by $P = I - Q$, where $I$ is the identity map. 
Then $Q$ and $P$ are orthogonal projections in $L^2$.

We shall make use of the following: 
\begin{notation}
If $\eta:\Om \rar \RR^n$ is a sufficiently regular embedding and 
 $\cO$ is a pseudo-differential operator defined on functions on $\eta(\Om)$, we let
 $\cO_\eta$, which acts on functions defined on $\Om$, be given by
 \begin{gather}
 \cO_\eta(h) = ( \cO ( h \circ \eta^{-1} ) ) \circ \eta.
 \nonumber
 \end{gather}
\label{notation_sub}
\end{notation}
We remark that using notation \ref{notation_sub}, equations (\ref{free_boundary_full})
can be written as
\begin{subnumcases}{\label{free_boundary_full_eta}}
 \ddot{\eta}  = - \nabla p \circ \eta   & in $ \Om$, \label{basic_fluid_motion_full_eta} \\
  \operatorname{div}_\eta (\dot{\eta}) = 0  & in  $\Om$, \label{equation_p_full_eta} \\
 \left. q \right|_{\partial \Om} = \kappa \cB   & on  $\partial \Om$, \label{bry_p_full_eta} \\
 \eta(0) = \id,~\dot{\eta}(0)=u_0,
\end{subnumcases}
where $q = p \circ \eta$ and $\cB = \cA \circ \eta$.
Equations (\ref{free_boundary_full_eta}) reveal yet another advantage of Lagrangian 
coordinates, as all equations are now written in terms of the fixed domain $\Om$.

Finally, theorems related to one-parameter groups of operators and
abstract differential equations will be needed. 
We also state them here for the reader's convenience.

\begin{theorem}
Let $X$ be a Banach space and let $Z$ be a densely defined closed
operator on $X$. Assume that every real $\la$ is in the resolvent  
of $Z$ and that 
\begin{gather}
\p (Z + \la)^{-1} \p \leq \frac{c_1}{|\la|},
\nonumber
\end{gather}
for some constant $c_1>0$. Then, $Z$ generates a $C^0$ semi-group 
of transformations $e^{tZ}:X \rar X$,
such that $\p e^{tZ} \p \leq c_1$. 
If  $Z^\prime$ is a bounded operator with norm $\p Z^\prime \p \leq c_2$,
then $Z + Z^\prime$ also generates a $C^0$ semi-group and
\begin{gather}
\p e^{tZ} \p \leq c_1 e^{c_2|t|},
\nonumber
\end{gather}
\label{theorem_generator_semi_group}
\end{theorem}

\begin{theorem}
Let $X$ and $Y$ be Hilbert spaces such that $Y$ is densely and continuously
embedded in $X$. Let $\{Z(t) \, | \, 0 \leq t \leq T \}$ be a family 
of operators on $X$, each of which generates a $C^0$-semi-group and
assume that:

(i) There exist  constants $\alpha$ and $\beta$ such that, for each $t$ and for all 
positive $\tau$,
\begin{gather}
\p e^{\tau Z(t)} \p \leq \alpha e^{\beta \tau},
\nonumber
\end{gather}

(ii) For each $t$, $e^{\tau Z(t)}$ restricted to $Y$ is a $C^0$ semi-group on $Y$. 
There exist constants $\ga$ and $\de$ such that, for each $t$, there exists an inner
product on $Y$, whose norm $\prescript{}{t}{\p \cdot \p}$ 
gives the topology of $Y$, and
such that 
\begin{gather}
\p e^{\tau Z(t)} \p_{\operatorname{Op}(t)} \leq  \de e^{\ga \tau},
\nonumber
\end{gather}
where $\operatorname{Op}(t)$ is the operator norm on $B(Y)$ induced by
$\prescript{}{t}{\p \cdot \p}$. Furthermore, there exist constants
$\mu$ and $\nu$ such that, for any $t_1,t_2 \in [0,T]$ and any $y \in Y$,
\begin{gather}
\prescript{}{t_2}{\p y \p} \leq \mu e^{\nu |t_2-t_1|} 
\prescript{}{t_1}{\p y \p}.
\nonumber
\end{gather}

(iii) $Y$ is included in the domains of each $Z(t)$, and $Z(t)$ is 
continuous as a map from $[0,T]$ to $B(Y,X)$.

(iv) $Z(t)$ is reversible in the sense that $\widehat{Z}(t) = 
-Z(T-t)$ also satisfies (i), (ii), and (iii).

Then, there exists a unique family of operators $U(t,\tau) \in B(X)$, defined
for $t, \tau \in [0,T]$, such that 

(a) $U(t,\tau)$ is strongly continuous as a function of $\tau$ and $t$,
$U(\tau,\tau) = I$ (the identity operator), and
\begin{gather}
\p U(t,\tau) \p \leq c_1 e^{c_2|t - \tau|},
\nonumber
\end{gather}
for some constants $c_1,c_2$ depending only on $\alpha, \beta, \gamma,
\de, \mu,$ and $\nu$.

(b) $U(t,\tau) = U(t,\tau^\prime) U(\tau^\prime, \tau)$.

(c) For all $y \in Y$,
\begin{gather}
\frac{\partial}{\partial t} \left( U(t,\tau) y \right) 
= Z(t) U(t,\tau) y,
\nonumber
\end{gather}
where $\frac{\partial}{\partial t}$ means right derivative at 
$t=0$ and left derivative at $t=T$.

(d) For all $y \in Y$,
\begin{gather}
\frac{\partial}{\partial \tau} \left( U(t,\tau) y \right) 
= -  U(t,\tau) Z(\tau) y,
\nonumber
\end{gather}
where $\frac{\partial}{\partial \tau }$ means right derivative at 
$\tau=0$ and left derivative at $\tau=T$.

(e) $U(t,\tau)Y \subseteq Y$ and for any $t,\tau \in [0,T]$,
\begin{gather}
\p U(t,\tau) \p_{\operatorname{Op}(t)} \leq 
c_3 e^{c_4 T} e^{c_5 |t - \tau|},
\nonumber
\end{gather}
for some constants $c_3, c_4, c_5$ depending only on $\alpha, \beta, \gamma,
\de, \mu,$ and $\nu$.

(f) $U(t,\tau)$ is strongly continuous into $Y$, as a function of $t$ and
$\tau$, and therefore 
\begin{gather}
\frac{\partial}{\partial t}( U(t,\tau) y ) = Z(t) U(t,\tau) y
\nonumber
\end{gather}
is continuous in $X$ as a function of $t$ and $\tau$.
\label{theorem_evolution_operator}
\end{theorem}

\begin{theorem}
Let $X$, $Y$ and $Z(t)$ be as in theorem \ref{theorem_evolution_operator}.
Let $e(t)$ be a continuous curve in $Y$ and define $y(t)$ to be
\begin{gather}
y(t) = U(t,0) y_0 + \int_0^t U(t,\tau) e(\tau) \, d\tau,
\nonumber
\end{gather}
for $0 \leq t \leq T$, $y_0 \in Y$. $U$ is the evolution operator given
 by theorem \ref{theorem_evolution_operator}.
 Then $y(t)$ is a continuous curve in $Y$ which 
is $C^1$ in $X$, and it is the unique solution to the equation
\begin{gather}
\dot{y}(t) = Z(t)y(t) + e(t)
\nonumber
\end{gather}
such that $y(0) = y_0$.
\label{theorem_abstract_ODE}
\end{theorem}

The proof of theorem \ref{theorem_generator_semi_group} can be found,
for instance, in chapter 12 of \cite{Hille}. 
Theorems \ref{theorem_evolution_operator}
and \ref{theorem_abstract_ODE} are proven in \cite{K_1} (see also \cite{K}). We remark that the results
of \cite{K_1} are much more general than the above. Here, we stated
them in a form suitable for our purposes. See also \cite{E2}.

\section{The space $\ccE_\mu^s(\Om)$\label{space_smoother_emb}}
Here we construct the space $\ccE_\mu^s(\Om)$, as outlined in 
the introduction. We assume that $s > \frac{n}{2} + 2$.

To start we note that the equation
\begin{align}
 J(\id + \nabla f) = 1
 \nonumber
\end{align}
can be written as
\begin{align}
\Delta f + \cN(f) = 0,
\label{nl_Dirichlet_f}
\end{align}
where
\begin{gather}
\cN(f) = f_{xx}f_{yy} + f_{xx}f_{zz} + f_{yy}f_{zz} - f_{xy}^2 - f_{xz}^2 
- f_{yz}^2 + \det (D^2 f).
\label{non_linear_terms_f_extension}
\end{gather}
Equation (\ref{nl_Dirichlet_f}) can be considered as a non-linear Dirichlet problem for $f$, and so for $f$ small, $f$ should be determined by its boundary values. 
We shall present the argument  for three dimensions, which is the
main case of interest in this work. The interested reader can generalize the
 construction 
of $\ccE_\mu^s(\Om)$ to higher dimensions.

Given $h \in H^{s+2}(\partial \Om)$, we 
are interested in solving 
\begin{subnumcases}{\label{nl_Dirichlet_f_3d_system}}
\Delta f + \cN(f)  = 0 & in $\Om$,
 \label{nl_Dirichlet_f_3d} \\
f = h & on $\partial \Om$. \label{nl_Dirichlet_bry}
\end{subnumcases}
Define a map 
\begin{align}
& F: H^{s+2}(\partial \Om) \times H^{s+\frac{5}{2}}(\Om) 
\rar H^{s+2}(\partial \Om) \times H^{s + \frac{1}{2}}(\Om) ,
\nonumber \\
{\rm by} \:\:\:\:\:\: & F(h, f) = ( f\left|_{\partial \Om}\right. - h, \Delta f +\cN(f) ).
\nonumber
\end{align}

Notice that $F$ is $C^1$ in the neighborhood of the origin and $F(0,0) = 0$, where we denote by $0$ 
the origin in the product Hilbert space $H^{s+2}(\partial \Om) \times H^{s+\frac{5}{2}}(\Om) $. Letting $w \in  H^{s+2}(\Om)$, we obtain
\begin{gather}
D_2 F(0,0)(w) = (  w\left|_{\partial \Om}\right., \Delta w ),
\label{linearization}
\end{gather}
where $D_2$ is the partial derivative of $F$ with respect to 
its second argument. From the uniqueness of solutions to the 
 Dirichlet problem 
it follows that $D_2 F(0,0)$ is an isomorphism, and therefore 
by the implicit function theorem there exists a neighborhood of zero
in $H^{s+2}(\partial \Om)$, (which we can take without 
loss of generality to be a ball
 $\cB_{\de_0}^{s+2}(\partial \Om)$) and 
a $C^1$
map $\varphi: \cB_{\de_0}^{s+2}(\partial \Om) \rar H^{s+\frac{5}{2}}(\Om)$
satisfying $\varphi(0) = 0$, and 
such 
that $F(h, \varphi(h)) = 0$ for all $h \in \cB_{\de_0}^{s+2}(\Om)$.
In other words, $f = \varphi(h)$ solves (\ref{nl_Dirichlet_f_3d_system}).

Furthermore, $D \varphi = -(D_2 F)^{-1} D_1 F$.
Thus $D \varphi$ is injective at the origin 
(in fact, it is not difficult to see that the derivative
$D \varphi(0)$ is the harmonic extension map), 
and  so $\varphi$ is injective near zero.
From this and 
the above it then follows that $\varphi( \cB^{s+2}_{\de_0}(\partial \Om))$
is a submanifold of $H^{s+\frac{5}{2}}(\Om)$.

Recall now the definition (\ref{big_phi}). Notice that
$\Phi$ is well defined (if $\de_0$ is small) and its image
belongs to $\cE_\mu^s(\Om)$ since $J(\beta) = 1$ and, by 
construction,  $J(\id + \nabla f) = 1$.

We have therefore proven:

\begin{proposition}
Let $s > \frac{n}{2} + 2$ and let $B^{s+2}_{\de_0}(\partial \Om)$ be 
the open ball of radius $\de_0$ in $H^{s+2}(\partial \Om)$. Then,
if $\de_0$ is sufficiently small, there exists an
embedding $\varphi: B^{s+2}_{\de_0}(\partial \Om) \rar H^{s+\frac{5}{2}}(\Om)$, 
given explicitly by $\varphi(h) = f$, where $f$ solves (\ref{nl_Dirichlet_f_3d_system}). Moreover, the map $\Phi$ given by (\ref{big_phi}) is well defined.
\label{embedding}
\end{proposition}

\begin{definition}
Under the hypotheses of proposition \ref{embedding}, we define
$\mathscr{E}_\mu^s(\Om) \subseteq \cE_\mu^s(\Om)$ by
\begin{gather}
\mathscr{E}_\mu^s(\Om)
 = \Phi\big( \cD_\mu^s (\Om) \times \varphi(\cB_{\de_0}^{s+2}(\partial \Om))  \big).
\nonumber
\end{gather}
\end{definition}

\section{A new system of equations \label{setting}}

In this section, we shall derive a different set of equations for 
the free boundary problem (\ref{free_boundary_full}). 
In what follows, we shall make use of the well known decomposition of  a vector
field into its gradient and divergence free parts, as presented 
in section \ref{auxiliary}.
Hence, recall that $Q: H^s(\Om, \RR^n) \rar \nabla H^{s+1}(\Om, \RR^n)$ and 
 $P: H^s(\Om,\RR^n) \rar \dive^{-1}(0)_\nu$
 (where $\dive^{-1}(0)_\nu$ denotes divergence free vector fields 
tangent to $\partial \Om$) are the operators realizing this decomposition. They
satisfy $P + Q = I$, where $I$ is the identity map, and 
since  $\nabla H^{s+1}(\Om, \RR^n)$ and $\dive^{-1}(0)_\nu$
are $L^2$-orthogonal, it follows that $Q$ and $P$ are orthogonal projections in $L^2$.

To derive the new system, assume 
that 
solutions $\eta$ to (\ref{free_boundary_full}) can be written as
\begin{align}
\eta = (\id + \nabla f) \circ \beta,
\label{def_f_beta} 
\end{align}
with $\beta \in \cD_\mu^s(\Om)$, $\nabla f \in H^s(\Om)$
and with $f$ satisfying (\ref{nl_Dirichlet_f_3d}).
In this case we also observe that 
\begin{align}
 \beta(0) = \id,~~\nabla f(0) = 0.
 \nonumber
\end{align}

It is customary to write the pressure as a sum of an interior and a boundary term, namely,
\begin{gather}
p = p_0 + \kappa \cA_H,
\label{split_pressure}
\end{gather}
so that the system (\ref{free_boundary_full}) takes the form
 \begin{subnumcases}{\label{free_boundary} }
  \ddot{\eta}  = - \nabla p \circ \eta = - (\nabla p_0 + \kappa \nabla \cA_H)\circ \eta & in $\Om$,
\label{basic_fluid_motion} \\
\dive(\dot{\eta}\circ \eta^{-1} ) = 0 &  in $\eta(\Om)$, \\
  \Delta p_0 = - \dive (\nabla_u u) & in $\eta(\Om)$ , 
  \label{equation_p_0} \\
\left. p_0 \right|_{\partial \eta(\Om)} = 0 & on $ \partial \eta(\Om)$,
\label{bry_p_0} \\
  \Delta \cA_H = 0 &  in  $\eta(\Om)$,
  \label{harmonic_ext_def} \\
\left. \cA_H \right|_{\partial \eta(\Om)} =  \cA &   on $\partial \eta(\Om)$,
 \label{equation_harm_ext} \\
 \eta(0) = \id,~\dot{\eta}(0)=u_0.
\end{subnumcases}

Differentiating (\ref{def_f_beta}) in time gives
\begin{align}
 \dot{\eta} = (\nabla \dot{f} + v \cdot D \,\nabla f + v) \circ \beta,
\label{dot_eta_f_beta}
\end{align}
where $v$ is defined by 
\begin{align}
 \dot{\beta} = v \circ \beta.
\label{definition_v}
\end{align}
Using $\eta(0) = \id$ and $\dot{\eta}(0) = u_0$, 
from (\ref{dot_eta_f_beta}) we obtain
\begin{align}
 u_0 = \nabla \dot{f}(0) + v_0.
\nonumber
\end{align}
We  therefore obtain
\begin{gather}
 v_0 = P u_0,
\nonumber
\end{gather}
and
\begin{gather}
\nabla \dot{f}(0) = Q u_0.
\nonumber
\end{gather}
Differentiating (\ref{dot_eta_f_beta}) again and using (\ref{basic_fluid_motion}) gives the following equation:
\begin{align}
 \nabla \ddot{f} + 2 D_v \nabla \dot{f} + D^2_{vv} \nabla f + (\dot{v} + v \cdot \nabla v) D \, \nabla f
+ \dot{v} + v \cdot \nabla v = -\nabla p\circ (\id + \nabla f),
\label{equation_ddot_nabla_f}
\end{align}
where the operator $D^2_{vv}$, acting on a vector $w$, is given in coordinates by 
\begin{gather}
( D^2_{vv} w )^i = v^j v^l \partial_{j}\partial_{l} w^i,
\label{D_vv_two}
\end{gather}
or in invariant form by
\begin{gather}
D^2_{v v} w = D_{v} \nabla_{v} w - D_{\nabla_{v} v} w.
\nonumber
\end{gather}
Define $L$ on the space of maps from $\Om$ to $\RR^n$ by 
\begin{gather}
L = \id +D^2f,
\label{L_def}
\end{gather}
and let 
\begin{gather}
L_1 = P L,
\label{L_1_def}
\end{gather}
and
\begin{gather}
L_2 = QL,
\label{L_2_def}
\end{gather}where 
$P$ and $Q$ are as in section \ref{auxiliary}.  Notice that $L_1$ is invertible on the image of $P$ if $f$
is small, since in this case it will be close to the identity. 
Then we can write an arbitrary vector field $X$ as
\begin{gather}
X = L L_1^{-1} P (X) + ( Q - L_2 L_1^{-1}P ) (X),
\label{decomp}
\end{gather}
and thus effect the decomposition
\begin{gather}
H^s(\Om, \RR^n) = LP(H^s(\Om, \RR^n)) \oplus Q(H^s(\Om,\RR^n)).
\nonumber
\end{gather}
Decomposing (\ref{equation_ddot_nabla_f}) in this fashion,
and recalling that $f$ also has to satisfy (\ref{nl_Dirichlet_f}),
we obtain
\begin{subnumcases}{\label{system_f_v_full}}
\nabla \ddot{f}  +   ( Q - L_2 L_1^{-1}P ) 
( 2 D_v \nabla \dot{f}  + D^2_{vv} \nabla f + LQ(\nabla_v v) ) \nonumber \\
\hspace{2.8cm}
= -(Q - L_2 L_1^{-1}P) (\nabla p\circ (\id + \nabla f) ) & in $\Om$,
\label{system_f_v_f_dot_dot} \\
\Delta f + \cN(f) \hspace{0.75cm} = 0 & in $\Om$, 
\label{system_f_nl_Dirichlet}  \\
 \dot{v} +  P(\nabla_v v) + L_1^{-1}P (
  2 D_v \nabla\dot{f} + D^2_{vv} \nabla f )
  + L_1^{-1} P( LQ( \nabla_v v) ) \nonumber \\
  \hspace{2.8cm} = -L_1^{-1} P  (\nabla p\circ (\id + \nabla f) )   & in $\Om$, 
  \label{system_f_v_v_dot} \\
\nabla f (0) = 0, \, \, \nabla \dot{f} (0) = Q u_0, \,\,
 v(0) = P u_0, \label{initial conditions}
\end{subnumcases}
where $\cN$ is given by (\ref{non_linear_terms_f_extension}).
Equation (\ref{system_f_v_v_dot}) implies   that $v$ also satisfies $\dive(v) = 0$ and
$\langle v , \nu \rangle = 0$, where $\nu$ is the unit outer normal to
$\partial \Om$, as $v$ is in the image of $P$.
The system
(\ref{system_f_v_full}) contains two equations for $f$, namely,
(\ref{system_f_v_f_dot_dot}) and (\ref{system_f_nl_Dirichlet}), but no boundary
condition. We shall transform these into a more standard evolution
equation for $f$ by restricting (\ref{system_f_v_f_dot_dot}) to $\partial \Om$,
obtaining an evolution equation for $\left. f \right|_{\partial \Om}$, 
with an extension 
to $\Om$ given via (\ref{system_f_nl_Dirichlet}).

It is illustrative to point out that (\ref{system_f_v_v_dot}) formally reduces
to the Euler equations in the fixed domain when $\nabla f \equiv 0$ and thus, 
formally, $\zeta = \beta$. This is in agreement
with the intuition discussed in the introduction that if the boundary displacements
controlled by $\nabla f$ approach zero when $\kappa \rar \infty$, the solution
$\eta$ should approach $\zeta$. Such an intuitive appeal notwithstanding, 
equation (\ref{system_f_v_v_dot}) will not be used directly in our proof.
There are two reasons for this.
First, as explained in the introduction, there is no reason
to suspect that  the pressure will converge when
we take the limit $\kappa \rar \infty$, hence no good control of the right-hand
side of (\ref{system_f_v_v_dot}) can be expected. Second, 
in view of the regularity of $p$  stated in theorem (\ref{main_theorem}), 
the right hand side of (\ref{system_f_v_v_dot})
will be in $H^{s-\frac{3}{2}}$. Known techniques, therefore, can only yield
$v$, and hence $\beta$, in $H^{s-\frac{3}{2}}$. However, $\beta \in H^s(\Om)$ 
is required for $\eta \in H^s(\Om)$, see (\ref{def_f_beta}). Our proof
overcomes these difficulties by making the most of the Lagrangian description of the 
fluid, which is consistent with the general idea that Lagrangian coordinates
are ``better behaved" than Eulerian ones (see similar discussion in \cite{E2} and
\cite{ED}).

\section{Geometry of the boundary and analysis of $\nabla f$ \label{section_geom_boundary} }

In light of proposition \ref{embedding}, $f_\kappa$ is determined
by its boundary values, provided it is small. In this section, we shall
show that  $\left. f_\kappa \right|_{\partial \Om}$ obeys an equation of the form
\begin{align}
\begin{cases}
\ddot{f}_\kappa = \ccA_\kappa(\beta_\kappa, v_\kappa, p_\kappa, f_\kappa) + \ccB_\kappa(\beta_\kappa, v_\kappa, p_\kappa, \dot{f}_\kappa ) + \ccC_\kappa(\beta_\kappa, v_\kappa, p_\kappa) \,\,\text{ on } \, \,
\partial \Om \\
 f_\kappa(0) = 0, \, \dot{f}_\kappa(0) = f_1,
 \end{cases} 
\nonumber
\end{align}
where $\ccA_\kappa$ is a third order pseudo-differential operator on 
$f_\kappa$, $\ccB_\kappa$ is
first order, $\ccC_\kappa$ is a zeroth order operator on $v_\kappa$,
and $f_0$ and $f_1$ are known functions.
The desired equation will be equation (\ref{eq_f_bry}) below, and the present section 
is  a derivation of (\ref{eq_f_bry}) from (\ref{system_f_v_f_dot_dot}).
This amounts essentially to a series of lengthy calculations, and some readers
may want to skip them and move directly to section \ref{section_f_boundary}.

From now on, the subscript 
$\kappa$ will be omitted in all quantities.
Throughout this sections we assume we are given a sufficiently regular
solution of (\ref{system_f_v_full}). Moreover, as 
many of the derivations below are valid provided that 
$f$ is sufficiently small, we shall assume so throughout this 
section. This smallness condition will be made precise in section 
\ref{section_f_boundary}. Several standard geometric constructions will be employed
below. They can be found, for example, in \cite{Han,Spivak}.

\begin{notation} Let 
\begin{gather}
\widetilde{\eta} = \eta \circ \beta^{-1} \equiv \id + \nabla f .
\nonumber 
\end{gather}
Unless stated otherwise, from now on quantities with $\widetilde{~}$
are defined on the domain $\widetilde{\eta}(\Om)$. For example, 
if $N$ denotes the normal to $\partial \eta(\Om)$, then
$\widetilde{N}$ is the normal to $\partial \widetilde{\eta}(\Om)$.
\end{notation}

\subsection{A rewritten equation for $\nabla f$}
From (\ref{L_def}), (\ref{L_1_def}), and (\ref{L_2_def}), we see that
\begin{align}
\begin{split}
Q - L_2 L_1^{-1}P 
 & = Q \left( \id - (\id + D^2 f) L_1^{-1} P \right) \\
 & = Q - Q D^2 f L^{-1} P,
 \end{split}
\nonumber
\end{align}
where we used the fact that $Q$ vanishes on the image of $P$. 

Let
\begin{gather}
\widetilde{F} = \cA_H \circ \widetilde{\eta}.
\label{F_tilde_def}
\end{gather}
and 
\begin{gather}
\widetilde{q}_0 = p_0 \circ \widetilde{\eta}.
\label{q_0_tilde_def}
\end{gather}
Notice that $\widetilde{F}$ depends on $f$. A calculation gives
\begin{gather}
\nabla \cA_H \circ \widetilde{\eta} = \nabla \widetilde{F} ( D \widetilde{\eta} )^{-1},
\label{nabla_A_eta_tilde}
\end{gather}
and 
\begin{gather}
\nabla  p_0 \circ \widetilde{\eta} = \nabla \widetilde{q}_0 ( D \widetilde{\eta} )^{-1},
\label{nabla_p_0_eta_tilde}
\end{gather}
Then (\ref{F_tilde_def}), (\ref{q_0_tilde_def}), (\ref{nabla_A_eta_tilde}), (\ref{nabla_p_0_eta_tilde}),
and (\ref{split_pressure}) give
\begin{gather}
\nabla p \circ \widetilde{\eta} = \nabla \widetilde{q}_0
+  
\nabla \widetilde{q}_0\left( ( D \widetilde{\eta} )^{-1} - \id  \right) +
\kappa \nabla \widetilde{F} + \kappa \nabla  \widetilde{F}(-\id + 
(D\widetilde{\eta})^{-1} ).
\nonumber
\end{gather}

Using the above, (\ref{system_f_v_f_dot_dot}) can be written as
\begin{align}
\begin{split}
\nabla \ddot{f} & + \kappa \nabla \widetilde{F} +
\kappa \nabla \Delta_\nu^{-1} \dive \left(\nabla  \widetilde{F}(-\id 
+ (D\widetilde{\eta})^{-1} ) \right)
- \kappa \nabla \Delta_\nu^{-1} \dive\left( D^2 f L_1^{-1} P (\nabla \widetilde{F}
(D\widetilde{\eta})^{-1} ) \right) \\
& + 2\nabla \Delta_\nu^{-1} \dive\left( D_v \nabla \dot{f} \right)
- 2 \nabla \Delta_\nu^{-1} \dive\left( D^2 f L_1^{-1} P D_v \nabla \dot{f} \right) \\
& + \nabla \Delta_\nu^{-1} \dive\left( D_{vv}^2 \nabla f\right)
- \nabla \Delta_\nu^{-1} \dive\left( D^2 f L_1^{-1} P D_{vv}^2 \nabla f \right) \\
& + \nabla \Delta_\nu^{-1} \dive\left( D^2 f Q(\nabla_v v) \right)
- \nabla \Delta_\nu^{-1} \dive\left( D^2 f L_1^{-1} P D^2f Q(\nabla_v v ) \right) \\
& 
+ \nabla  \Delta_\nu^{-1} \dive\left(  \nabla \widetilde{q}_0 \left( 
(D \widetilde{\eta})^{-1} - \id \right) \right) 
 - \nabla \Delta_\nu^{-1} \dive\left( 
D^2 f L_1^{-1} P \left( \nabla \widetilde{q}_0( (D\widetilde{\eta})^{-1} - \id ) \right)
\right) \\
= & - \nabla \widetilde{q}_0 - \nabla \Delta_\nu^{-1} \dive\left( \nabla_v v \right).
\end{split}
\label{new_form_f_eq}
\end{align}
In the above, the terms in $\nabla \Delta_\nu^{-1} \dive$ appear upon writing $Q$ explicitly.
 The operator $\Delta_\nu^{-1} \circ \dive$ is given by
\begin{gather}
 \Delta_\nu^{-1} \dive\left( w \right) = g,
\nonumber
\end{gather}
where $g$ solves
\begin{gather}
\begin{cases}
\Delta g = \dive(w), & \text{ in } \Om, \\
\frac{\partial g}{\partial \nu} = \langle w, \nu \rangle, & \text{ on } \partial \Om.
\end{cases}
\nonumber
\end{gather}
Notice that $\Delta_\nu^{-1} \circ \dive$ is defined up to an additive constant, so 
$\nabla \Delta_\nu^{-1} \circ \dive$ is defined uniquely. 
\begin{remark}
We notice for further reference, that in (\ref{new_form_f_eq})
the first term in every line, except for the last and the next-to-the last lines, 
is linear in $f$, with the remaining terms being non-linear (in $f$).
\label{remark_linear_non_linear_terms_f_eq}
\end{remark}

\subsection{Local coordinates\label{section_local_coord}}
In order to have a more explicit description of the operator $\widetilde{F}$ acting on $f$,
we employ local coordinates.

Working locally, we choose coordinates
$(x^1,x^2,x^3)$
near $\partial \Om$ such that the domain and its boundary are  given by
\begin{gather}
\Om = \{ x^ 3 > 0 \}, \hspace{0.5cm} 
\partial \Om = \{ x^3 = 0 \},
\nonumber
\end{gather}
so that 
\begin{gather}
\left. \frac{\partial}{\partial x^1}\right|_{x^3=0} 
\hspace{0.5cm} \text{ and } \hspace{0.5cm} 
\left. \frac{\partial}{\partial x^2 }\right|_{x^3=0}
\nonumber
\end{gather}
are tangent to $\partial \Om$. We write $x=(x^\prime,x^3)$.
In these coordinates, the Euclidean metric is represented by the matrix
\begin{gather}
g = (g_{\al\be}), \, \al, \be = 1, 2, 3,
\nonumber
\end{gather}
with the induced metric on $\partial \Om$ being simply
\begin{gather}
g_{ij}(x^\prime,0), \,  i,j = 1,2.
\nonumber
\end{gather}
Also we can assume that $g_{33} = 1$ and $g_{i3} = 0$, $i=1,2$.
\begin{notation}
Unless stated otherwise, Greek indices will run over $1,2,3$ and Latin indices
over $1,2$. $(\al\be)$ means symmetrization on $\al$, $\be$, i.e., $t_{(\al\be)}
= \frac{1}{2}( t_{\al\be} + t_{\be\al} )$. The summation
convention is assumed throughout.
\label{indices_convention}
\end{notation}
If $U$ is the coordinate chart, we always assume $x \in V \subset \subset U$
so that $\eta(x) \in U$ and $\widetilde{\eta}(x) \in U$; this is always
possible when $\widetilde{\eta}$ is near the identity, which will be the
case of interest below. Let
\begin{gather}
r = \left. \eta \right|_{\partial \Om},
\nonumber
\end{gather}
i.e.
\begin{gather}
r(x^\prime) = \eta(x^\prime, 0),
\nonumber
\end{gather}
and write $r = (r^1, r^2, r^3)$. Analogously we have $\widetilde{r}$.
With $\beta = (\be^1,\be^2,\be^3)$, it follows that
\begin{gather}
r^\al(x^\prime) = \be^\al(x^\prime,0) + \nabla f^\al \circ \beta(x^\prime,0),
\nonumber
\end{gather}
where
\begin{gather}
\nabla f^\al = g^{\al\mu}\partial_\mu f.
\nonumber
\end{gather}
Notice that since $\beta(\partial \Om) = \partial \Om$, 
\begin{gather}
\be^3(x^\prime, 0 ) = 0.
\nonumber
\end{gather}
A basis $\{X_1, X_2 \}$ for the tangent space of $r(\partial \Om) = \partial \eta(\Om)$ is 
given by 
\begin{gather}
X_i = \partial_i r = Dr \left(\left.  \frac{\partial}{\partial x^i}\right|_{x^3=0} \right),
\nonumber
\end{gather}
where $Dr$ is the derivative of $r$. Component-wise,
\begin{gather}
X^\al_i = g^{\al\mu} \partial_{\mu\nu } f \circ \beta \partial_i \, \beta^\nu 
+ \partial_i g^{\al\mu} \partial_\mu f \circ \beta + \de^\al_j \partial_i \be^j,
\nonumber
\end{gather}
where we have used $\be^3(x^\prime,0) = 0$ and $\de^\al_\be$ is the
Kronecker delta.

The unit (inward) normal to $r(\partial \Om)$
is 
\begin{gather}
N^\al = \frac{ \ve^\al_{\mss \be \ga} X_1^\be X_2^\ga}{ 
\sqrt{ 
\ve^\la_{\mss \mu \nu} \ve_{\la\si\tau} X_1^\mu X_1^\si X_2^\nu X_2^\tau 
}
},
\nonumber
\end{gather}
where $\ve_{\al\be\ga}$ is the totally anti-symmetric tensor (with the convention
$\ve_{123}=1$). The metric $\overline{g}$ induced on $r(\partial \Om)$ is
\begin{gather}
\overline{g}_{ij} = g(X_i, X_j) = g_{\al\be} X^\al_i X^\be_j.
\nonumber
\end{gather}
More explicitly,
\begin{align}
\begin{split}
\overline{g}_{ij}  & = 
g^{\al\be} \partial_{\be\nu} f \circ \be \partial_{\al\la} f \circ \be \partial_i \be^\nu 
\partial_j \be^ \la
+ 2 \partial_{k\nu} f \circ \be \partial_{(i} \be^\nu \partial_{j)} \be^k
\\
& + 2 \partial_{\al \nu} f \circ \be \partial_\mu f \circ \be \partial_{(i } g^{\al\mu} 
\partial_{j)} \be^\nu
+ 2 g_{\al k } \partial_\mu f \circ \be \partial_{( i }g^{\al\mu} \partial_{j)} \be^k
\\
& + g_{\al\be} \partial_i g^{\al\mu} \partial_j g^{\be\la} \partial_\mu f \circ \be 
\partial_\la f \circ \be 
+ g_{kl} \partial_i \be^k \partial_j \be^l,
\end{split}
\nonumber
\end{align}
where $(ij)$ means symmetrization in $i,j$ (see notation \ref{indices_convention}).
The second fundamental form of $\partial \eta(\Om)$ is defined by
the equivalent expressions
\begin{gather}
\cA_{ij} = - g(\nabla_i N, X_j) =  g(N,\nabla_i X_j ),
\nonumber
\end{gather}
where $\nabla$ is the Levi-Civita connection, and the negative 
sign on the first equality occurs because $N$ is the inner normal. Component-wise
\begin{align}
\begin{split}
\cA_{ij} & = g_{\al \be} N^\al \nabla_i X_j^\be \\
& = g_{\al\be} \partial_i X_j^\be N^\al + g_{\al\be} \Ga^\be_{i \mu} X^\mu_j N^\al,
\end{split}
\label{sff_def}
\end{align}
where $\Ga_{\al\be}^\ga$ are the Christoffel symbols.
Computing get
\begin{align}
\begin{split}
\partial_i X^\al_j & = g^{\al \mu} \partial_{\mu\la\nu} f \circ \be 
\partial_i \be^\nu \partial_j \be^\la + g^{\al\mu} \partial_{\mu \nu} f \circ \be 
\partial_{ij} \be^\nu + 2 \partial_{(i} \be^\nu \partial_{j)} g^{\al\mu} \partial_{\mu\nu} f \circ \be \\
& + \partial_{ij} g^{\al\mu} \partial_\mu f \circ \be + \de^\al_{\mss k} \partial_{ij} \be^k.
\end{split}
\nonumber
\end{align}
Also the mean curvature is defined by
\begin{gather}
\cA = \overline{g}^{ij} \cA_{ij}.
\label{mean_curvature_def}
\end{gather}

Note that the pressure splits into an interior term and a boundary term,
$p = p_0 + \kappa \cA_H$ (see (\ref{free_boundary})), 
and since (\ref{system_f_v_f_dot_dot}) involves 
$\nabla \cA_H \circ \widetilde{\eta}$ rather than
$\nabla \cA_H \circ \eta$, we shall need expressions for quantities
on the boundary $\widetilde{r}(\partial \Om) = \partial \widetilde{\eta}(\Om)$.
This amounts to setting $\be = \id$ in the above expressions, 
leading to
\begin{gather}
\widetilde{X}_i^\al = g^{\al\mu} \partial_{\mu i} f
+ \partial_i g^{\al\mu} \partial_\mu f + \de^\al_{\mss i},
\label{set_beta_id_X}
\end{gather}
\begin{gather}
\partial_i \widetilde{X}_j^\al = g^{\al \mu} \partial_{\mu ij} f + 
2 \partial_{(i} g^{\al\mu} \partial_{j) \mu} f + \partial_{ij} g^{\al\mu} \partial_\mu f,
\label{set_beta_id_partial_X}
\end{gather}
\begin{gather}
\widetilde{N}^\al = \frac{ g^{\al3} + T^\al(f) }{ 
\sqrt{ 
1 + 2 g^{\la 3} T_\la(f) + T_\la(f) T^\la(f) 
}
},
\label{set_beta_id_N}
\end{gather}
and
\begin{align}
\begin{split}
\widetilde{g}_{ij} & = g^{\mu\nu} \partial_{\mu i} f \partial_{\nu j } f 
+ 2 \partial_{ij} f + 2 \partial_{\mu(i} f \partial_{j)} g^{\mu\nu} \partial_\nu f
+ g_{\al\be} \partial_i g^{\al\mu} \partial_\mu f \partial_j g^{\be\nu} \partial_\nu f 
\\
&+ g_{\al j} \partial_i g^{\al\mu} \partial_\mu f + g_{\be i} \partial_j g^{\be\mu} 
\partial_\mu f + g_{ij},
\end{split}
\label{g_tilde_def}
\end{align}
where $\widetilde{g}$ is the induced metric on $\widetilde{r}(\partial \Om)$
and
\begin{align}
\begin{split}
T_\al(f) & = \ve_{\al\be\ga}( g^{\be\mu} \partial_{\mu 1} f + \partial_1 g^{\be\mu}
\partial_\mu f)
(g^{\ga\mu} \partial_{\mu 2} f + \partial_2 g^{\ga \mu} \partial_\mu f)
\\
& + \ve_{\al\be 2}(g^{\be\mu} \partial_{\mu 1} f + \partial_1 g^{\be\mu} \partial_\mu f)
+ \ve_{\al 1 \ga}(g^{\ga\mu} \partial_{\mu 2} f + \partial_2 g^{\ga \mu} \partial_\mu f).
\end{split}
\label{T_def}
\end{align}
Recall that $g_{33} =1$ and $g_{3 i } = 0$. Had this not been the case,
the first term inside the square root in the expression for $\widetilde{N}$
would be $g^{33}$.

We shall also need the inverse of the induced metric $\widetilde{g}$, which is given by
\begin{gather}
\widetilde{g}^{-1} = \frac{1}{\det \widetilde{g} }
\left(\begin{array}{cc}
\widetilde{g}_{22} & - \widetilde{g}_{12} \\
- \widetilde{g}_{21} & \widetilde{g}_{11}
\end{array}\right).
\label{g_tilde_inverse}
\end{gather}

We proceed to obtain more manageable expressions than those above.
 Using Taylor's theorem with integral remainder,
\begin{align}
\begin{split}
 & \frac{ 1 }{ 
\sqrt{ 1 + 2 g^{\la 3} T_\la(f) + T_\la(f) T^\la(f)  } } \\
&   = 1 - \left( g^{\la 3} T_\la(f) + \frac{1}{2}T_\la(f) T^\la(f) \right)
\bigintss_0^1 \frac{ 1-t}{ \left[1 + t(2g^{\la 3} T_\la(f) + T_\la(f)
T^\la(f) ) \right]^\frac{3}{2} } \, dt \\
& = 1 + \cM(f),
\end{split}
\label{Taylor1}
\end{align}
where $\cM(f)$ is defined by the above expression.
Combining (\ref{set_beta_id_N}), (\ref{T_def}), and  (\ref{Taylor1}),
\begin{align}
\begin{split}
\widetilde{N}^\al & = (g^{\al 3} + T^\al(f) )(1+ \cM(f) ) \\
& = g^{\al 3} + \cM^\al(f),
\end{split}
\label{N_tilde_expansion}
\end{align}
where $M^\al(f)$ is defined by the above expression.

We write (\ref{g_tilde_def}) as 
\begin{gather}
\widetilde{g}_{ij} = g_{ij} + \cN_{ij}(f),
\label{g_tilde_N_f}
\end{gather}
with $\cN_{ij}(f)$ defined in an obvious way. From this,
\begin{gather}
\det ( \widetilde{g} ) = \det(g) + \cD(f),
\nonumber
\end{gather} 
where $\cD(f)$ contains all terms in $\det ( \widetilde{g} )$ that depend 
on $f$. Using Taylor's theorem again,
\begin{align}
\begin{split}
\frac{1}{ \det ( \widetilde{g} )} & = 
\frac{1}{\det(g)} - \frac{\cD(f)}{(\det(g))^2}
\bigintss_0^1 \frac{ 1- t}{ \left( 1 + t \frac{\cD(f)}{\det(g)} \right)^2} \, dt.
\end{split}
\label{Taylor2}
\end{align}
From (\ref{g_tilde_inverse}), (\ref{g_tilde_N_f}), and (\ref{Taylor2}),
we find

\begin{align}
\begin{split}
\widetilde{g}^{-1} 
& = 
\left( 
\frac{1}{\det(g)} - \frac{\cD(f)}{(\det(g))^2}
 \bigintss_0^1 \frac{ 1- t}{ \left( 1 + t \frac{\cD(f)}{\det(g)} \right)^2} \, dt 
\right)
\Big(\begin{array}{cc}
g_{22} + \cN_{22}(f) & - g_{12} - \cN_{12}(f) \\
- g_{21}- \cN_{21}(f)  & g_{11} + \cN_{11}(f) 
\end{array} \Big ) \\
& = g^{-1} + \cF(f),
\end{split}
\nonumber
\end{align}
where the matrix $\cF(f)$ is defined by this expression and contains all 
contributions in $f$. We write the above component-wise as
\begin{gather}
\widetilde{g}^{ij} = g^{ij} + \cF^{ij}(f).
\label{g_tilde_expansion}
\end{gather}
We remark that  $-g^{\al3} \partial_\al f = \partial_\nu f$ (the outer normal derivative of $f$). 
Also, since the boundary Laplacian $\overline{\Delta}$ is given by 
(see remark \ref{indices_convention})
\begin{gather}
\overline{\Delta} = \frac{1}{\sqrt{|g|}} \partial_i\left ( \sqrt{|g|} g^{ij} \partial_j \right ),
\nonumber
\end{gather}
with $|g| =\det(g_{ij})$, we can write
\begin{gather}
g^{ij} g^{\al 3} \partial_{ij  \al} f  =
 -\overline{\Delta} \partial_\nu f + \mathcal{X}(f),
\nonumber
\end{gather}
where $\cX(f)$, which is defined by this expression, involves at most second
derivatives of $f$. We also note 
that in the present system of coordinates, the second fundamental form
of $\partial \Om$ is simply $\Ga_{ij}^3$, where $\Ga$ are the Christoffel
symbols, and the mean curvature of $\partial \Omega$, which we denote
$\cA_{\partial \Om}$, is $g^{ij} \Ga_{ij}^3$.
From these observations,  (\ref{sff_def}), (\ref{mean_curvature_def}), (\ref{set_beta_id_X}), (\ref{set_beta_id_partial_X}),
(\ref{N_tilde_expansion}), and (\ref{g_tilde_expansion}),
we conclude that 
\begin{gather}
\widetilde{F} = \cA \circ \widetilde{\eta} = 
 -\overline{\Delta} \partial_\nu f  
 -\frac{1}{2} \cA_{\partial \Om} \overline{\Delta} f + 
   \psQ f  +  \psq f 
   + \cA_{\partial \Om}
      \text{ on } \partial \Om,
\label{F_tilde_operator}
\end{gather}
where $\psQ$ 
and $\psq$ are, respectively,
third- and second-order pseudo-differential operators.
In the present coordinate system, they take the following form
\begin{align}
\begin{split}
\psQ h & = \left( \cF^{ij}(f) g^{\al 3} + g^{ij} \cM^\al(f) 
+\cF^{ij}(f) \cM^\al(f) \right)\partial_{\al ij} h,
\end{split}
\label{psQ_linear}
\end{align}
and
\begin{align}
\begin{split}
\psq h & = 
\left( \cF^{ij}(f) + g^{ij} \cM^\al g_{\al\be} + g_{\al\be} \cF^{ij}(f) \cM^\al(f) 
\right)
\left( 2 \partial_{(i } g^{3\mu}
\partial_{j)} h + \partial_{ij} g^{3\mu} \partial_\mu h \right)  
\\ 
&
+ \left( \Ga_{i\mu}^3 \cF^{ij}(f) + g^{ij} g_{\al\be} \Ga_{i\mu}^\be \cM^\al(f)
+ g_{\al\be} \Ga_{i\mu}^\be \cF^{ij}(f) \cM^\al(f) \right)
\Big( g^{\mu \la} \partial_{\la j} h  \\
& + \partial_j g^{\mu\la}\partial_\la h \Big)
+ \cQ^{ij}(h).
\end{split}
\label{psq_linear}
\end{align}
In the above, $\cQ^{ij}(h)$ is the linear operator in $h$
with $f$-dependent coefficients, 
naturally associated with the term
$g_{\al\be} \Ga_{ij}^\be \cF^{ij}(f) \cM^\al(f) + \Ga^3_{ij} \cF^{ij}(f) + g_{\al\be}g^{ij} \Ga_{ij}^3 \cM^\al(f)$
that figures in the mean curvature. More precisely, from our definitions it follows
that
\begin{gather}
g_{\al\be} \Ga_{ij}^\be \cF^{ij}(f) \cM^\al(f) + \Ga^3_{ij} \cF^{ij}(f) + g_{\al\be}g^{ij} \Ga_{ij}^3 \cM^\al(f) =
 a^i(f)\partial_{i 3} f + b^{ij}(f) \partial_{ij} f + 
c^\al(f) \partial_\al f,
\nonumber
\end{gather}
where $a^i$, $b^{ij}$, and $c^\al$ are smooth functions of $f$
and its derivatives of order at most two, provided that $f$ is small 
(see the beginning of this section). Then,
\begin{gather}
\cQ^{ij}(h) = 
a^i(f)\partial_{i 3} h + b^{ij}(f) \partial_{ij} h + 
c^\al(f) \partial_\al h.
\nonumber
\end{gather}
 The explicit form of  of $\cQ^{ij}$ is
too long and cumbersome, and will not be necessary for our purposes.

Summing up, $\psQ$ and $\psq$ are,
 respectively,
third- and second-order pseudo-differential operators 
whose coefficients depend smoothly on $f$ and its derivatives of at most 
second order, and such that $\psQ  = \psq$ = 0 if $f=0$ (in particular, 
$\psQ$ and $\psq$ contain no zeroth order terms in $f$)

\section{Analysis of $\left. f \right|_{\partial \Om}$\label{section_f_boundary}}
We shall now work modulo constants. This suffices to our purposes since we are interested
in obtaining estimates for $\nabla f$. Doing so, we can drop the gradient
in front of every term (\ref{new_form_f_eq}), obtaining an equation for $f$ which,
upon restriction to the boundary, gives an equation for $\left. f \right|_{\partial \Om}$.
It reads, after using (\ref{F_tilde_operator}),
\begin{align}
\begin{split}
 \ddot{f} &
 -\kappa \overline{\Delta} \partial_\nu f  
 -\frac{1}{2} \kappa \cA_{\partial \Om} \overline{\Delta} f
  + 
   \kappa \psQ f  +  \kappa \psq f 
 \\
& +
\kappa  \Delta_\nu^{-1} \dive \Big[ \Big(\nabla \cH_{\widetilde{\eta}} (
 -\overline{\Delta} \partial_\nu f  
 -\frac{1}{2} \cA_{\partial \Om} \overline{\Delta} f + 
   \psQ f  \\
   & +  \psq f 
    ) \Big) (-\id + (D\widetilde{\eta})^{-1} ) \Big] 
 \\
& 
- \kappa  \Delta_\nu^{-1} \dive\Big[ D^2 f L_1^{-1} P \Big(
\Big(\nabla \cH_{\widetilde{\eta}} ( 
 -\overline{\Delta} \partial_\nu f  
 -\frac{1}{2} \cA_{\partial \Om} \overline{\Delta} f + 
   \psQ f  \\
   & +  \psq f 
     )\Big)
(D\widetilde{\eta})^{-1} \Big) \Big] \\
& + 2 \Delta_\nu^{-1} \dive\left( D_v \nabla \dot{f} \right)
- 2  \Delta_\nu^{-1} \dive\left( D^2 f L_1^{-1} P D_v \nabla \dot{f} \right) \\
& +  \Delta_\nu^{-1} \dive\left( D_{vv}^2 \nabla f\right)
-  \Delta_\nu^{-1} \dive\left( D^2 f L_1^{-1} P D_{vv}^2 \nabla f \right) \\
& +  \Delta_\nu^{-1} \dive\left( D^2 f Q(\nabla_v v) \right)
-  \Delta_\nu^{-1} \dive\left( D^2 f L_1^{-1} P D^2f Q(\nabla_v v ) \right) \\
& +   \Delta_\nu^{-1} \dive\left(  \nabla \widetilde{q}_0 \left( 
(D \widetilde{\eta})^{-1} - \id \right) \right) 
 -  \Delta_\nu^{-1} \dive\left( 
D^2 f L_1^{-1} P \left( \nabla \widetilde{q}_0( (D\widetilde{\eta})^{-1} - \id ) \right)
\right) \\
= &  -  \Delta_\nu^{-1} \dive\left( \nabla_v v \right),
\, \text{ on } \, \partial \Om.
\end{split}
\label{eq_f_bry}
\end{align}
Above, $\cH$ is the harmonic extension operator in the domain 
$\partial \widetilde{\eta} (\Om) \equiv \partial ((\id +\nabla f)(\Om))$, and we recall
that  $\cH_{\widetilde{\eta}}$ is given by (see notation \ref{notation_sub})
\begin{gather}
\cH_{\widetilde{\eta}}(h) = (\cH( h \circ \widetilde{\eta}^{-1} ))\circ \widetilde{\eta},
\nonumber
\end{gather}
for $h: \partial \Om \rar \RR$.  In (\ref{eq_f_bry}),
the function $h$ in the argument of $\cH_{\widetilde{\eta}}$ is
\begin{gather}
 -\overline{\Delta} \partial_\nu f
-\frac{1}{2} \cA_{\partial \Om} \overline{\Delta} f + 
\psQ f   +  \psq f.
\nonumber
\end{gather} 

The operators $\psQ$ and
$\psq$ were defined in section \ref{section_local_coord}, although their precise
form will not be important here. Rather, it will be important that 
\begin{gather}
\p \psQ h \p_{s,\partial} \leq C \p f \p_{s+2,\partial} \p h \p_{s+3,\partial},
\nonumber
\end{gather}
and
\begin{gather}
\p \psq h \p_{s,\partial} \leq C \p f \p_{s+2,\partial} \p h \p_{s+2,\partial}.
\nonumber
\end{gather}
\begin{remark}
In (\ref{eq_f_bry}), the terms in $\kappa \cA_{\partial \Om}$ have been
dropped because $\cA_{\partial \Om}$ is assumed constant,
and therefore these terms do not contribute in (\ref{new_form_f_eq}).
Had $\cA_{\partial \Om}$ not been constant, such terms, linear in $\kappa$,
would grow without bound in the limit $\kappa \rar \infty$,
agreeing with the intuitive idea discussed in the introduction that
no convergence should be obtained in this case. 
The term $-\widetilde{q}_0$ corresponding to $- \nabla \widetilde{q}_0$ does not figure in equation (\ref{eq_f_bry}) either
in that $\widetilde{q}_0$ vanishes on $\partial \Om$, since $p_0$ vanishes on 
$\partial \eta(\Om)$.
\label{remark_constant_mean_curvature}
\end{remark}
(\ref{eq_f_bry}) was derived under the assumption that $f$ is a sufficiently
regular and small solution to (\ref{system_f_v_full}). In particular,
$f$ was thought of as defined over $\Om$. When viewing (\ref{eq_f_bry})
as an equation for $\left. f \right|_{\partial \Om}$, it is important to 
realize that it depends on the way $\left. f \right|_{\partial \Om}$
is extended to $\Om$. We are ultimately interested in the case
when such an extension is carried out using proposition \ref{embedding}.
However, in order to apply techniques of continuous semi-groups,
we need the extension to be defined for any function in $H^{s+2}(\partial \Om)$,
$s > \frac{3}{2} + 2$, and not only for those that are sufficiently small.

Let $\psi: \RR \rar [0,1]$ be a smooth monotone function such that 
$\psi =1$ on $(-\infty,\frac{\de_0}{3}]$, 
and $\psi = 0$ on $[2\frac{\de_0}{3}, \infty)$, 
where $\de_0$ is given by propostion \ref{embedding}.
 Given $h \in H^{s+2}(\partial \Om)$,
consider the problem
\begin{subnumcases}{\label{nl_Dirichlet_mod}}
\Delta f + \psi(\p h \p_{s+2,\partial}^2 ) \cN(f)  = 0 & in $\Om$,
 \label{nl_Dirichlet_mod_f} \\
f = h & on $\partial \Om$. \label{nl_Dirichlet_bry_mod}
\end{subnumcases}
where $\cN$ is as in (\ref{non_linear_terms_f_extension}).

To solve (\ref{nl_Dirichlet_mod}), we argue as in the proof
of proposition \ref{embedding}. We define a map
\begin{align}
& F: H^{s+2}(\partial \Om) \times H^{s+\frac{5}{2}}(\Om) 
\rar H^{s+2}(\partial \Om) \times H^{s + \frac{1}{2}}(\Om) ,
\nonumber \\
{\rm by} \:\:\:\:\:\: & F(h, f) = ( f\left|_{\partial \Om}\right. - h, \Delta f +\psi(\p h \p_{s+2,\partial}^2 )\cN(f) ),
\nonumber
\end{align}
and notice that $F$ is $C^1$, since in a Hilbert space the square norm
function is continuously differentiable. Moreover, its linearization
is given by (\ref{linearization}). Thus, arguing as in proposition
\ref{embedding},  
we find a small $r>0$ and a map
$\varphi :\cB_{r}^{s+2}(\partial \Om) \rar H^{s+\frac{5}{2}}(\Om)$,
where $\varphi(h) = f$
 gives a unique
solution to (\ref{nl_Dirichlet_mod}). Noticing that 
 (\ref{linearization})
does not involve $\psi$, if $\de_0 >0$ in proposition
\ref{embedding} is sufficiently small, we can take
$r\geq \de_0$. Solutions to (\ref{nl_Dirichlet_mod})
agree with those of (\ref{nl_Dirichlet_f_3d_system}) if
$\p h \p_{s+2,\partial} <  \sqrt{\frac{\de_0}{3}}$, and with the harmonic
extension of $h$ if $\p h \p_{s+2,\partial} > \sqrt{\frac{2\de_0}{3}}$. 
Defining $\varphi$ as the harmonic extension when 
$\p h \p_{s+2,\partial} > \sqrt{\frac{2\de_0}{3}}$, we obtain a map
$\varphi: H^{s+2}(\partial \Om) \rar H^{s+\frac{5}{2}}(\Om)$, given
by $\varphi(h) = f$, where $f$ solves (\ref{nl_Dirichlet_mod}), and agrees
with the solution of (\ref{nl_Dirichlet_f_3d_system}) if $h$ is
sufficiently small. 

Furthermore, since $\varphi$ is continuous, given $\epsilon >0$ we
can choose $\de_0$ so small that 
$\p \varphi(h) \p_{s+\frac{5}{2}} < \epsilon$. But if $f$ is a solution of (\ref{nl_Dirichlet_mod}) with $\p f \p_{s+\frac{5}{2}} \leq \epsilon$, 
and 
$\epsilon$ is sufficiently small, then by elliptic theory the solution  
obeys the estimate
\begin{gather}
\p f \p_{s+\frac{5}{2}} \leq C \p h \p_{s+ 2, \partial  },
\label{elliptic_estimate_f_bry}
\end{gather}
where the constant $C$ depends only on $\epsilon$, $s$, $\Om$.
Finally, if $h_H$ is the harmonic extension of $h$, from standard elliptic 
theory we have the estimate
\begin{gather}
\p f - h_H \p_{s+\frac{5}{2} }
\leq C \p f \p_{s+\frac{5}{2}}^2 \leq C \p f \p_{s+2,\partial}^2,
\label{f_near_harmonic}
\end{gather}
where the last inequality follows by (\ref{elliptic_estimate_f_bry}). In particular
\begin{gather}
\p f - h_H \p_{s+\frac{5}{2} } \leq C\de_0^2,
\label{f_near_harmonic_delta}
\end{gather}

\begin{definition}
We shall call the solution $f$ of (\ref{nl_Dirichlet_mod}) constructed above
the $\psi$-harmonic extension of $h$.
\end{definition}
Notice that the $\psi$-harmonic extension depends on the choice
of $\psi$ which, in turns, depends on $\de_0$. We fix these quantities once and for all.

\begin{notation}
Recall that $H_0(\partial \Om)$ denotes the Sobolev space modulo constants,
and that $\cA_{\partial \Om}$ is the mean curvature of $\partial \Om$.
\end{notation}

\begin{lemma} If  $\de_0$ is sufficiently small,  where $\de_0$ is defined 
as above, then the operator (which depends on $\de_0$)
$-\overline{\Delta} \partial_\nu - \frac{1}{2} \cA_{\partial \Om}
\overline{\Delta}: H^{s+2}_0(\partial \Om) \subset
H^0_0(\partial \Om) \rar H^{s}_0(\partial \Om) \subset H^0_0(\partial \Om)$,
where  $s > \frac{n}{2} + 2 = \frac{3}{2} + 2$ and  $\partial_\nu$
is computed using the $\psi$-harmonic extension to $\Om$,
 is an elliptic, positive,
invertible, 
third-order pseudo-differential operator.
\label{lemma_elliptic_pseudo_2}
\end{lemma}
\begin{proof}
If $\partial_\nu$ is computed using the harmonic extension,
then by known properties 
of the Neumann
operator \cite{TayPSO}, $-\overline{\Delta} \partial_\nu$ 
is an elliptic,  
invertible, 
third-order pseudo-differential operator.
Thus, the same holds true when the $\psi$-harmonic extension is used
because of (\ref{f_near_harmonic_delta}), and in particular 
$-\overline{\Delta} \partial_\nu  - \frac{1}{2} \cA_{\partial \Om} \overline{\Delta}$
is also a third-order pseudo-differential operator.

For what follows, recall that the constancy of $\cA_{\partial \Om}$ implies that $\partial \Om$
is a sphere 
\cite{Ale}, and therefore $\cA_{\partial \Om} > 0$. 

Suppose that 
$-\overline{\Delta} \partial_\nu f - \frac{1}{2} \overline{\Delta} \cA_{\partial \Om} f = 0$. Since
the kernel of $-\overline{\Delta}$ is zero in $H^{s+2}_0(\partial \Om)$, we have
\begin{gather}
\partial_\nu f + \frac{1}{2} \cA_{\partial \Om} f = 0.
\nonumber
\end{gather}
Using integration  by parts,
\begin{align}
\begin{split}
0=  \int_{\partial \Om} f (\partial_\nu f + \frac{1}{2} \cA_{\partial \Om} f )
= \int_\Om |\nabla f|^2 - \int_\Om f 
\psi(\p f \p_{s+2,\partial}^2 ) \cN(f)
+ \frac{1}{2} \int_{\partial \Om} \cA_{\partial \Om} f^2,
\end{split}
\label{proof_elliptic_1}
\end{align}
where we recall that $\cN$ is given by (\ref{non_linear_terms_f_extension}).
Using the Cauchy-Schwarz inequality, (\ref{bilinear}), (\ref{non_linear_terms_f_extension}), $|\psi| \leq 1$, the interpolation inequality, and
the fact that under our assumptions $s+\frac{5}{2} > 6$, we see that
\begin{gather}
\Big | \int_\Om f 
\psi(\p f \p_{s+2,\partial}^2 ) \cN(f)\Big |
\leq C \p f \p_1 \p f \p_2 \p f \p_4 \leq C \p f \p_1^{1 + \frac{6}{5} }
\p f \p_6^\frac{4}{5}.
\label{proof_elliptic_2}
\end{gather}
Also since we are working modulo constants: 
\begin{gather}
\p f \p_0 \leq C \p \nabla f \p_0.
\label{proof_elliptic_3}
\end{gather}
Combining (\ref{proof_elliptic_1}), (\ref{proof_elliptic_2}), and
(\ref{proof_elliptic_3}), we find
\begin{gather}
0 \geq \frac{\cA_{\partial \Om}}{2} \p f \p_0^2 
+ \p f \p_1^2( 1 - \p f \p_1^\frac{1}{5} \p f \p_6^\frac{4}{5} ),
\nonumber
\end{gather}
and thus we conclude that $f=0$ if $\de_0$ is very small (recall 
(\ref{elliptic_estimate_f_bry})). The invertibility result now follows from the 
Fredholm alternative, and positivity from that of $\overline{\Delta}\partial_\nu$.
\end{proof}

\begin{remark}
In what follows, we will use the fact that $\Delta_\nu^{-1} \circ \dive$ is a bounded
linear map between $H_0^s(\Om)$ and $H^{s+1}_0(\Om)$.
\end{remark}
In order to study (\ref{eq_f_bry}), we first consider  a linear
equation, with $f$-dependent coefficients, naturally associated with 
(\ref{eq_f_bry}). Since the map $f \mapsto 
(D\widetilde{\eta})^{-1} = (\id + D^2f)^{-1}$ is not linear, we consider
the following linear map connected to $(D\widetilde{\eta})^{-1}$.
For $f$ small in $H^{s+\frac{5}{2}}(\Om)$, the Sobolev embedding theorem 
tells us that $(D\widetilde{\eta})^{-1}$ is well-defined, and one can 
write
\begin{align}
\begin{split}
 (D\widetilde{\eta})^{-1} - \id
= (D\widetilde{\eta})^{-1}(\id - D\widetilde{\eta}) = - (D\widetilde{\eta})^{-1} D^2f.
\end{split}
\nonumber
\end{align}
This suggests considering the linear map 
\begin{gather}
B_f(h) = - (D\widetilde{\eta})^{-1} D^2h,
\nonumber
\end{gather}
which satisfies the estimate
\begin{gather}
\p B_f(h) \p_{s+\frac{1}{2}} \leq
C \left( 1 +  \p f \p_{s+\frac{5}{2}} \right)  
\p h \p_{s+\frac{5}{2}} \leq C\ \p h \p_{s+\frac{5}{2}},
\nonumber
\end{gather}
provided that $f$ is small. 
\begin{notation}
It is convenient to write 
$B_f(D^2h)$ for $B_f(h)$ to facilitate keeping track of the number of 
derivatives.
\end{notation}

We are thus led to the following 
linear equation for $h$:
\begin{align}
\begin{split}
 \ddot{h} &
 -\kappa \overline{\Delta} \partial_\nu h  
 -\frac{1}{2} \kappa \cA_{\partial \Om} \overline{\Delta} h
  + 
   \kappa \psQ h  +  \kappa \psq h \\
& +
\kappa  \Delta_\nu^{-1} \dive \Big[ \Big(\nabla \cH_{\widetilde{\eta}} (
 -\overline{\Delta} \partial_\nu h  
 -\frac{1}{2} \cA_{\partial \Om} \overline{\Delta} h + 
   \psQ h  \\
   & +  \psq h 
   ) \Big) (-\id + (D\widetilde{\eta})^{-1} ) \Big] 
 \\
& 
- \kappa  \Delta_\nu^{-1} \dive\Big[ D^2 f L_1^{-1} P \Big(
\Big(\nabla \cH_{\widetilde{\eta}} ( 
 -\overline{\Delta} \partial_\nu h  
 -\frac{1}{2} \cA_{\partial \Om} \overline{\Delta} h + 
   \psQ h  \\
   & +  \psq h 
      )\Big)
(D\widetilde{\eta})^{-1} \Big) \Big] \\
& + 2 \Delta_\nu^{-1} \dive\left( D_v \nabla \dot{h} \right)
- 2  \Delta_\nu^{-1} \dive\left( D^2 f L_1^{-1} P D_v \nabla \dot{h} \right) \\
& +  \Delta_\nu^{-1} \dive\left( D_{vv}^2 \nabla h\right)
-  \Delta_\nu^{-1} \dive\left( D^2 f L_1^{-1} P D_{vv}^2 \nabla h \right) \\
& +  \Delta_\nu^{-1} \dive\left( D^2 h Q(\nabla_v v) \right)
-  \Delta_\nu^{-1} \dive\left( D^2 f L_1^{-1} P D^2 h  Q(\nabla_v v ) \right) \\
& 
 +   \Delta_\nu^{-1} \dive\left(  \nabla \widetilde{q}_0 
 B_f(D^2h) \right) 
 - \Delta_\nu^{-1} \dive\left( 
D^2 f L_1^{-1} P \left( \nabla \widetilde{q}_0( B_f(D^2h) \right) \right) \\
&
= -  \Delta_\nu^{-1} \dive\left( \nabla_v v \right),
\, \text{ on } \, \partial \Om.
\end{split}
\label{eq_h_bry}
\end{align}
In (\ref{eq_h_bry}), it is still understood, as before, that 
$\widetilde{\eta} = \id + D f$. A simpler, also linear, problem
connected to  (\ref{free_boundary_full}) and 
(\ref{eq_f_bry}), was studied by the first author 
in \cite{Dis_linear}.

We shall write  (\ref{eq_h_bry}) as a first order
system. In view of   lemma 
\ref{lemma_elliptic_pseudo_2},
the operator $\cL: H_0^{s+3}(\partial \Om) \rar H_0^s(\partial \Om)$
 given by
\begin{gather}
\cL = -\overline{\Delta} \partial_\nu - \frac{1}{2} \cA_{\partial \Om}
\overline{\Delta},
\nonumber
\end{gather}
(where $\partial_\nu$ is computed using the $\psi$-harmonic extension to $\Om$)
has a square root 
$\cS: H_0^{s+\frac{3}{2}}(\partial \Om) \rar H_0^s(\partial \Om)$, i.e.,
\begin{gather}
\cS^2 = \cL.
\label{cS_defi}
\end{gather}
(see, e.g., \cite{K_Linear}).
Lemma \ref{lemma_elliptic_pseudo_2} also implies that $\cS^{-1}$ exists.

Letting $z = (\sqrt{\kappa} \cS h, \dot{h} )$ 
we will construct solutions to (\ref{eq_h_bry}) by analyzing 
the following  system for $z$:
\begin{gather}
\partial_t z + A_\kappa(t) z = \cG, \, \text{ on } \partial \Om,
\label{first_order_simple}
\end{gather}
where 
\begin{gather}
\cG \in 
C^0( [0,T], H^{s+\frac{1}{2}}_0(\partial \Om) \times H^{s+\frac{1}{2}}_0(\partial \Om) ) \cap C^1( [0,T], H^{s-1}_0(\partial \Om) \times H^{s-1}_0(\partial \Om) ),
\nonumber
\end{gather}
$T > 0$,
is given by
\begin{align}
\begin{split}
\cG = ( 0,   -  \Delta_\nu^{-1} \dive\left( \nabla_v v \right) ).
\end{split}
\label{cG_definition}
\end{align}
With 
\begin{gather}
f \in L^\infty([0,T],H^{s+2}(\partial \Om)) \cap C^0([0,T], H^{s+1}(\partial \Om)),
\nonumber
\end{gather}
and $f$ extended to $\Om$ via the $\psi$-harmonic extension, 
\begin{gather}
v \in C^0([0,T], H^s(\Om, \RR^3))\, \text{with}\, \dive(v) = 0,
\nonumber
\end{gather}
and
\begin{gather}
\widetilde{q}_0 \in 
C^0 ([0,T],H^{s+1}(\Om,\RR)),
\nonumber
\end{gather}
$s > \frac{3}{2} + 2$,
$A(t)=A_\kappa(t)$ is a one-parameter family of operators
$A(t): H_0^{s+\frac{1}{2}}(\partial \Om) \rar
H_0^{s-1}(\partial \Om)$  written as a finite sum
\begin{gather}
A = \sum_{i=0}^{12} \cZ_i,
\nonumber
\end{gather}
with the  operators $\cZ_i = \cZ_i(t)$, in turn,  given as follows.
\begin{align}
\begin{split}
\cZ_0 = 
\sqrt{\kappa}
 \left(
\begin{matrix}
0 & -1 \\
1 & 0
\end{matrix}
\right) 
\cS. 
\end{split}
\nonumber
\end{align}

\begin{align}
\begin{split}
\cZ_1 =   
\sqrt{\kappa}\left(
\begin{matrix}
0 & 0 \\
1 & 0
\end{matrix}
\right) 
\cW_1,
\end{split}
\nonumber
\end{align}
where
\begin{align}
\begin{split}
& \cW_1: H_0^{s+\frac{1}{2}}(\partial \Om) \rar H_0^{s-1}(\partial \Om), \\
& h \mapsto \left. \left( \psQ \cS^{-1} h \right) \right|_{\partial \Om} .
\end{split}
\nonumber
\end{align}

\begin{align}
\begin{split}
\cZ_2 =   
\sqrt{\kappa}\left(
\begin{matrix}
0 & 0 \\
1 & 0
\end{matrix}
\right) 
\cW_2,
\end{split}
\nonumber
\end{align}
where
\begin{align}
\begin{split}
& \cW_2: H_0^{s+\frac{1}{2}}(\partial \Om) \rar H_0^{s}(\partial \Om), \\
& h \mapsto \left. \left( \psq \cS^{-1} h \right) \right|_{\partial \Om} .
\end{split}
\nonumber
\end{align}

\begin{align}
\begin{split}
\cZ_3 =   
\sqrt{\kappa}\left(
\begin{matrix}
0 & 0 \\
1 & 0
\end{matrix}
\right) 
\cW_3,
\end{split}
\nonumber
\end{align}
where
\begin{align}
\begin{split}
& \cW_3: H_0^{s+\frac{1}{2}}(\partial \Om) \rar H_0^{s-1}(\partial \Om), \\
& h \mapsto \left. \left( 
\Delta_\nu^{-1} \dive \Big[ \Big(\nabla \cH_{\widetilde{\eta}} (
\cS h +    \psQ \cS^{-1} h   +  \psq \cS^{-1} h 
   ) \Big) (-\id + (D\widetilde{\eta})^{-1} ) \Big]
 \right) \right|_{\partial \Om} .
\end{split}
\nonumber
\end{align}

\begin{align}
\begin{split}
\cZ_4 =   
- \sqrt{\kappa}\left(
\begin{matrix}
0 & 0 \\
1 & 0
\end{matrix}
\right) 
\cW_4,
\end{split}
\nonumber
\end{align}
where
\begin{align}
\begin{split}
& \cW_4: H_0^{s+\frac{1}{2}}(\partial \Om) \rar H_0^{s-1}(\partial \Om), \\
& h \mapsto  \left. \left( 
\Delta_\nu^{-1} \dive\Big[ D^2 f L_1^{-1} P \Big(
\Big(\nabla \cH_{\widetilde{\eta}} ( 
 \cS h + 
   \psQ  \cS^{-1} h  +  \psq \cS^{-1} h 
     )\Big)
(D\widetilde{\eta})^{-1} \Big) \Big]
\right) \right|_{\partial \Om} .
\end{split}
\nonumber
\end{align}

\begin{align}
\begin{split}
\cZ_5 =  2 
\left(
\begin{matrix}
0 & 0 \\
0 & 1
\end{matrix}
\right) 
\cW_5,
\end{split}
\nonumber
\end{align}
where
\begin{align}
\begin{split}
& \cW_5: H_0^{s+\frac{1}{2}}(\partial \Om) \rar H_0^{s-\frac{1}{2}}(\partial \Om), \\
& h \mapsto \left. \left( \Delta_\nu^{-1} \dive\left( D_v \nabla h \right) \right)\right|_{\partial \Om} .
\end{split}
\nonumber
\end{align}

\begin{align}
\begin{split}
\cZ_6 =  -2 
\left(
\begin{matrix}
0 & 0 \\
0 & 1
\end{matrix}
\right) 
\cW_6,
\end{split}
\nonumber
\end{align}
where
\begin{align}
\begin{split}
& \cW_6: H_0^{s+\frac{1}{2}}(\partial \Om) \rar H_0^{s-\frac{1}{2}}(\partial \Om), \\
& h \mapsto \left(   \Delta_\nu^{-1} \dive \left. \left( D^2 f L_1^{-1} P D_v \nabla h \right) \right)\right|_{\partial \Om} .
\end{split}
\nonumber
\end{align}

\begin{align}
\begin{split}
\cZ_7 = \frac{1}{\sqrt{\kappa}} 
\left(
\begin{matrix}
0 & 0 \\
1 & 0
\end{matrix}
\right) 
\cW_7, 
\end{split}
\nonumber
\end{align}
where
\begin{align}
\begin{split}
& \cW_7: H_0^{s+\frac{1}{2}}(\partial \Om) \rar H_0^{s}(\partial \Om), \\
& h \mapsto \left. \left( \Delta_\nu^{-1} \dive\left( D^2_{vv} \nabla \cS^{-1} h \right) \right)\right|_{\partial \Om}. 
\end{split}
\nonumber
\end{align}

\begin{align}
\begin{split}
\cZ_8 = - \frac{1}{\sqrt{\kappa}} 
\left(
\begin{matrix}
0 & 0 \\
1 & 0
\end{matrix}
\right) 
\cW_8, 
\end{split}
\nonumber
\end{align}
where
\begin{align}
\begin{split}
& \cW_8: H_0^{s+\frac{1}{2}}(\partial \Om) \rar H_0^{s}(\partial \Om), \\
& h \mapsto \left. \left( \Delta_\nu^{-1} \dive\left( 
D^2f L_1^{-1} P D^2_{vv} \nabla \cS^{-1} h \right) \right)\right|_{\partial \Om}. 
\end{split}
\nonumber
\end{align}

\begin{align}
\begin{split}
\cZ_9 =  \frac{1}{\sqrt{\kappa}} 
\left(
\begin{matrix}
0 & 0 \\
1 & 0
\end{matrix}
\right) 
\cW_9 , 
\end{split}
\nonumber
\end{align}
where
\begin{align}
\begin{split}
& \cW_9: 
H_0^{s+\frac{1}{2}}(\partial \Om) \rar H_0^{s+\frac{1}{2}}(\partial \Om), \\
& h \mapsto \left. \left( \Delta_\nu^{-1} \dive\left( ( D^2 \cS^{-1} h ) Q(\nabla_v v) \right)\right)\right|_{\partial \Om} .
\end{split}
\nonumber
\end{align}

\begin{align}
\begin{split}
\cZ_{10} =  -\frac{1}{\sqrt{\kappa}} 
\left(
\begin{matrix}
0 & 0 \\
1 & 0
\end{matrix}
\right) 
\cW_{10} , 
\end{split}
\nonumber
\end{align}
where
\begin{align}
\begin{split}
& \cW_{10}: 
H_0^{s+\frac{1}{2}}(\partial \Om) \rar H_0^{s+\frac{1}{2}}(\partial \Om), \\
& h \mapsto \left. \left( \Delta_\nu^{-1} \dive\left( 
( D^2f L_1^{-1} P D^2 \cS^{-1} h ) Q(\nabla_v v) \right) \right) 
\right|_{\partial \Om} .
\end{split}
\nonumber
\end{align}

\begin{align}
\begin{split}
\cZ_{11} = 
 \frac{1}{\sqrt{\kappa}} 
\left(
\begin{matrix}
0 & 0 \\
1 & 0
\end{matrix}
\right) 
\cW_{11},
\end{split}
\nonumber
\end{align}
where
\begin{align}
\begin{split}
& \cW_{11}: H_0^{s+\frac{1}{2}}(\partial \Om) \rar H_0^{s+\frac{1}{2}}(\partial \Om), \\
& h \mapsto \left. \left( \Delta_\nu^{-1} \dive\left( \nabla \widetilde{q}_0 
B_f( D^2 \cS^{-1} h ) \right) \right)\right|_{\partial \Om}. 
\end{split}
\nonumber
\end{align}

\begin{align}
\begin{split}
\cZ_{12} = 
- \frac{1}{\sqrt{\kappa}} 
\left(
\begin{matrix}
0 & 0 \\
1 & 0
\end{matrix}
\right) 
\cW_{12},
\end{split}
\nonumber
\end{align}
where
\begin{align}
\begin{split}
& \cW_{12}: H_0^{s+\frac{1}{2}}(\partial \Om) \rar H_0^{s+\frac{1}{2}}(\partial \Om), \\
& h \mapsto \left. \left( 
\Delta_\nu^{-1} \dive\left( 
D^2 f L_1^{-1} P \left( \nabla \widetilde{q}_0( B_f(D^2 S^{-1}h) \right) \right)
\right) \right|_{\partial \Om}. 
\end{split}
\nonumber
\end{align}
In the above, we use the fact that $\dive(v) =0$ implies $Q(\nabla_v v) \in H^s(\Om)$.
In these expression, again we assume that quantities are extended
to $\Om$ via the $\psi$-harmonic extension when necessary.

\begin{notation}
We denote by $\cM \equiv \cM_{R,\ell,T,s+ 2}$ the 
set of functions $f: [0,T] \rar H_0^{s+ 2}(\partial \Om,\RR)$ such that
\begin{gather}
\p f(t) \p_{s+ 2,\partial} \leq R, \,
\text{ and }
\p f(t) - f(t^\prime) \p_{s+\frac{1}{2},\partial} \leq \ell |t - t^\prime|,
\nonumber
\end{gather}
$0 \leq t, t^\prime \leq T$, where $s > \frac{n}{2} + 2$.
\label{cM_notation}
\end{notation}

\begin{remark}
For future reference, we note that $f \in C^0([0,T], H^{s+1}(\partial \Om))$
if $f \in \cM$. Indeed, an application of the interpolation inequality gives
\begin{align}
\begin{split}
\p f(t) - f(t^\prime) \p_{s+1,\partial}&
 \leq \p f(t) - f(t^\prime) \p^\frac{2}{3}_{s+\frac{1}{2},\partial}
 \p f(t) - f(t^\prime) \p^\frac{1}{3}_{s+2,\partial} \\
 & \leq (2\ell^2 R)^\frac{1}{3} |t-t^\prime|^\frac{2}{3}.
\end{split}
\nonumber
\end{align}
\label{remark_continuity_future}
\end{remark}

\begin{remark}
In the proofs below, we make extensive use of  (\ref{bilinear}) to 
estimate the several products involved. We do not mention the application of
(\ref{bilinear})
at every step in order to avoid repetition.
\end{remark}

\begin{notation}
In the ensuing estimates, it will be important to keep track of the 
dependence of several constants on the constants
 $K_0$, $K_0^\prime$ and $K_3$ that appear 
in the hypotheses of the statements. This is done by writing
$C=C(K_0)$ and similar expressions. The dependence
on fixed quantities such as $s$, $\Om$ etc, however, will not be
indicated, nor will be the dependence on $\widetilde{q}_0$,
$\dot{\widetilde{q}}_0$, $v$, and $\dot{v}$, which are fixed throughout
the theorems of this section.
\label{notation_constants}
\end{notation}

\begin{proposition}
Assume that 
\begin{gather}
v \in C^0\left([0,T], H^s(\Om, \RR^3)\right), \, \dive(v) = 0,
\nonumber
\end{gather}
and
\begin{gather}
\widetilde{q}_0 \in 
C^0 \left([0,T],H^{s+1}(\Om,\RR)\right),
\nonumber
\end{gather}
$s > \frac{3}{2} + 2$, $T >0$. Let $f \in \cM$ satisfy
\begin{gather}
\p f \p_{s+2,\partial} \leq \frac{K_0}{\sqrt{\kappa}},
\nonumber
\end{gather}
for some constant $K_0$, $0 \leq t \leq T$, and assume that
 $f$ is extended
to $\Om$ via its $\psi$-harmonic extension.
Finally, let
\begin{gather}
\cG \in 
C^0( [0,T], H^{s+\frac{1}{2}}_0(\partial \Om) \times H^{s+\frac{1}{2}}_0(\partial \Om) ) \cap C^1( [0,T], H^{s-1}_0(\partial \Om) \times H^{s-1}_0(\partial \Om) ),
\nonumber
\end{gather}
and $z_0 \in H_0^{s+\frac{1}{2}}(\partial \Om, \RR) \times
H_0^{s+\frac{1}{2}}(\partial \Om, \RR)$ be given.
Then, if $\kappa$ is sufficiently large, there exists
a unique solution
\begin{gather}
z \in C^0([0,T], H_0^{s+\frac{1}{2}}(\partial \Om, \RR)) \cap 
C^1([0,T], H_0^{s-1}(\partial \Om, \RR))
\nonumber
\end{gather}
of (\ref{first_order_simple}), satisfying $z(0) = z_0$. (we note  that the operators
in equation (\ref{first_order_simple}) are constructed with the help of the 
$\psi$-harmonic extension.)
\label{prop_first_order}
\end{proposition}
\begin{proof}
Since $f$ is defined on $\Om$ by its $\psi$-harmonic extension,
we find that (using (\ref{elliptic_estimate_f_bry}))
\begin{gather}
\p f \p_{r+\frac{1}{2}} \leq C \p f \p_{r,\partial},
\nonumber
\end{gather}
and thus
\begin{gather}
\p f(t) \p_{s+ \frac{5}{2}} \leq \frac{K_0}{\sqrt{\kappa}}.
\nonumber
\end{gather}
We consider 
(\ref{first_order_simple}) as an abstract evolution equation in 
$X = H^0_0(\partial \Om) \times H^0_0(\partial \Om)$ for the unknown $z$.
Our goal is to verify that the operator $-A(t)$ satisfies the conditions
of theorem \ref{theorem_evolution_operator}, and then to apply theorem
\ref{theorem_abstract_ODE}. For this,
we take $Y$ in that theorem as $Y = H_0^{s+\frac{1}{2}}(\partial \Om) \times H^{s+\frac{1}{2}}_0(\partial \Om)$, and let $H^\frac{3}{2}_0(\partial \Om) \times
H_0^\frac{3}{2}(\partial \Om)$ be the domain of $A(t)$.
Denote
\begin{gather}
\widehat{A}_\kappa(t) \equiv \widehat{A}(t) \equiv  \widehat{A} = 
\frac{1}{\sqrt{\kappa}} A_\kappa(t).
\nonumber
\end{gather}
As $f$ is small in $H^{s+\frac{5}{2}}(\Om)$ if $\kappa$ is large,
the Sobolev embedding theorem guarantees 
that $(D\widetilde{\eta})^{-1}= (\id + D^2f)^{-1}$ and $L_1^{-1}$ are
well-defined.
Using 
Taylor's theorem with integral remainder to estimate
$(\id + D^2 f)^{-1}$, similarly to what was done in
the calculations of section  \ref{section_local_coord}, we obtain
\begin{gather}
\p (D\widetilde{\eta})^{-1} \p_{r-1} \leq C (1+ \p f \p_{r+1}),
\label{D_eta_inv_expansion_1}
\end{gather}
$1 \leq r \leq s$. 
In the above, and in what follows, we assume that $\kappa$ is sufficiently large to allow us to bound  powers of norms of $f$ by terms linear (in the norms of) $f$.
A similar application of Taylor's theorem also yields
\begin{gather}
\p -\id + (D\widetilde{\eta})^{-1} \p_{r-1} \leq C \p f \p_{r+1},
\label{D_eta_inv_expansion_2}
\end{gather}
$1\leq r \leq s$.
From (\ref{psQ_linear}), (\ref{psq_linear}), and
lemma \ref{lemma_elliptic_pseudo_2}, we obtain the estimates
\begin{gather}
\p \cW_1 h \p_{r-1,\partial} =
\p \psQ \cS^{-1} h \p_{r-1,\partial} \leq C \p f \p_{r+\frac{3}{2}}
\p h \p_{r+\frac{1}{2},\partial},
\label{quasi_linear_estimates_psQ}
\end{gather}
and
\begin{gather}
\p \cW_2 h \p_{r-1,\partial} = \p \psq \cS^{-1} h \p_{r-1,\partial} \leq C \p f 
\p_{r+\frac{3}{2}}
\p h \p_{r-\frac{1}{2},\partial},
\label{quasi_linear_estimates_psq}
\end{gather}
$1\leq r \leq s$.

We claim that if $g \in H^{s-1}(\partial \Om)$, then 
$\cH_{\widetilde{\eta}}(g) \in H^{s-\frac{1}{2}}(\Om)$. To see this, notice that
 $\widetilde{G} = \cH_{\widetilde{\eta}}(g)$ means $\widetilde{G} = G \circ \widetilde{\eta}$,
 where $G$ solves 
\begin{gather}
\begin{cases}
\Delta  G = 0, & \text{ in } \widetilde{\eta}(\Om), \\
G  = g \circ \widetilde{\eta}^{-1}, & \text{ on } \partial \widetilde{\eta} (\Om).
\end{cases}
\nonumber
\end{gather}
But using notation \ref{notation_sub}, this can be rewritten as
\begin{gather}
\begin{cases}
\Delta_{\widetilde{\eta}}  \widetilde{G} = 0, & \text{ in } \Om, \\
\widetilde{G}  = g , & \text{ on } \partial \Om.
\end{cases}
\nonumber
\end{gather}
If $\widetilde{G}$ is defined over $\Om$, we compute, 
\begin{gather}
\Delta_{\widetilde{\eta}} \widetilde{G}
= ((D \widetilde{\eta} )^{-1})^\al_\be ( \partial_{\al \ga} \widetilde{G}  ) 
((D \widetilde{\eta} )^{-1})^{\be \ga} 
+ 
\partial_\al 
((D \widetilde{\eta} )^{-1})^{\al\be} \partial_\be \widetilde{G}, 
\nonumber
\end{gather}
where $((D \widetilde{\eta} )^{-1})^\al_\be$ are the entries of $(D \widetilde{\eta} )^{-1}$.
Thus, $\Delta_{\widetilde{\eta}}$ has the form
\begin{gather}
\Delta_{\widetilde{\eta}} G = a^{\al\be}\partial_{\al\be} G + b^\al \partial_\al G.
\nonumber
\end{gather}
We see that that 
$a^{\al\be}$ is in $H^{s+\frac{1}{2}}(\Om,\RR^{3^2})$ and
$b^\al$ is in $H^{s-\frac{1}{2}}(\Om, \RR^3)$ since
$\nabla f \in H^{s+\frac{3}{2}}(\Om)$.
Furthermore, 
as $\nabla f$ is small in $H^{s+\frac{3}{2}}(\Om)$ and $s> \frac{3}{2} + 2$, 
we know that $(a^{\al\be})$ is positive definite.
Thus,  $\Delta_{\widetilde{\eta}}$ is an elliptic operator,
and the solution $\widetilde{G}$ will be in 
 $H^{s-\frac{1}{2}}(\Om)$ if $g$ is in $H^{s-1}(\partial \Om)$. Moreover,
 we have the estimate
\begin{gather}
\p \cH_{\widetilde{\eta}}(g) \p_{s-\frac{1}{2}} \leq C(f) \p g \p_{s-1,\partial},
\nonumber
\end{gather}
where the constant $C(f)$ depends on $\p \nabla f \p_{s+\frac{3}{2}}$.

Combining the above properties of $\cH_{\widetilde{\eta}}$ with 
(\ref{D_eta_inv_expansion_2}), 
(\ref{quasi_linear_estimates_psQ}), and (\ref{quasi_linear_estimates_psq}), 
it follows that
\begin{align}
\begin{split}
& \p \cW_3 h \p_{r-1,\partial} 
\\
&= 
\p 
\Delta_\nu^{-1} \dive \Big[ \Big(\nabla \cH_{\widetilde{\eta}} (
\cS h +    \psQ \cS^{-1} h   +  \psq \cS^{-1} h 
   ) \Big) (-\id + (D\widetilde{\eta})^{-1} ) \Big]
 \p_{r-1,\partial} 
\\
&  \leq C \p f \p_{r+\frac{3}{2}}
\p h \p_{r+\frac{1}{2},\partial}.
\end{split}
\label{quasi_linear_estimates_lower_1}
\end{align}
$1\leq r \leq s$.
Similarly,
(\ref{D_eta_inv_expansion_1}), 
(\ref{quasi_linear_estimates_psQ}), and (\ref{quasi_linear_estimates_psq}) give
\begin{align}
\begin{split}
& \p 
\Delta_\nu^{-1} \dive \Big[ D^2f L_1^{-1}P  \Big(\nabla \cH_{\widetilde{\eta}} (
\cS h +    \psQ \cS^{-1} h   +  \psq \cS^{-1} h 
   ) \Big) ( D\widetilde{\eta} )^{-1} \Big]
 \p_{r-1,\partial} 
\\
&  \leq C \p f \p_{r+\frac{3}{2}}
\p h \p_{r+\frac{1}{2},\partial},
\end{split}
\nonumber
\end{align}
so that
\begin{align}
\begin{split}
 \p \cW_4 h \p_{r-1,\partial} 
\leq C \p f \p_{r+\frac{3}{2}}
\p h \p_{r+\frac{1}{2},\partial},
\end{split}
\label{quasi_linear_estimates_lower_2}
\end{align}
if $1\leq r \leq s$.

If $\kappa$ sufficiently large, (\ref{D_eta_inv_expansion_1}),
(\ref{quasi_linear_estimates_psQ}), (\ref{quasi_linear_estimates_psq}),
(\ref{quasi_linear_estimates_lower_1}),  (\ref{quasi_linear_estimates_lower_2}),
and the form of $A$, 
imply that $\widehat{A}$ is arbitrarily close to
\begin{gather}
\widehat{\cZ}_0 = \frac{1}{\sqrt{\kappa}} \cZ_0
\nonumber
\end{gather}
in $B(H^r_0(\partial \Om) \times H^r_0(\partial \Om), X)$, with $\frac{3}{2} \leq r \leq s + \frac{1}{2}$; thus in particular in 
$B(Y,X)$ and in $B( H^\frac{3}{2}(\partial \Om) \times
H^\frac{3}{2}(\partial \Om), X)$.  Indeed, 
(\ref{quasi_linear_estimates_psQ}), (\ref{quasi_linear_estimates_psq}),
(\ref{quasi_linear_estimates_lower_1}), and
 (\ref{quasi_linear_estimates_lower_2}) show that
 \begin{gather}
  \frac{1}{\sqrt{\kappa}} \sum_{i=0}^4 \cW_i
 \nonumber
 \end{gather}
is close to $\widehat{\cZ}_0$, and the remaining terms in $\widehat{A}$
are lower order operators multiplied by $\frac{1}{\sqrt{\kappa}}$.

Now we are ready to verify the several hypotheses of theorem \ref{theorem_evolution_operator}. We start by first showing that
theorem \ref{theorem_generator_semi_group} can be applied.
By the operator $A$, we in fact mean its maximal operator, so 
that $A$ is in fact closed in $X$. We shall not, however, make notational
distinctions between $A$ and its extension (see \cite{K_1} for details). 

We start by showing that $\widehat{A}$ generates a semi-group.
Notice that $\widehat{\cZ}_0 + \la$, $\la \in \RR$, is invertible (with bounded
inverse), and so is $\widehat{A} + \la$. Write
\begin{gather}
\widehat{A} = (\widehat{A} - \widehat{\cZ}_0 ) + \widehat{\cZ}_0 
\equiv
\cE + \widehat{\cZ}_0.
\label{cE_definition}
\end{gather}
From the above estimates, we know that 
\begin{gather}
\p \cE z \p_{0,\partial} \leq C(K_0) \varepsilon \p z \p_{\frac{3}{2},\partial},
\label{cE_small}
\end{gather}
where $\varepsilon$ can be as small as we want provided that $\kappa$
is sufficiently large. Set 
\begin{gather}
w = (\widehat{A} + \la ) z,
\label{w_definition}
\end{gather}
and 
use the fact that $\widehat{\cZ}_0$ is skew-symmetric to obtain
\begin{align}
\begin{split}
(w, z)_{0,\partial} & = (\cE z, z)_{0,\partial} 
+ |\la| \p z \p_{0,\partial}^2, 
\end{split}
\nonumber
\end{align}
from which we get (using the Cauchy-Schwarz inequality and (\ref{cE_small}))
\begin{align}
\begin{split}
 |\la| \p z \p_{0,\partial}^2 & \leq
 \p z \p_{0,\partial} \p w \p_{0,\partial}
 +  \p \cE z \p_{0,\partial} \p  z \p_{0,\partial} \\
& \leq
 \p z \p_{0,\partial} \p w \p_{0,\partial}
 + C(K_0) \varepsilon \p z \p_{\frac{3}{2},\partial} \p  z \p_{0,\partial}.
\end{split}
\label{proof_semigroup_1}
\end{align}
Next, we show that
\begin{gather}
\p z \p_{\frac{3}{2},\partial} \leq C (|\la| \p z \p_{0,\partial} +
\p w \p_{0, \partial} ).
\label{proof_semigroup_2}
\end{gather}
Recalling the definition of $\widehat{\cZ}_0$, (\ref{cE_definition}) and (\ref{w_definition}), we
can write $w$ more explicitly as
\begin{gather}
\begin{pmatrix}
w_1 \\
w_2
\end{pmatrix}
= 
\begin{pmatrix}
\la & -\cS  \\
\cS + \cE_{21} & \la + \cE_{22}
\end{pmatrix}
\begin{pmatrix}
z_1 \\
z_2
\end{pmatrix}
.
\label{w_z_explicit}
\end{gather}
From the first row of (\ref{w_z_explicit}),
\begin{gather}
z_2 = \la \cS^{-1} z_1 - \cS^{-1} w_1 
\nonumber
\end{gather}
which implies
\begin{gather}
\p z_2 \p_{\frac{3}{2},\partial} \leq
+ C (|\la| \p z_1 \p_{0,\partial} + \p w_1 \p_{0,\partial} ).
\label{proof_semigroup_2_a}
\end{gather}
The second row of (\ref{w_z_explicit}) gives
\begin{gather}
z_1 = -\cS^{-1} \cE_{21} z_1 - \la \cS^{-1} z_2 - \cS^{-1} \cE_{22} z_2
+ \cS^{-1} w_2,
\nonumber
\end{gather}
which gives (with the help of (\ref{cE_small}) and (\ref{proof_semigroup_2_a})),
\begin{gather}
\p z_1 \p_{\frac{3}{2},\partial} \leq  C(K_0)\varepsilon
\p z_1 \p_{\frac{3}{2},\partial} + C |\la| \p z_2 \p_{0,\partial}
+ C (|\la| \p z_1 \p_{0,\partial} + \p w_1 \p_{0,\partial}  
+ \p w_2 \p_{0,\partial}).
\label{proof_semigroup_2_b}
\end{gather}
The term $ C(K_0) \varepsilon \p z_1 \p_{\frac{3}{2},\partial}$ can  be absorbed 
on the left-hand side by making $\varepsilon$ small, and then
(\ref{proof_semigroup_2_a}) and (\ref{proof_semigroup_2_b}) combine 
to give
(\ref{proof_semigroup_2}). 

From (\ref{proof_semigroup_1}) and (\ref{proof_semigroup_2}) it now follows that
\begin{align}
\begin{split}
 |\la| \p z \p_{0,\partial}^2 
& \leq
   C(K_0) (1 + \varepsilon) \p z \p_{0,\partial} \p w \p_{0,\partial}
 + C(K_0) \varepsilon|\la|  \p  z \p_{0,\partial}^2 .
\end{split} 
\nonumber
\end{align}
Thus, if $\varepsilon$ is sufficiently small, we find
$\p z \p_{0,\partial} \leq \frac{C(K_0)}{|\la|} \p w \p_{0,\partial}$, or, 
since $z = (\widehat{A} + \la)^{-1} w$,
\begin{gather}
\p (\widehat{A} + \la )^{-1} w \p_{0,\partial} \leq\frac{C(K_0)}{|\la|} 
\p w \p_{0,\partial}.
\label{A_la_inverse_bound}
\end{gather}
From theorem \ref{theorem_generator_semi_group}, we conclude that the family
$\widehat{A} = \widehat{A}(t)$ generates a continuous semi-group, and that 
(i) of theorem \ref{theorem_evolution_operator} is satisfied.

We now consider (ii) of theorem \ref{theorem_evolution_operator}.
As shown above, $\widehat{A}$ is close to $\widehat{\cZ}_0$ if $\kappa$
is large. Thus, $\La = (\widehat{A}(t))^{\frac{2s+1}{3}}$ gives an isomorphism
from $Y$ to $X$, and $\La \widehat{A}(t) \La^{-1}$ generates a semi-group on 
$X$. This implies, by the results of \cite{K_1}, that $e^{\tau \widehat{A}(t)}$
restricts to a semi-group on $Y$. Moreover, setting
\begin{gather}
\prescript{}{t}{(z_1,z_2)} = 
( (\widehat{A}(t))^{\frac{2s+1}{3}} z_1,
(\widehat{A}(t))^{\frac{2s+1}{3}} z_2 )_{0,\partial},
\nonumber
\end{gather}
one obtains an inner product that generates the topology of $Y$. Then,
an argument similar to the one leading to (\ref{A_la_inverse_bound}) gives
\begin{gather}
\p (\widehat{A} + \la )^{-1} \p_{\operatorname{Op}(t)} \leq\frac{C(K_0)}{|\la|}.
\nonumber
\end{gather}
From this, the remaining conditions of (ii) follow. 
From the continuity of $t \mapsto f(t)$ in the $H^{s+1}(\partial \Omega)$ norm
(remark \ref{remark_continuity_future}), 
the assumptions on $v$ and $\widetilde{q}_0$ and 
and the form of $A$, we obtain the continuity of $t\mapsto \widehat{A}(t)$
in $B(Y,X)$.
The remaining hypotheses
of theorem \ref{theorem_evolution_operator} are routinely checked,
and therefore we obtain the desired evolution operator for
$\widehat{A}$, and hence for $A$ as well.

Invoking theorem \ref{theorem_abstract_ODE} and given 
an initial condition, we obtain a solution to (\ref{first_order_simple}), namely
\begin{gather}
z \in C^0([0,T],Y) \cap C^1([0,T], X)
\nonumber
\end{gather}
 It remains to verify that 
$t \mapsto A(t)z(t)$ is continuous with respect to the
$H^{s-1}$ norm. Because $z \in C^0([0,T],Y)$, we have
$\left. D^3 \cS^{-1} z \right|_{\partial \Om} \in C^0([0,T], H_0^{s-1}(\partial\Om) 
\times H_0^{s-1}(\partial \Om))$. The result now follows since
the coefficients of $A$ depend on at most two derivatives of $f$, and are
therefore continuous in the $H^{s-1}(\partial \Om)$ norm because 
$f \in \cM$ (see again  remark \ref{remark_continuity_future}).
\end{proof}

\begin{proposition}
In proposition \ref{prop_first_order}, assume further that
\begin{gather}
v \in C^0\left([0,T], H^s(\Om, \RR^3)\right) \cap C^1([0,T], H^{s-\frac{3}{2}}(\Om, \RR^3)), 
\nonumber
\\
\widetilde{q}_0 \in C^0 \left([0,T],H^{s+1}(\Om,\RR)\right) \cap C^1 ([0,T],H^{s-\frac{1}{2}}(\Om,\RR)),
\nonumber
\end{gather}
and
\begin{gather}
f \in C^0([0,T], H_0^{s+2}(\partial \Om, \RR)) \cap 
C^1([0,T], H_0^{s+\frac{1}{2}}( \partial \Om, \RR))
\cap 
C^2([0,T], H_0^{s-1}( \partial \Om, \RR))
\nonumber
\end{gather}
satisfies
\begin{gather}
\p \dot{f} \p_{s+\frac{1}{2},\partial} \leq K_0^\prime,
\nonumber
\end{gather}
for some constant $K_0^\prime$. Then the solution $z$ satisfies the estimate
\begin{align}
\begin{split}
\p z(t) \p_{s+\frac{1}{2},\partial} & \leq 
K_1 \p z(0) \p_{s+\frac{1}{2},\partial} + 
\sup_{0\leq \tau \leq t} \frac{K_1(1+t)}{\sqrt{\kappa}} \p \cG(\tau) \p_{s-1,\partial}
\\
&+ 
\sup_{0\leq \tau \leq t} \frac{K_1(1+t)}{\sqrt{\kappa}} \p \partial_\tau\cG(\tau) \p_{s-1,\partial},
\end{split}
\label{z_estimate}
\end{align}
for a constant $K_1$ given by 
$K_1 = \sup_{0\leq \tau \leq T} \cP(\tau)$, with $\cP$ a 
continuous function of
\begin{gather}
\p v(\tau) \p_s, \p \dot{v} (\tau) \p_{s-\frac{3}{2}},
\p \widetilde{q}_0 (\tau) \p_{s+1}, 
\p \dot{\widetilde{q}}_0 (\tau) \p_{s-\frac{1}{2}},
T, K_0, \text{ and }  K_0^\prime.
\nonumber
\end{gather}
\label{prop_first_order_estimate}
\end{proposition}
\begin{proof}
It follows from the calculations in the proof of
proposition \ref{prop_first_order} that 
\begin{gather}
A^{-1}_\kappa = \frac{1}{\sqrt{\kappa}} \widehat{A}_\kappa^{-1},
\nonumber
\end{gather}
and that the norm of $A^{-1}_\kappa$ is bounded by $\frac{C(K_0)}{\sqrt{\kappa}}$ for large $\kappa$. Hence,
\begin{gather}
\p A_\kappa^{-1} w \p_{s+\frac{1}{2},\partial} \leq \frac{C(K_0)}{\sqrt{\kappa}} \p w \p_{s-1,\partial}.
\label{estimate_A_kappa_inverse}
\end{gather}
Next, invoking Duhamel's principle, we have
\begin{align}
\begin{split}
z(t) & = \cU(t,0) z(0) + \int_0^t \cU(t,\tau) \cG(\tau)  \, d\tau \\
& = \cU(t,0) z(0) + \int_0^t \cU(t,\tau) A_\kappa(\tau) (A_\kappa(\tau))^{-1} \cG(\tau) 
\, d\tau  \\
& = \cU(t,0) z(0) + \int_0^t \partial_\tau \cU(t,\tau) (A_\kappa(\tau))^{-1} \cG(\tau) \, d\tau   \\
& = \cU(t,0) z(0) +  (A_\kappa(t))^{-1} \cG(t) -  \cU(t,0)(A_\kappa(0))^{-1} \cG(0) \\
& + \int_0^t \cU(t,\tau)  (A_\kappa(\tau))^{-1} (\partial_\tau A_\kappa(\tau) )  (A_\kappa(\tau))^{-1} \cG(\tau) \, d\tau - \int_0^t  (A_\kappa(\tau))^{-1} \partial_\tau \cG(\tau)\, d\tau,
\end{split}
\label{Duhamel}
\end{align}
where $\cU$ is the evolution operator for $-A_\kappa$, satisfying
\begin{gather}
\partial_\tau \cU(t,\tau) = \cU(t,\tau) A_\kappa(\tau)
\nonumber
\end{gather}
and $\cU(t,t) = I$ (identity operator), and we have integrated by parts and used
\begin{gather}
\partial_\tau (A_\kappa(\tau))^{-1} = - (A_\kappa(\tau))^{-1} \partial_\tau A_\kappa(\tau)
(A_\kappa(\tau))^{-1}.
\nonumber
\end{gather}
Differentiating $A_\kappa(t)$, and using the hypotheses on $f$, $\widetilde{q}_0$,
and $v$, one checks, with the help of (\ref{estimate_A_kappa_inverse}), that 
\begin{gather}
\p (A_\kappa(\tau))^{-1} \partial_\tau A_\kappa(\tau)
(A_\kappa(\tau))^{-1} w \p_{s+\frac{1}{2},\partial} \leq \frac{C(K_0,K_0^\prime)}{\sqrt{\kappa}} \p w \p_{s-1, \partial}.
\label{estimate_A_kappa_time_der}
\end{gather}
We remark that while the expression for $\partial_\tau A_\kappa(\tau)$ 
is rather cumbersome, and thus will not be given here, the estimate
(\ref{estimate_A_kappa_time_der}) follows essentially by counting derivatives, 
 with some of the less-regular terms handled via standard 
arguments. For instance, $\dot{v} \in H^{s-\frac{3}{2}}(\Om,\RR^3)$.
But $\dot{v}$ is divergence-free and tangent to the boundary,
thus $\dive (\nabla_v \dot{v} ) \in H^{s-\frac{5}{2}}(\Om)$ and
$\langle \nu , \nabla_v \dot{v} \rangle = - 
\langle \nabla_v \nu, \dot{v} \rangle
\in H^{s-2}(\partial \Om)$, so that $Q(\nabla_v \dot{v} ) 
\in H^{s-\frac{3}{2}}(\Om, \RR^3)$.
Thus 
$\left( \Delta_\nu^{-1} \dive\left( ( D^2 \cS^{-1} h ) Q(\nabla_v \dot{v}) \right)\right)
\in H^{s-\frac{1}{2}}(\Om)$, or
\begin{gather}
\left. \left( \Delta_\nu^{-1} \dive\left( ( D^2 \cS^{-1} h ) Q(\nabla_v \dot{v}) \right)\right)\right|_{\partial \Om} \in H^{s-1}(\partial \Om),
\nonumber
\end{gather}
 This
is needed in order that (\ref{estimate_A_kappa_time_der}) make sense
(see $\cW_9$ in the definition of $A_\kappa$).

Then combining (\ref{estimate_A_kappa_inverse}) and (\ref{estimate_A_kappa_time_der}) with 
(\ref{Duhamel}) gives
\begin{align}
\begin{split}
\p z(t) \p_{s+\frac{1}{2},\partial} & \leq 
C(K_0,K_0^\prime) \Big ( \p z(0) \p_{s+\frac{1}{2},\partial} + 
\sup_{0\leq \tau \leq t} \frac{(1+t)}{\sqrt{\kappa}} \p \cG(\tau) \p_{s-1,\partial}
\\
&+ 
\sup_{0\leq \tau \leq t} \frac{(1+t)}{\sqrt{\kappa}} \p \partial_\tau\cG(\tau) \p_{s-1,\partial} \Big ).
\end{split}\nonumber
\end{align}
Note that the constant $C(K_0,K_0^\prime)$ 
appearing in the above
estimate has the desired form (see notation \ref{notation_constants}). 
This finishes the proof.
\end{proof}

\begin{proposition}
Let $h_1 \in H_0^{s+\frac{1}{2}}(\partial \Om, \RR)$,
and assume the hypotheses of  proposition \ref{prop_first_order}.
Then, if $\kappa$ is sufficiently large, there exists
a unique solution
\begin{gather}
h \in C^0([0,T], H_0^{s+2}(\partial \Om, \RR)) \cap 
C^1([0,T], H_0^{s+\frac{1}{2}}(\partial \Om, \RR))
\cap 
C^2([0,T], H_0^{s-1}(\partial \Om, \RR))
\nonumber
\end{gather}
of (\ref{eq_h_bry}), satisfying $h(0) = 0$, and
$\dot{h}(0) = h_1$,  where  
where $h$ is extended to $\Om$ via 
its $\psi$-harmonic extension. If in addition the hypotheses 
of proposition \ref{prop_first_order_estimate} hold,
then   
there exists a constant $K_2>0$,
 such that, 
\begin{align}
\begin{split}
\p h(t) \p_{s+2,\partial} & \leq 
\frac{K_2}{\sqrt{\kappa}} \p \dot{h}(0) \p_{s+\frac{1}{2},\partial} 
+  
 \frac{K_2(1+T)}{\kappa},
 \end{split}
\nonumber
\end{align}
 \begin{align}
\begin{split}
\p \dot{h} \p_{s+\frac{1}{2},\partial} & \leq 
K_2\p \dot{h}(0) \p_{s+\frac{1}{2},\partial} 
+  
\frac{K_2(1+T)}{\sqrt{\kappa}},
 \end{split}
\nonumber
\end{align}
and $\p \ddot{ h } \p_{s-1,\partial} \leq K_2$,
$0 \leq t \leq T$, where $K_2 = \sup_{0\leq \tau \leq T} \cP(\tau) $,
with $\cP$ a continuous function of  
\begin{gather}
\p v (\tau) \p_s, 
\p \dot{v} (\tau) \p_{s-\frac{3}{2}},
 \p \widetilde{q}_0(\tau)  \p_{s+1}, 
 \p \dot{\widetilde{q}}_0 (\tau) \p_{s-\frac{1}{2}}, T,
K_0, \text{ and } K_0^\prime.
\nonumber
\end{gather}
\label{prop_linear_eq_h}
\end{proposition}
\begin{proof}
The argument
is more or less standard, so we shall go over it briefly.
Let $\cG$ be given by (\ref{cG_definition}).
Consider the solution to (\ref{first_order_simple}) with initial condition
$z(0) = (0, \dot{h}(0))$. Define $h$ as the solution of
\begin{gather} 
\partial_t h = z_2,
\nonumber
\end{gather}
with initial condition $h(0) = 0$. Then
\begin{align}
\begin{split}
\sqrt{\kappa} \cS h (t) & = \sqrt{\kappa} \int_0^t \partial_t \cS h(\tau) \, d\tau 
\\
& = \sqrt{\kappa} \int_0^t \cS z_2(\tau) \, d\tau \\
& = z_1(t),
\end{split}
\nonumber
\end{align}
after using the first line of (\ref{first_order_simple}), i.e., 
$\partial_t z_1 = \sqrt{\kappa}\cS z_2$. By inspection we see that 
$h = \frac{1}{\sqrt{\kappa}} \cS^{-1} z_1$ solves (\ref{eq_h_bry})
and has the correct regularity.

From the properties of $\cS$ that follow from lemma 
\ref{lemma_elliptic_pseudo_2},
\begin{gather}
\frac{\sqrt{\kappa}}{C} \p h \p_{s+2,\partial} \leq \p \sqrt{\kappa} \cS h \p_{s+\frac{1}{2}} 
\leq C\sqrt{\kappa}\p h \p_{s+2,\partial},
\label{elliptic_norm_equivalent}
\end{gather} 
which combined with (\ref{z_estimate}) and the definition of $z$ gives
\begin{align}
\begin{split}
\p h \p_{s+2,\partial} & \leq 
C(K_1) \Big (
\frac{1}{\sqrt{\kappa}} \p \dot{h}(0) \p_{s+\frac{1}{2},\partial} 
+  
\sup_{0\leq \tau \leq t} \frac{(1+t)}{\kappa} \p \cG(\tau) \p_{s-1,\partial}
\\
&
+ 
\sup_{0\leq \tau \leq t} \frac{(1+t)}{\kappa} \p \partial_\tau\cG(\tau) \p_{s-1,\partial} \Big),
\end{split}
\nonumber
\end{align}
and
\begin{align}
\begin{split}
\p \dot{h} \p_{s+\frac{1}{2},\partial} & \leq 
C(K_1) \Big( \p \dot{h}(0) \p_{s+\frac{1}{2},\partial} 
+  
\sup_{0\leq \tau \leq t} \frac{(1+t)}{\sqrt{\kappa}} \p \cG(\tau) \p_{s-1,\partial}
\\
&+ 
\sup_{0\leq \tau \leq t} \frac{(1+t)}{\sqrt{\kappa}} \p \partial_\tau\cG(\tau) \p_{s-1,\partial} \Big ),
\end{split}
\nonumber
\end{align}
where $K_1$ is the constant appearing in the conclusion 
of proposition \ref{prop_first_order_estimate},
and we have used $h(0) = 0$.
The estimate for $\ddot{h}$ now follows by solving for
$\ddot{h}$ in equation (\ref{eq_h_bry}) and using the above 
estimates for $h$ and $\dot{h}$.
We then invoke  the definition of $\cG$ to obtain the result.
Here, we need to check that $\cG$ indeed has the correct regularity.
This follows from the fact that $v$ and $\dot{v}$ are divergence 
free and tangent to the boundary, 
so that
$Q(\nabla_v v) \in H^s(\Om,\RR^3)$
and
$Q(\nabla_v \dot{v}) \in H^{s-\frac{3}{2}}(\Om, \RR^3)$
(see similar discussion right after (\ref{estimate_A_kappa_time_der})).
\end{proof}
The main result  of this section, theorem \ref{theo_f_eq} below,
uses the Schauder fixed point theorem and the Arzel\`a-Ascoli theorem,
which we state for
the reader's convenience. 

\begin{theorem} (Schauder fixed point)
Let $S$ be a closed convex set in a Banach space, and let $T$ be a continuous 
mapping of $S$ into itself such that the image $T(S)$ is pre-compact. Then
$T$ has a fixed point.
\label{Schauder}
\end{theorem}
\begin{proof} \cite{Schwartz} or \cite{GT}.
\end{proof}

\begin{theorem} (Arzel\`a-Ascoli) Let $X$ be a topological space and $(Y,d)$ a metric space.
Consider $C^0(X,Y)$ with the topology of compact convergence. Let $\cF$ be a subset
of $C^0(X,Y)$. If $\cF$ is equi-continuous under $d$ and the set
\begin{gather}
\cF_x = \{ f(x) \, | \, f \in \cF \}
\nonumber
\end{gather}
has compact closure for each $x \in X$, then $\cF$ is contained in a compact subspace of 
$C^0(X,Y)$.
\end{theorem}
\begin{proof} \cite{Munkres}.
\end{proof}

\begin{theorem}
Assume that 
\begin{gather}
v \in C^0\left([0,T], H^s(\Om, \RR^3)\right)
\cap
C^1([0,T], H^{s-\frac{3}{2}}(\Om, \RR^3))
, \, \dive(v) = 0,
\nonumber
\end{gather}
and
\begin{gather}
\widetilde{q}_0 \in 
C^0 \left([0,T],H^{s+1}(\Om,\RR)\right)
\cap
C^1 ([0,T],H^{s-\frac{1}{2}}(\Om,\RR)),
\nonumber
\end{gather}
$s > \frac{3}{2} + 2$, $T >0$. Let 
$f_1 \in H^{s+\frac{1}{2}}_0(\partial \Om)$ satisfy
\begin{gather}
\p f_1 \p_{s+\frac{1}{2},\partial} \leq \frac{K_3}{\sqrt{\kappa}},
\label{f_1_hypothesis}
\end{gather}
for some constant $K_3$. Finally, let $\cG$ be given by (\ref{cG_definition}).
Then, if $\kappa$ is sufficiently large, there exists a $T^\prime \in (0,T]$
and a solution 
\begin{gather}
f \in C^0([0,T^\prime], H_0^{s+2}(\partial \Om, \RR)) \cap 
C^1([0,T^\prime], H_0^{s+\frac{1}{2}}( \partial \Om, \RR))
\cap 
C^2([0,T^\prime], H_0^{s-1}( \partial \Om, \RR)),
\nonumber
\end{gather}
satisfying (\ref{eq_f_bry}) with initial conditions
$f(0) = 0$, and $\dot{f}(0) = f_1$. In (\ref{eq_f_bry}), it is
understood that $f$ is extended $\psi$-harmonically to $\Om$. Furthermore, $f$
obeys the estimate
\begin{gather}
\p f \p_{s+2,\partial} \leq \frac{K_4}{\kappa},
\nonumber
\end{gather}
\begin{gather}
\p \dot{f} \p_{s+\frac{1}{2},\partial} \leq \frac{K_4}{\sqrt{\kappa}},
\nonumber
\end{gather}
and
\begin{gather}
\p \ddot{f} \p_{s-1,\partial } \leq K_4.
\nonumber
\end{gather}
The constant $K_4$ is given by
  $K_4 = \sup_{0\leq \tau \leq T} \cP(\tau) $,
with $\cP$ a continuous function of  
\begin{gather}
\p v (\tau) \p_s, 
\p \dot{v} (\tau) \p_{s-\frac{3}{2}},
 \p \widetilde{q}_0(\tau)  \p_{s+1}, 
 \p \dot{\widetilde{q}}_0 (\tau) \p_{s-\frac{1}{2}},
 T, \text{ and } K_3.
\nonumber
\end{gather}
\label{theo_f_eq}
\end{theorem}
\begin{proof}
Consider the set $\cM$ from notation \ref{cM_notation} with the metric
\begin{gather}
d(f,g) = \sup_{0 \leq t \leq T} \p f(t) - g(t) \p_{0,\partial} .
\nonumber
\end{gather}
We start by noticing that $\cM$ is a complete metric space in the metric 
$d$. Indeed, if $f_n \rar f$ in the metric $d$, then, from the facts that
$B_R(0) \subset H^{s+2}(\partial \Om)$ is weakly compact where $B_R(0)$ denotes the ball about zero of radius $R$ in the metric $d$ and that the 
embedding 
$H^{s+2}(\partial \Om) \subset H^0(\partial \Om)$ is compact, we conclude
$\p f \p_{s+2,\partial} \leq R$. By the same reasoning we also find that,
for each $t$, $f(t) \in H^{s+\frac{1}{2}}(\partial \Om)$, 
and $f_n(t) \rar f(t)$ 
in $H^{s+\frac{1}{2}}(\partial \Om)$ since $H^{s+2}(\partial \Om) 
\subset H^{s+\frac{1}{2}}(\partial \Om)$ compactly. But,
\begin{align}
\begin{split}
\p f(t) - f(t^\prime) \p_{s+\frac{1}{2},\partial}
& \leq 
\p f(t) - f_n(t) \p_{s+\frac{1}{2},\partial}
+ 
\p f_n(t) - f_n(t^\prime) \p_{s+\frac{1}{2},\partial}
+
 \p f_n(t^\prime) - f(t^\prime) \p_{s+\frac{1}{2},\partial}
\\
& 
\leq \p f(t) - f_n(t) \p_{s+\frac{1}{2},\partial}
+ 
\ell |t - t^\prime|
+
\p f_n(t^\prime) - f(t^\prime) \p_{s+\frac{1}{2},\partial}.
\end{split}
\nonumber
\end{align}
Passing to the limit we conclude that
$\p f(t) - f(t^\prime) \p_{s+\frac{1}{2},\partial} \leq \ell |t - t^\prime|$,
thus $f \in \cM$. Therefore, $\cM$ is a closed subset of the Banach space
$L^\infty([0,T],H^0_0(\partial \Om))$. One immediately checks that
$\cM$ is convex.

Given $f \in \cM$ satisfying
\begin{gather}
\p f \p_{s+2,\partial} \leq \frac{K_0}{\sqrt{\kappa}},
\label{hypothesis_f_proof}
\end{gather}
proposition \ref{prop_linear_eq_h} yields
a solution $h$ to (\ref{eq_h_bry}) if $\kappa$ is sufficiently large.
We shall show that, 
if $T$ is small, $\kappa$ large, and $R$ and $\ell$ are suitably chosen, 
the association $f \mapsto h$ 
defines a map $\Phi: \cM \rar L^\infty([0,T],H^0_0(\partial \Om))$,
which (i) takes $\cM$ into itself, (ii) is continuous, and
(iii) has image $\Phi(\cM)$ pre-compact in 
$L^\infty([0,T],H^0_0(\partial \Om))$.

To fix the constant $K_0$ in (\ref{hypothesis_f_proof}), 
we put $K_0 = K_3$.  
Choose $R = R(K_0,\kappa) = \frac{K_0}{\sqrt{\kappa}}$,
so  that (\ref{hypothesis_f_proof}) holds
for $f \in \cM$. Thus, in light of proposition \ref{prop_linear_eq_h},
we obtain a map
\begin{gather}
\Phi: \cM \rar C^0([0,T], H_0^{s+2}(\partial \Om)) \cap 
C^1([0,T], H_0^{s+\frac{1}{2}}(\partial \Om))
\cap 
C^2([0,T], H_0^{s-1}(\partial \Om)),
\nonumber
\end{gather}
given by $\Phi(f) = h$, where $h$ is the solution of (\ref{eq_h_bry})
with initial conditions $h(0) = 0$, $\dot{h}(0) = f_1$. 
Notice that $\Phi$ can be viewed as a map from $\cM$ to 
$L^\infty([0,T],H^0_0(\partial \Om))$.

Letting $z = (\sqrt{\kappa} \cS h, \dot{h})$, where $\cS$ is given
by  (\ref{cS_defi}), $z$ satisfies (\ref{first_order_simple}), and is
given by
\begin{align}
\begin{split}
z(t) & = \cU(t,0) z(0) + \int_0^t \cU(t,\tau) \cG(\tau)  \, d\tau,
\end{split}
\label{z_U_cG}
\end{align}
where $z(0) = (0, f_1)$, $\cU$ is the evolution operator associated
with $-A_\kappa(t)$ (see proposition \ref{prop_first_order}), and
 $\cG$ is given by (\ref{cG_definition}). From this
and (\ref{elliptic_norm_equivalent}),
\begin{align}
\begin{split}
\p h \p_{s+2,\partial} & \leq \frac{C(K_0)}{\sqrt{\kappa}}
 \p f_1 \p_{s+\frac{1}{2},\partial}
+ \frac{C(K_0) T}{\sqrt{\kappa}} \sup_{0\leq \tau \leq T} \p \cG( \tau ) \p_{s+\frac{1}{2},\partial} \\
& \leq \frac{C(K_0)  K_3}{\kappa}
+ \frac{C(K_0) T}{\sqrt{\kappa}} \sup_{0\leq \tau \leq T} \p \cG( \tau ) \p_{s+\frac{1}{2},\partial},
\end{split}
\label{estimate_using_dot_f_small}
\end{align}
using (\ref{f_1_hypothesis}). We note that $C$ does depend
 on $K_0$ because of the estimates for $A_\kappa(t)$
in the proof of proposition \ref{prop_first_order}.
Choose $T$  small  and $\kappa$
large
so that
\begin{gather}
\frac{C(K_0) T}{\sqrt{\kappa}} \sup_{0\leq \tau \leq T} \p \cG( \tau ) \p_{s+\frac{1}{2},\partial} 
< \frac{R}{2},
\nonumber
\end{gather}
and
\begin{gather}
 \frac{C(K_0) K_3}{\kappa} < \frac{R}{2}.
\nonumber
\end{gather}
are satisfied.

Note that this last inequality is possible despite the dependence 
of $R$ on $K_0$ and $\kappa$ 
because $R$ is of the order of $\frac{1}{\sqrt{\kappa}}$, while the left hand side of the 
last inequality is of order $\frac{1}{\kappa}$. We conclude that
$\p h \p_{s+2,\partial} \leq R$. (\ref{z_U_cG}) also gives
\begin{align}
\begin{split}
\p \dot{h} \p_{s+\frac{1}{2},\partial} & \leq C(K_0) 
 \p f_1 \p_{s+\frac{1}{2},\partial}
+ C(K_0) T \sup_{0\leq \tau \leq T} \p \cG( \tau ) \p_{s+\frac{1}{2},\partial} \\
& \leq \frac{C(K_0) K_3}{\sqrt{\kappa}}
+C(K_0) T \sup_{0\leq \tau \leq T} \p \cG( \tau ) \p_{s+\frac{1}{2},\partial}.
\end{split}
\label{estimate_h_dot_cM}
\end{align}
This implies 
$\p h(t) - h(t^\prime) \p_{s+\frac{1}{2},\partial} \leq\ell |t - t^\prime|$
if we choose $\ell$ large enough as to have
the right side of (\ref{estimate_h_dot_cM}) less than $\ell$. 
We conclude that $h \in \cM$, i.e., 
$\Phi$ maps $\cM$ into itself.

Next, we study the continuity of $\Phi$. 
This requires estimating $h - \widetilde{h} \equiv \Phi(f) - \Phi(\widetilde{f})$,
where $f,\widetilde{f} \in \cM$.
Let $z$ be as above, and 
$\widetilde{z} = (\sqrt{\kappa} \cS \widetilde{h}, \dot{\widetilde{h}} )$,
with $-\widetilde{A}_\kappa$ the operator in (\ref{first_order_simple})
with $\widetilde{f}$ in place of $f$, and $\widetilde{\cU}$ the corresponding
evolution operator. We have the estimate:
\begin{align}
\begin{split}
\p z - \widetilde{z} \p_{0,\partial}  & \leq
C(K_0) \int_0^T 
 \p (A(t) - \widetilde{A}(t) ) z(t) \p_{0,\partial} \, dt.
\end{split}
\label{perturbation}
\end{align}
We recall how (\ref{perturbation}) is obtained. Computing we find
\begin{align}
\begin{split}
\partial_s (\widetilde{\cU}(t,s) \cU(s,r) y) & = 
\partial_s \widetilde{\cU}(t,s) \cU(s,r) y + \widetilde{\cU}(t,s)
\partial_s \cU(s,r) y \\
& =  \widetilde{\cU}(t,s) \widetilde{A}(s) \cU(s,r) y -
\widetilde{\cU}(t,s) A(s) U(s,r) y.
\end{split}
\nonumber
\end{align}
Then integrating between $r$ and $t$, we get
\begin{gather}
\widetilde{\cU}(t,r)y - \cU(t,r)y 
= \int_r^t \widetilde{\cU}(t,s) (\widetilde{A}(s) - A(s) ) \cU(s,r)y \, ds.
\label{estimate_Kato_II}
\end{gather}
Setting $r=0$ and $y = z(0)$ in (\ref{estimate_Kato_II}) gives
\begin{gather}
\widetilde{\cU}(t,0)z(0) - \cU(t,0)z(0) = 
\int_0^t \widetilde{\cU}(t,s) (\widetilde{A}(s) - A(s) ) \cU(s,0) z(0)\, ds,
\label{estimate_Kato_II_a}
\end{gather}
Setting $y= \cG(r)$ in (\ref{estimate_Kato_II}) 
and integrating in $r$ yields
\begin{align}
\begin{split}
\int_0^t ( \widetilde{\cU}(t,r) - \cU(t,r))\cG(r) \, dr 
& = \int_0^t \int_r^t \widetilde{\cU}(t,s) (\widetilde{A}(s) - A(s) ) \cU(s,r) \cG(r) \, ds \, dr 
\\
&
=  \int_0^t  \widetilde{\cU}(t,s) (\widetilde{A}(s) - A(s) ) \int_0^s \cU(s,r) \cG(r)  \, dr \, ds.
\end{split}
\label{estimate_Kato_II_b}
\end{align}
(\ref{estimate_Kato_II_a}) and (\ref{estimate_Kato_II_b})  imply
(\ref{perturbation}); see \cite{Kato_hyp_2} for details.

We need to estimate the difference 
$(\widetilde{A}(t) - A(t)) z(t)$ in (\ref{perturbation}). 
For this, we point out that the operators $\cZ_0$, $\cZ_5$, $\cZ_7$,
and
$\cZ_9$ that figure in the definition of $A$ 
do not depend on $f$, being therefore  the same for $\widetilde{A}(t)$ 
and $A(t)$. Hence, they cancel out in the  difference
$\widetilde{A}(t) - A(t)$.

Continuing we have
\begin{align}
\begin{split}
\p ( \widetilde{A}(t) - A(t) )z \p^2_{0,\partial} 
& =
(  (\widetilde{A}(t) - A(t) )z, (\widetilde{A}(t) - A(t))z )_{0, \partial}
\\
&
=( (\widetilde{A}(t) - A(t))^* ( \widetilde{A}(t) - A(t) )z, z )_{0, \partial}
\\
& \leq 
\p (\widetilde{A}(t) - A(t))^* ( \widetilde{A}(t) - A(t) )z \p_{-r,\partial}
\p z \p_{r, \partial}
\end{split}
\label{estimate_difference_As_prel}
\end{align}
where $(\widetilde{A}(t) - A(t))^*$ is the adjoint of 
$\widetilde{A}(t) - A(t) $ in $H^0_0(\partial \Om)$.
$\widetilde{A}^*$ and $A^*$
are pseudo-differential operators or order $\frac{3}{2}$ whose
coefficients depend on at most four derivatives of $f$, 
and
 in the last step we used the generalized Cauchy-Schwarz inequality
$(a,b)_0 \leq \p a \p_{-r} \p b \p_r$, with $\p \cdot \p_{-r}$
denoting the negative  Sobolev norm, and $r \leq s+\frac{1}{2}$ a number that
will be conveniently chosen.  

Since  $(\widetilde{A}(t) - A(t))^*$ involves at most four derivatives 
of $f$, $f \in H^{s+2}(\partial \Om)$,
and $s > \frac{n}{2} + 2 = \frac{n-1}{2} + \frac{5}{2}$ (so that 
$s -2 > \frac{n-1}{2}$ ), we can apply (\ref{bilinear}) to get
\begin{align}
\begin{split}
\p (\widetilde{A}(t) - A(t))^* ( \widetilde{A}(t) - A(t) )z \p_{-r,\partial} & \leq 
C \sqrt{\kappa} \p \partial^4 f\p_{s-2,\partial} 
\p ( \widetilde{A}(t) - A(t) )z \p_{-r+ \frac{3}{2},\partial} \\
& \leq \frac{C(K_0)}{\sqrt{\kappa}} \p ( \widetilde{A}(t) - A(t) )z \p_{-r+ \frac{3}{2},\partial},
\end{split}
\nonumber
\end{align}
where $\partial^4 f$ symbolically represents terms in at most four
derivatives of $f$, and where we used the fact that 
the $s+2$ norm of $f$ gives a 
$\frac{1}{\kappa}$ factor. Thus
(\ref{estimate_difference_As_prel}) reads
\begin{align}
\begin{split}
\p ( \widetilde{A}(t) - A(t) )z \p^2_{0,\partial} 
& 
\leq 
\frac{C(K_0)}{\sqrt{\kappa}} \p ( \widetilde{A}(t) - A(t) )z \p_{-r+\frac{3}{2},\partial}
\p z \p_{s+\frac{1}{2}, \partial}.
\end{split}
\label{estimate_difference_As}
\end{align}
Recall now how (\ref{first_order_simple}) was obtained from  (\ref{eq_h_bry}). Each term in (\ref{eq_h_bry}) is a pseudo-differential
operator of order at most three, whose coefficients depend on at most
second derivatives of $f$, although this dependence may be non-local
due to the presence of the operator $\Delta_\nu^{-1}\circ \dive$.
If $p(D)h$ is one of such pseudo-differential operators acting
on $h$, the corresponding term in (\ref{first_order_simple}) is
of the  form
\begin{gather}
\frac{1}{\sqrt{\kappa}} p(D)\circ \cS^{-1} (\sqrt{\kappa} \cS h) \equiv 
\frac{1}{\sqrt{\kappa}} p(D)\circ \cS^{-1} z_1,
\nonumber
\end{gather}
and if $p(D)\dot{h}$ is one of the operators that acts 
on $\dot{h}$, the corresponding term in (\ref{first_order_simple}) is
of the form
\begin{gather}
 p(D) \dot{h} \equiv 
 p(D) z_2.
\nonumber
\end{gather}
With these considerations in mind, we proceed to the estimates below.
Recalling  (\ref{psQ_linear}), we have
\begin{align}
\begin{split}
\p \sqrt{\kappa}{\mathcal{Q}^{(3)}(\partial^2 \widetilde{f}, \partial^3)} \cS^{-1} z_1
& - \sqrt{\kappa} \psQ \cS^{-1} z_1 \p_{-r+\frac{3}{2},\partial} 
\\
& = 
\sqrt{\kappa} \p 
( a^{\alpha ij}( D^2 \widetilde{f}, D\widetilde{f}, \widetilde{f} )
-
a^{\alpha ij}( D^2 f, D f,f ) )
\partial_{\alpha ij} \cS^{-1} z_1 \p_{-r+\frac{3}{2},\partial} 
\\
& 
\leq 
C \sqrt{\kappa} \p 
( a^{\alpha ij}( D^2 \widetilde{f}, D\widetilde{f}, \widetilde{f} )
-
a^{\alpha ij}( D^2 f, D f,f ) )
\p_{-r+\frac{3}{2},\partial} \p \partial_{\alpha ij} \cS^{-1} z_1  \p_{s-1,\partial}.
\end{split}
\nonumber
\end{align}
In the last inequality, we used the fact that $z_1 \in H^{s+\frac{1}{2}}_0(\partial \Om)$, 
so that $\partial_{\alpha ij} \cS^{-1} z_1$ belongs to $H^{s-1}_0(\partial \Om)$; 
we also used
the inequality $s > \frac{n}{2} + 2$ so that (\ref{bilinear}) is valid.
As the coefficients $a^{\alpha ij}$ are smooth functions
of their arguments, we conclude
\begin{align}
\begin{split}
\p \sqrt{\kappa}{\mathcal{Q}^{(3)}(\partial^2 \widetilde{f}, \partial^3)} \cS^{-1} z_1
& - \sqrt{\kappa} \psQ \cS^{-1} z_1 \p_{-r+\frac{3}{2},\partial} 
\\
&
\leq 
C(K_0) \sqrt{\kappa} \p  \widetilde{f} - f \p_{-r+\frac{7}{2},\partial}
 \p  z_1  \p_{s+\frac{1}{2},\partial}
\\
&
\leq C(K_0) \sqrt{\kappa} \p  \widetilde{f} - f\p_{-r+\frac{7}{2},\partial},
\end{split}
\nonumber
\end{align}
after using 
\begin{gather}
\p z_1 \p_{s+\frac{1}{2},\partial} \leq C
\label{z_1_bounded}
\end{gather} 
since $z_1 = \sqrt{\kappa} \cS h$, and  
$\p h \p_{s+2,\partial} \leq \frac{K_0}{\sqrt{\kappa}}$.

The other terms are similarly estimated so we get
\begin{align}
\begin{split}
\p ( \widetilde{A}(t) - A(t) )z \p_{-r+\frac{3}{2},\partial} 
\leq \sqrt{\kappa} C(K_0) \p  \widetilde{f} - f\p_{-r+\frac{7}{2},\partial}.
\end{split}
\label{estimate_difference_As_2}
\end{align} 
Then combining (\ref{estimate_difference_As}) and (\ref{estimate_difference_As_2}) we get
\begin{align}
\begin{split}
\p ( \widetilde{A}(t) - A(t) )z \p_{0,\partial} 
&
\leq  C(K_0) 
\sqrt{  \p  \widetilde{f} - f\p_{-r+\frac{7}{2},\partial}}.
\end{split}
\label{estimate_difference_As_3}
\end{align}
Here we have used $\p z \p_{s+\frac{1}{2},\partial} \leq C$, which suffices since
$z_1$ is bounded by (\ref{z_1_bounded}), 
and $\p z_2 \p_{s+\frac{1}{2},\partial}$ is bounded 
by (\ref{estimate_h_dot_cM}) and our choice of $\ell$.
We now choose $r=\frac{7}{2}$ (which is less than $s+\frac{1}{2}$ since 
$s > \frac{3}{2} +2$), so (\ref{estimate_difference_As_3}) gives
\begin{align}
\begin{split}
\p ( \widetilde{A}(t) - A(t) )z \p_{0,\partial} 
& 
\leq  C(K_0)  \sqrt{d(\widetilde{f},f)}.
\end{split}
\label{estimate_difference_As_4}
\end{align}
On the other hand, invoking (\ref{elliptic_norm_equivalent}) with
$s$ replaced by $-\frac{1}{2}$, and 
once again using  $z = (\sqrt{\kappa} \cS h, \dot{h})$,
we get
\begin{align}
\begin{split}
\p \widetilde{z} - z \p_{0,\partial}
& \geq C \sqrt{\kappa} \p \widetilde{h} - h \p_{\frac{3}{2},\partial}
\geq  C \sqrt{\kappa} \p \widetilde{h} - h \p_{0,\partial}.
\end{split}
\label{z_difference_h_difference}
\end{align}
Combining (\ref{perturbation}), (\ref{estimate_difference_As_4}), 
(\ref{z_difference_h_difference}), and recalling the definitions of $h$
and $\widetilde{h}$ we get
\begin{gather}
d(\Phi(\widetilde{f}),\Phi(f)) \leq \frac{C(K_0)}{\sqrt{\kappa}} T \sqrt{ d(\widetilde{f}, f) }.
\label{continuity_Phi}
\end{gather}
This establishes the continuity of the map $\Phi$. 

We now show the pre-compactness of $\Phi(\cM)$. 
Recall that  $f \in C^0([0,T],H^0_0(\partial \Om))$ if 
$f \in \cM$. 
Let $\{ f_\nu \}_{\nu \in \cI} \subset \cM$, $\cI$ an indexing set.
Invoking the Lipschitz condition once more,
\begin{gather}
\p f_\nu(t) - f_\nu(t^\prime) \p_{0,\partial}
\leq \p f_\nu (t) - f_\nu (t^\prime) \p_{s+\frac{1}{2},\partial}
\leq \ell |t - t^\prime|,
\nonumber
\end{gather}
we see that  $\{ f_\nu \}_{\nu \in \cI}$ is equi-continuous as
a family of maps from $[0,T]$ to $H^0_0(\partial \Om)$. Also, for each fixed 
$t$, the set 
\begin{gather}
\{ f_\nu(t) \, | \, \nu \in \cI \}
\nonumber
\end{gather}
has compact closure in the $H^0$-topology in view of the
compact embedding $H^{s+2}_0(\partial \Om) \subset H^0_0(\partial \Om)$
and the bound $\p f_\nu \p_{s+2,\partial} \leq R$.
Combined with the continuity of the functions $f_\nu$, we have thus
verified the conditions to apply the Arzel\`a-Ascoli theorem, and we 
conclude that $\{ f_\nu \}$ has compact closure in 
$C^0([0,T],H^0_0(\partial \Om))$,
and therefore in $\cM$ since $\cM$ is complete. The same then holds for
$\{ \Phi(f_\nu) \}_{\nu \in \cI}$ because of (\ref{continuity_Phi}), i.e., 
the continuity of  $\Phi: \cM \rar \cM$.
This shows that $\Phi(\cM)$ is pre-compact
in $L^\infty([0,T],H^0_0(\partial \Om))$.

We can now invoke theorem \ref{Schauder}, namely, 
the Schauder fixed point theorem, 
to conclude that $\Phi$ has a fixed point in $\cM$, i.e.,
there exists $f_* \in \cM$ such that $\Phi(f_*) = f_*$. This $f_*$ solves
(\ref{eq_f_bry}) with initial conditions $f_*(0) = 0$ and $\dot{f}_*(0) = f_1$.
In view of our choices, we see that 
$f_*$ and $\dot{f}_*$
satisfy the further hypotheses on $f$ and $\dot{f}$ in 
proposition \ref{prop_first_order_estimate} with $K_0 = K_3$ and
$K_0^\prime = \ell$. Therefore, by proposition \ref{prop_linear_eq_h},  we obtain
the desired estimates for $f_*$, $\dot{f}_*$, and $\ddot{f}_*$. 
\end{proof}
 
\subsection{Analysis of $f$ in the interior}
The results of section \ref{section_f_boundary} give us a solution
to (\ref{eq_f_bry}) with appropriate initial conditions. Since 
the operators in (\ref{eq_f_bry}) involve the $\psi$-harmonic extension
(\ref{nl_Dirichlet_mod_f}), for $f$ sufficiently small (i.e., $\kappa$
large), the extension of $f$ to $\Om$ also satisfies (\ref{system_f_nl_Dirichlet}).
We remind the reader that 
(\ref{eq_f_bry}) is (\ref{system_f_v_f_dot_dot}) without the gradient,
or (\ref{new_form_f_eq}),
restricted to the boundary.

As before we can drop the gradient in front
of every term in (\ref{system_f_v_f_dot_dot}) and work modulo constants. This
leads to an evolution equation for $f$ of the form
\begin{align}
\begin{cases}
\ddot{f} + \ccA(v, f)(f) + \ccB( v, f )(\dot{f}) + \ccC( v, \widetilde{q})
= 0
 \,\,\text{ in } \, \,
 \Om, \\
 f(0) = 0, \, \dot{f} = f_1,
 \end{cases} 
\label{f_interior_before_restriction}
\end{align}
where $\ccA$ is a third order pseudo-differential with coefficients depending 
on $v$ and $f$, $\ccB$ is
first order with coefficients depending on $v$ and $f$, and 
$\ccC$ is a lower order operator on $v$ and $\widetilde{q}$.
Here, as in section \ref{section_f_boundary}, $\widetilde{q}$
 will be a given function
which replaces $p_0\circ\widetilde{\eta}$, and $f_0$ and $f_1$ are known functions.

Let $f$ be given by theorem \ref{theo_f_eq}. $f$ is then defined on 
$\Om$ and satisfies (\ref{system_f_nl_Dirichlet}), as stated earlier.
Since (\ref{f_interior_before_restriction}) gives (\ref{eq_f_bry}) on $\partial \Om$,
if we plug $f$ into the left hand side of (\ref{f_interior_before_restriction})
we obtain
\begin{gather}
\ddot{f} + \ccA(v, f)(f) + \ccB( v, f )(\dot{f}) + \ccC( v, \widetilde{q})
= \omega,
\nonumber
\end{gather}
where $\omega$ has the property that $\left. \omega \right|_{\partial \Om}$
is zero in $H^s_0(\partial \Om)$ (recall that we solved (\ref{eq_f_bry}) modulo
constants), so $\omega$ is constant on $\partial \Om$. Therefore,
if we work modulo functions that are constant on the boundary (which
suffices for our purposes), we find that $f$ automatically satisfies
the interior equation (\ref{f_interior_before_restriction}), and thus
(\ref{system_f_v_f_dot_dot}) as well.

\section{Proof of theorem \ref{main_theorem}: existence\label{section_existence}}

In this section we prove the existence part of theorem \ref{main_theorem}.
(recall remark \ref{remark_large_kappa}).
Before doing so, we summarize how the argument will be implemented. 

\subsection{Overview of the argument\label{overview}}

Here we motivate how the iteration yielding 
a solution to (\ref{free_boundary_full}) is implemented, and also fix some notation
for future reference, while following the same notation as above for quantities
that have already been introduced.

Assume we are given a solution to (\ref{free_boundary_full}), so that
 $\eta(t)$ is a curve of volume-preserving embeddings, 
$\dot{\eta} = u \circ \eta$, and $\dive u(t) = 0$. Write
\begin{gather}
u(t) = P u(t) + Q u(t) = P u(t) + \nabla h (t),
\label{u_decomp}
\end{gather}
where $h(t): \eta(t)(\Om) \rar \RR$ is harmonic. Recall that 
$u$ satisfies
\begin{gather}
\frac{\partial u}{\partial t} + \nabla_u u = - \nabla p,
  \text{ in } \IntDom.
\label{u_Eulerian}
\end{gather}
Letting $Pu = w$ and using the fact that 
\begin{gather}
\nabla_{\nabla h} \nabla h = \frac{1}{2} \nabla | \nabla h |^2,
\label{identity_nabla_h}
\end{gather}
we obtain, (after applying $P$ to (\ref{u_Eulerian})),
\begin{gather}
\frac{\partial w}{\partial t} + P (\nabla_u w +
\nabla_w \nabla h)  + \nabla H  =  0,  
  \text{ in } \IntDom,
\nonumber
\end{gather}
where we notice that $\nabla H$ is divergence-free and has normal component equal
to $\langle w, (N\circ \eta)\,\dot{} \circ \eta^{-1} \rangle + \langle \nabla_u w, N\rangle$.
Composing with $\eta$ gives
\begin{align}
\begin{split}
(w\circ \eta)\,\dot{} & 
= ( \frac{\partial w}{\partial t} + \nabla_u w ) \circ \eta \\
& = (  -P(\nabla_w \nabla h) + Q (\nabla_u w) ) \circ \eta + \nabla H \circ \eta.
\end{split}
\label{w_comp_eta_dot}
\end{align}
Letting $z = w\circ \eta$ then gives
\begin{gather}
\dot{z} = Q_\eta( (\nabla_u)_\eta (z) ) - 
P_\eta ( (\nabla_w)_\eta (\nabla h\circ \eta) ) + \nabla H \circ \eta.
\label{z_eq}
\end{gather}
Next, consider $\cD_\mu^s(\Om)$. It is a submanifold of 
$H^s(\Om,\RR^n)$.
Therefore, it has a
normal bundle given by the $L^2$ metric on $H^s(\Omega, \RR^n)$.  This metric
is of course invariant under right composition by elements of $D^s_{\mu}(\Om)$.
A tangent vector to $D^s_{\mu}(\Om)$ at $\beta$ is of the form $v \circ \beta$ where $\dive v =0$ and $v$ is parallel to $\partial \Omega$.
Hence, a normal vector is of the form $\nabla f \circ \beta$.  The
exponential map from the normal bundle to $H^s(\Omega, \RR^n)$
is a diffeomorphism in a neighborhood of $D^s_{\mu}(\Om)$.  
Therefore, if $\eta$ is near
$D^s_{\mu}(\Om)$, then there exist $\beta$ and $\nabla g$ such that 
\begin{gather}
\eta = (\id + \nabla g ) \circ \beta.
\label{eta_decom_exp}
\end{gather}
In other words, decomposition (\ref{def_f_beta}) holds for all $\eta$
sufficiently close to $\cD^s_\mu(\Om)$, although it is important to stress
that in this decomposition, given by the exponential map, 
$g$ need not to satisfy (\ref{eq_f_bry}) or 
(\ref{system_f_v_f_dot_dot})-(\ref{system_f_nl_Dirichlet}), 
nor does it need to be in 
$H^{s+\frac{5}{2}}(\Om)$ as are the solutions constructed in section
\ref{section_f_boundary}.

Now we can describe the iteration. Assume that we are given
a differentiable curve  $\eta(t)$ of $H^s$ embeddings and  $u \in H^s( \eta( \Om) )$.
From (\ref{eta_decom_exp}) we obtain $\beta$, and thus $v$, both in $H^s(\Om)$.
Then we use theorem \ref{theo_f_eq} to solve 
(\ref{eq_f_bry}), or equivalently,  
(\ref{system_f_v_f_dot_dot})-(\ref{system_f_nl_Dirichlet}),  for $f$.
Next, we obtain $h$ by solving
\begin{gather}
\begin{cases}
\Delta h = 0, & \text{ in } \widetilde{\eta}(\Om), \\
\frac{\partial h}{\partial \widetilde{N}} =
\langle (\nabla \dot{f} + D_v \nabla f + v ) \circ \widetilde{\eta}^{-1},
\widetilde{N} \rangle & \text{ on } \partial \widetilde{\eta}(\Om),
\end{cases}
\label{h_equation_Neumann}
\end{gather}
where $\widetilde{\eta} = \id + \nabla f$ as usual and $\widetilde{N}$ 
is the normal to $\widetilde{\eta}(\Om)$. Since 
$\nabla f \in H^{s+\frac{3}{2}}(\Om)$, $\nabla \dot{f}
 \in H^{s}(\Om)$, and $v \in H^s(\Om)$, we find that $\nabla h \in H^s(\Om)$.
Next, solve (\ref{z_eq}), which is an ODE for $z$. Note that 
$\nabla h \in H^s(\Om)$ and $u \in H^s(\Om)$ so 
$z \in H^s(\Om)$. Finally, let the new $\eta$ be
\begin{gather}
\eta(t) = \id + \int_0^t (z + (\nabla h) \circ \widetilde{\eta} \circ \beta ).
\nonumber
\end{gather}
This also gives the new $u$ by  $u = z \circ \beta^{-1}\circ
\widetilde{\eta}^{-1} + \nabla h$.
Using our estimates on $f$, we can then show that this iteration has
a fixed point. 

\subsection{Successive approximations\label{section_suc_app}} 
Here we carry out the fixed point argument sketched above.
Denote by
$\Emb^s(\Om,\RR^3) \equiv \Emb^s(\Om)$ the space of $H^s$
embeddings of $\Om$ into $\RR^3$ (not necessarily volume-preserving). \\

\textbf{Induction hypothesis}
{\it (step $n$, at which the $(n+1)$st quantities will be determined).
Assume inductively that we are given 
\begin{gather}
\eta_n \in C^2([0,T],\text{\normalfont Emb}^s(\Om)) \cap C^1([0,T], H^s(\Om)) \cap C^2([0,T],H^{s-\frac{3}{2}}
(\Om)),
\nonumber
\end{gather}
satisfying $\eta_n(0) = \id$ and $\dot{\eta}_n(0) = u_0$, where $u_0$ is the
given initial velocity for (\ref{free_boundary_full}). Suppose that 
\begin{subequations}
\begin{align}
& \p \eta_n - \id  \p_s \leq R_s, 
\label{induction_bound_1} \\
& \p \dot{\eta}_n(0) \p_s  = \p u_0 \p_s  < \Rin 
\label{induction_bound_2} \\
& \p \dot{\eta}_n \p_s \leq \Rin, 
\label{induction_bound_3}  \\
& \p \ddot{\eta}_n(0) \p_{s-\frac{3}{2}} \leq \ddot{R}_s^0 < \Rin,
\label{induction_bound_4}  \\
& \p \ddot{\eta}_n \p_{s-\frac{3}{2}}  \leq \Rin,
\label{induction_bound_5} 
\end{align}
\end{subequations}
for some constants $R_s$, 
$\ddot{R}_s^0$, and $\Rin$ which will be suitably chosen.
 Assume also that
\begin{gather}
\dot{\eta}_n = \widehat{u}_n \circ (\id + \nabla f_n )\circ \beta_n,
\label{dot_eta_n_ind}
\end{gather}
with  $\widehat{u}_n$ divergence-free, 
$\widehat{u}_n \in H^s( (\id + \nabla f_n )\circ \beta_n(\Om) )$,
$\dot{\widehat{u}}_n \in H^{s-\frac{3}{2}}( (\id + \nabla f_n )\circ \beta_n(\Om) )$, 
where $\beta_n$ 
is a $C^1$-curve
$\beta_n: [0,T] \rar \cD_\mu^s(\Om)$ which 
also satisfies $\ddot{\beta}_n \in H^{s-\frac{3}{2}}(\Om)$, 
and 
\begin{gather}
\nabla f_n \in C^0([0,T], H^{s+\frac{3}{2}}(\Om)) \cap
C^1([0,T], H^{s}(\Om) )\cap 
C^2([0,T], H^{s-\frac{3}{2}}(\Om) ).
\nonumber
\end{gather}
Let $v_n$ be given by $\dot{\beta}_n = v_n \circ \beta_n$.
Assume further that $f_n$ 
is obtained from theorem \ref{theo_f_eq}. Notice that in employing 
theorem \ref{theo_f_eq}, $v$ and $\widetilde{q}_0$ are needed. 
We take $v_{n}$ as $v$, and let $\widetilde{q}_0$ be determined
from $\eta_n$ in the inductive process, as shown below.
Let $\nabla h_n$ and $w_n$ be the gradient and divergence free parts
of $\widehat{u}_n$, and let $z_n = w_n \circ (\id + \nabla f_n) \circ \beta_n$.
Finally, assume that 
\begin{gather}
\p z_n \p_s, \p w_n \p_s, \p \nabla h_n \p_s, 
 \p \widehat{u}_n \p_s, \p \beta_n \p_s, \p \dot{\beta}_n \p_s, \p v_n \p_s
\leq 
\Rin
\label{ind_bounds_1}
\end{gather}
and
\begin{gather}
\p \dot{z}_n \p_{s-1}, 
\p \dot{w}_n \p_{s-\frac{3}{2}},
\p \nabla \dot{h}_n \p_{s-\frac{3}{2}},
\p \nabla \dot{f}_n \p_{s}, 
 \p \dot{\widehat{u}}_n \p_{s-\frac{3}{2}},
\p \ddot{\beta}_n \p_{s-\frac{3}{2}}, \p \dot{v}_n \p_{s-\frac{3}{2}} \leq \Rin\, ,
\label{ind_bounds_2}
\end{gather}
where (\ref{ind_bounds_1}) and (\ref{ind_bounds_2}) mean that each one of those
quantities is bounded by $\Rin$.
} 
\begin{remark}
The reason for introducing $\widehat{u}_n$ is that $\eta_n$ may not be
volume preserving, so that $u_n$ given by $\dot{\eta}_n = u_n \circ \eta_n$
may not be divergence-free. We need, however, a velocity that is divergence
free to get the correct regularity for the pressure, since it will involve
$\dive (\nabla_u u)$ and we need this to depend on at most one derivative of
$u$. 
\end{remark}
We shall show that if $T$ and the constants $R_s$, $\Rin$, etc.
 are correctly chosen,
and $\kappa$ is large, 
one can construct $\eta_{n+1}$ and the corresponding quantities
satisfying the above conditions. This will give the desired sequence.
For simplicity, we divide the procedure into several steps.\\

\noindent \emph{Step $(n1)$.} 
We begin by noting that 
$\cD_\mu^s(\Om)$ has a smooth normal bundle inside $H^s(\Om,\RR^3)$,
with smooth exponential map, both with respect to the $L^2$-metric
\cite{EM}.
Thus, given the curve of embeddings 
$\eta_n$, if $R_s$ is sufficiently 
small the exponential map gives
\begin{gather}
\eta_n = (\id  + \nabla g_n) \circ \ga_n,
\nonumber
\end{gather}
with $\nabla g_n \in H^s(\Om)$ and $\ga_n \in \cD_\mu^s(\Om)$
having the same regularity properties as $\eta_n$, 
i.e., 
$\nabla g_n,\ga_n \in C^1([0,T], H^s(\Om)) \cap C^2([0,T],H^{s-\frac{3}{2}}(\Om))$.
In fact, the normal bundle to $D_{\mu}^s(\Om)$ at $\eta$ 
is \linebreak
 $\nabla_{\eta} \Delta^{-1}_{\eta} \dive_{\eta}( H^s(\Omega, R^n))$; 
$\dive_{\eta}:H^s \mapsto H^{s-1}$ and
$\nabla_{\eta} \Delta^{-1}_{\eta}: H^{s-1} \mapsto H^s$ are both smooth in $\eta$,
as shown in \cite{EM}. Thus the normal bundle is 
smooth in the $H^s$ topology even though it is normal only in the $L^2$ sense.
Therefore, the smooth exponential map on the normal bundle makes sense and it 
is a diffeomorphism in a neighborhood of the zero-section, and from this it
follows that  $\nabla g_n $ and $\ga_n$ do 
have the stated regularity.
It also follows that 
$\nabla g_n(0) = 0$ and $\ga_n(0) = \id$.
Define $\beta_{n+1} = \ga_n$. This also gives $v_{n+1}$ by 
$\dot{\beta}_{n+1} = v_{n+1} \circ \beta_{n+1}$. Since $\beta_{n+1}(0) = \id$,
given $R_s$, 
if $T$ is small we will have $\p \beta_{n+1} - \id \p_s \leq R_s$.
Differentiating $\eta_n$ in time, evaluating at zero, and applying 
$P$ (see (\ref{dot_eta_f_beta})), one obtains
\begin{gather}
\dot{\beta}_{n+1}(0) = v_{n+1}(0) = P \dot{\eta}_n(0) = P u_0.
\label{v_n_plus_1_t_zero}
\end{gather}
By the bound (\ref{induction_bound_2}) and the continuity in $t$
of $v_{n+1}$, we see that if $\Rin$ is taken sufficiently large
one has $\p v_{n+1} \p_s \leq \Rin$ and 
 $\p \dot{\beta}_{n+1} \p_s \leq \Rin$. Differentiating 
 $\eta_n$ twice in time, evaluating at zero, and applying $P$ produces
(see (\ref{equation_ddot_nabla_f}))
\begin{gather}
\dot{v}_{n+1}(0) + P(\nabla_{v_{n+1}(0)} v_{n+1}(0) )
= P(\ddot{\eta}_n(0)) - 2PD_{v_{n+1}(0)} \nabla \dot{g}_n(0).
\nonumber
\end{gather}
In view of (\ref{induction_bound_2}) again, we can bound 
$\nabla \dot{g}_n(0)$, and then, using (\ref{induction_bound_4})
and the continuity in time of the quantities involved, we obtain
$\p \dot{v}_{n+1} \p_{s-\frac{3}{2}} \leq \Rin$ and 
 $\p \ddot{\beta}_{n+1} \p_{s-\frac{3}{2}} \leq \Rin$, provided that $\Rin$ is
 chosen very large. \\

\noindent \emph{Step $(n2)$.}  Define $p_{0,n+1}$ by 
solving\footnote{Note that the subscript zero in $p_{0,n+1}$  is used to indicate that $p_{0,n+1}$ is the interior 
pressure and should not be confused with the step of the iteration which is $n$.}:
\begin{gather}
\begin{cases}
\Delta p_{0,n+1} = - \dive (\nabla_{\widehat{u}_n }\widehat{u}_n  ), & 
\text{ on } (\id + \nabla f_n )\circ \beta_n(\Om), \\
p_{0,n+1} = 0, & \text{ on }
\partial (\id + \nabla f_n )\circ \beta_n(\Om).
\end{cases}
\label{p_0_equation_inducation}
\end{gather}
Since $\widehat{u}_n$ is divergence-free, 
$\dive (\nabla_{\widehat{u}_n }\widehat{u}_n  )$ is in $H^{s-1}$, and thus
it is expected that $p_{0,n+1}$ is $H^{s+1}$-regular. 
This has to be verified, however, since one is solving a Dirichlet problem on the domain
$\partial (\id + \nabla f_n )\circ \beta_n(\Om)
= \partial (\id + \nabla f_n ) (\Om)$, which is not smooth
(compare the ensuing argument with that involving $\cH_{\widetilde{\eta}}$ in 
proposition
\ref{prop_first_order}).
Let $\widetilde{\eta}_n = (\id + \nabla f_n )$, and consider the operator
$\Delta_{\widetilde{\eta}_n}$ which acts on functions defined over $\Om$
(see notation \ref{notation_sub}). If $G$ is defined over $\Om$, a computation 
gives
\begin{gather}
\Delta_{\widetilde{\eta}_n} G
= ((D \widetilde{\eta}_n )^{-1})^\al_\be ( \partial_{\al \ga} G  ) 
((D \widetilde{\eta}_n )^{-1})^{\ga}_{\be} 
+ 
\partial_\al 
((D \widetilde{\eta}_n )^{-1})^{\al}_{\be} \partial_\be G, 
\nonumber
\end{gather}
where $((D \widetilde{\eta}_n )^{-1})^\al_\be$ are the entries of $(D \widetilde{\eta}_n )^{-1}$.
Thus, $\Delta_{\widetilde{\eta}_n}$ has the form
\begin{gather}
\Delta_{\widetilde{\eta}_n} G = a^{\al\be}\partial_{\al\be} G + b^\al \partial_\al G.
\nonumber
\end{gather}
Since $\nabla f_n \in H^{s+\frac{3}{2}}(\Om)$ we conclude that 
$a^{\al\be}$ is in $H^{s+\frac{1}{2}}(\Om,\RR^{3^2})$ and
$b^\al$ is in $H^{s-\frac{1}{2}}(\Om, \RR^3)$. Furthermore, 
since $\nabla f_n$ is small in $H^{s+\frac{3}{2}}(\Om)$ and $s> \frac{3}{2} + 2$, 
we find that $(a^{\al\be})$ is positive definite.
We conclude that $\Delta_{\widetilde{\eta}_n}$ is an elliptic operator
that takes $H^{s+1}(\Om)$ into $H^{s-1}(\Om)$ and, furthermore, that
it gives rise to an isomorphism $\Delta_{\widetilde{\eta}_n}: H_0^{s+1}(\Om)
\rar H_0^{s-1}(\Om)$. In particular, the corresponding Dirichlet
problem for $\Delta_{\widetilde{\eta}_n}$ is uniquely solvable.
But finding $p_{0,n+1}$ in (\ref{p_0_equation_inducation})
 is equivalent to solving
\begin{gather}
\begin{cases}
\Delta_{\widetilde{\eta}_n} ( p_{0,n+1} \circ \widetilde{\eta}_n )
 = - \dive (\nabla_{\widehat{u}_n }\widehat{u}_n  ) \circ \widetilde{\eta}_n, & 
\text{ in } \Om, \\
p_{0,n+1}\circ \widetilde{\eta}_n  = 0, & \text{ on }
\partial \Om.
\end{cases}
\nonumber
\end{gather}
From the above, we know that this has a unique solution 
$\widetilde{q}_{0,n+1} \in H^{s+1}(\Om)$, and thus we obtain
the desired $p_{0,n+1} = \widetilde{q}_{0,n+1} \circ (\widetilde{\eta}_n)^{-1}$ 
in $H^{s+1}((\id + \nabla f_n )\circ \beta_n(\Om))$.

Next, upon differentiating, we obtain
\begin{gather}
\Delta_{\widetilde{\eta}_n} \dot{\widetilde{q}}_{0,n+1}
 = - (\dive (\nabla_{\widehat{u}_n }\widehat{u}_n  ) \circ \widetilde{\eta}_n)\,\dot{}
 -((\Delta_{\widetilde{\eta}_n})\, \dot{} \, ) \, \widetilde{q}_{0,n+1},
\label{elliptic_eq_q_0_dot}
\end{gather}
and a standard computation shows that 
\begin{gather}
(\Delta_{\widetilde{\eta}_n})\, \dot{} = [ \nabla_{\widehat{u}_n}, \Delta ]_{\widetilde{\eta}_n}.
\label{elliptic_eq_q_0_dot_comm}
\end{gather}
Combining (\ref{elliptic_eq_q_0_dot}) and (\ref{elliptic_eq_q_0_dot_comm}), and
using the ellipticity of $\Delta_{\widetilde{\eta}_n}$, we obtain
$\dot{\widetilde{q}}_{0,n+1} \in H^{s-\frac{1}{2}}(\Om,\RR))$, by using the regularity of the quantities on the right hand side 
of (\ref{elliptic_eq_q_0_dot}). In fact, we know more: from the induction
hypothesis and the above constructions, we conclude that 
\begin{gather}
\widetilde{q}_{0,n+1} \in 
C^0 ([0,T],H^{s+1}(\Om,\RR))
\cap
C^1 ([0,T],H^{s-\frac{1}{2}}(\Om,\RR)).
\nonumber
\end{gather}
The ellipticity of $\Delta_{\widetilde{\eta}_n}$ 
also implies that 
$\p \widetilde{q}_{0,n+1} \p_{s+1}$ can be bounded in terms of $\p \widehat{u}_n \p_{s}$ and
$\p \nabla f_n \p_{s+\frac{3}{2}}$, and 
that  $\p \dot{\widetilde{q}}_{0,n+1} \p_{s-\frac{1}{2}}$ is bounded in terms of
$\p \widehat{u}_n \p_{s}$, $\p \widehat{u}_n \p_{s-\frac{3}{2}}$, 
$\p \nabla f_n \p_{s+\frac{3}{2}}$ and $\p \nabla \dot{f}_n \p_{s}$.
Such bounds hold in particular at $t=0$, when they are  given
in terms of the quantities just mentioned, but now evaluated at time zero. Thus,
as before, continuity in $t$ gives a bound on 
$\p \widetilde{q}_{0,n+1} \p_{s+1}$ and $\p \dot{\widetilde{q}}_{0,n+1} \p_{s+1}$ 
in terms of $\Rin$, provided that $\Rin$ is sufficiently large, 
i.e., 
 $\p \widetilde{q}_{0,n+1} \p_{s+1} \leq \Rin$ and 
 $\p \dot{\widetilde{q}}_{0,n+1}\p_{s-\frac{1}{2}} \leq  \Rin$. 
Invoking  (\ref{Sobolev_composition}) produces 
bounds for  $\p p_{0,n+1} \p_{s+1}$ and $\p \dot{p}_{0,n+1} \p_{s+1}$
in terms of $\Rin$ as well. \\

\noindent \emph{Step $(n3).$} With $v_{n+1}$ and $\widetilde{q}_{0,n+1}$
obtained above, we use theorem \ref{theo_f_eq} to solve the $f$-equation with initial conditions 
$f(0) = 0$ and $\dot{f}(0) = \left. \Delta_\nu^{-1} \dive u_0 \right|_{\partial \Om}$ 
(recall that $f$ is determined up to constants), obtaining $f_{n+1}$. 
In doing so, we need to assure that the initial condition 
$\dot{f}(0)$ satisfies (\ref{f_1_hypothesis}). This is the
case if $\p Q u_0 \p_s$ is taken sufficiently small, as assumed in theorem
\ref{main_theorem}.
The bounds obtained for $v_{n+1}$ and $\widetilde{q}_{0,n+1}$
determine the constant $K_4$ in that theorem. Taking $\Rin$ larger if necessary, 
we can take $K_4 = \Rin$. Note also that $J(\id + \nabla f_{n+1}) = 1$ and
$\id + \nabla f_{n+1}$ is an embedding. 

\begin{remark}
Notice that theorem \ref{theo_f_eq} does not say anything about the uniqueness of $f_{n+1}$,
thus at this point we let $f_{n+1}$ be any of the (possibly more than one) solutions
given by that theorem. We will eventually show that uniqueness does hold for the desired equation, i.e.,
(\ref{free_boundary_full}).
\end{remark}

\noindent \emph{Step $(n4).$} Obtain $h_{n+1}$ solving
\begin{gather}
\begin{cases}
\Delta h_{n+1} = 0, & \text{ in } (\id + \nabla f_{n+1})(\Om), \\
\frac{\partial h_{n+1}}{\partial \widetilde{N}_{n+1}} =
\langle (\nabla \dot{f}_{n+1} + D_{v_{n+1}} \nabla f_{n+1} + v_{n+1} ) \circ 
(\id + \nabla f_{n+1})^{-1},
\widetilde{N}_{n+1} \rangle & \text{ on } \partial(\id + \nabla f_{n+1}) (\Om),
\end{cases}
\nonumber
\end{gather}
where $\widetilde{N}_{n+1}$ is the unit normal to 
$\partial(\id + \nabla f_{n+1}) (\Om)$. This gives
$\nabla h_{n+1} \in H^s( (\id + \nabla f_{n+1}) (\Om) )$ and
$\nabla \dot{h}_{n+1} \in H^{s-\frac{3}{2}}( (\id + \nabla f_{n+1}) (\Om) )$.
We argue as above to conclude that 
$\p \nabla h_{n+1} \p_s$ and 
$\p \nabla \dot{h}_{n+1} \p_{s-\frac{3}{2}}$ are bounded by $\Rin$.
Here, again,  in producing the estimate we use elliptic theory, so some of the 
constants involved depend on the domain, namely, on $(\id + \nabla f_{n+1}) (\Om)$,
but these constants are controlled as before. \\

\noindent \emph{Step $(n5).$} Set $\overline{\eta}_{n+1} = 
(\id + \nabla f_{n+1}) \circ \beta_{n+1}$. By construction this is a volume-preserving embedding, so the velocity $\overline{u}_{n+1}$ given by
$\dot{\overline{\eta}}_{n+1} = \overline{u}_{n+1} \circ \overline{\eta}_{n+1}$
is divergence-free and has the same regularity and bounds as $\widehat{u}_n$. However
$\overline{\eta}_{n+1}$
is not yet the embedding we are seeking to conclude the $n$-th step. The reason
to introduce it is to obtain $\overline{u}_{n+1}$, since it will enter in the
equation for $z$ below, but it has to be defined on $(\id + \nabla f_{n+1})
\circ \be_{n+1}(\Om)$ (on which $\widehat{u}_n$ is not defined). Using
as input $\nabla h_{n+1}$, $\overline{u}_{n+1}$, $\nabla f_{n+1}$, and
$\be_{n+1}$, and the initial condition $P u_0$,
consider the following equation for $z_{n+1}$ 
(or, equivalently, for $w_{n+1}$; compare with 
(\ref{w_comp_eta_dot}) and (\ref{z_eq})): 
\begin{align}
\begin{split}
\dot{z}_{n+1} & = Q_{\overline{\eta}_{n+1}}
( (\nabla_{\overline{u}_{n+1}} )_{\overline{\eta}_{n+1}} (z_{n+1}) ) - 
P_{\overline{\eta}_{n+1}}
 ( (\nabla_{z_{n+1} \circ \overline{\eta}_{n+1}^{-1}})_{\overline{\eta}_{n+1}} \\
 & +
(\nabla h_{n+1}\circ {\overline{\eta}_{n+1}}) ) + \nabla H_{n+1} \circ \overline{\eta}_{n+1},
\end{split}
\label{z_equation_iteration}
\end{align}
with initial condition $z(0) = P u_0$, where $u_0$ is the initial velocity given in the 
statement of theorem \ref{main_theorem},
and $\nabla H_{n+1}$ is divergence-free and has normal component equal
to $\langle w_{n+1}, (\overline{N}_{n+1} \circ \overline{\eta}_{n+1})\,\dot{} \circ 
\overline{\eta}^{-1}_{n+1} \rangle + \langle\nabla_{\overline{u}_{n+1}} w_{n+1}, \overline{N}_{n+1}\rangle$.

We shall first show that 
(\ref{z_equation_iteration}) has a solution in $H^{s-1}$, and then that $z_{n+1}$
in indeed in $H^s$.

Although the boundary of $\overline{\eta}_{n+1}(\Om)$ 
is not smooth, it is 
sufficiently regular to guarantee that the operators $P_{\overline{\eta}_{n+1}}$ and $Q_{\overline{\eta}_{n+1}}$
are bounded on $H^s(\overline{\eta}_{n+1}(\Om))$.
In fact,  
$\nabla f_{n+1} \in H^{s+\frac{3}{2}}(\Om)$ and $\be_{n+1} \in \cD_\mu^s(\Om)$. Since $\be_{n+1}(\partial \Om) =
\partial \Om$, $\overline{\eta}_{n+1}(\partial \Om) = (\id + \nabla f_{n+1}) (\partial \Om)$. 
Thus, because $\nabla f_{n+1}
\in H^{s+\frac{3}{2}}(\Om)$, $\partial \overline{\eta}_{n+1}(\Om)$ can be written locally 
as the graph of an $H^{s+1}(\partial \Om)$ function and, therefore, the normal to $\partial 
\overline{\eta}_{n+1}(\Om)$
is $H^s$-regular. 

The term $P \nabla_{w_{n+1}} \nabla h_{n+1} $
 is in $H^{s-1}$ if $w_{n+1}$ is in 
$H^{s-1}$ because $\nabla h$ is in $H^s$, and $P$ is an operator of order 
zero. The term 
$Q \nabla_{\overline{u}_{n+1}} w_{n+1}$ is in $H^{s-1}$ if
$w_{n+1}$ is in $H^{s-1}$ and $P w_{n+1} = w_{n+1}$. Indeed,
if $w_{n+1}$ is in the image of $P$ we can write
\begin{gather}
Q \nabla_{\overline{u}_{n+1}} w_{n+1}
= [ Q, \nabla_{\overline{u}_{n+1}} ] w_{n+1}.
\nonumber
\end{gather}
The commutator $ [ Q, \nabla_{\overline{u}_{n+1}} ] $ is a zeroth
order operator depending on first derivatives of $\overline{u}_{n+1}$;
since $\overline{u}_{n+1}$ is $H^s$-regular, we obtain that 
$Q \nabla_{\overline{u}_{n+1}} w_{n+1}$ is in $H^{s-1}$.

Therefore, (\ref{z_equation_iteration}) can be viewed as a ODE
for $z_{n+1}$ in 
$P_{\overline{\eta}_{n+1}}( P H^{s-1} ( \overline{\eta}_{n+1} (\Om) ) )$,
i.e., the space of $H^{s-1}$ vector fields over $\Om$ of the form $X = W \circ
\overline{\eta}_{n+1}$ with $PW = W$. Indeed, an element 
$X \in P_{\overline{\eta}_{n+1}}( P H^{s-1} ( \overline{\eta}_{n+1} (\Om) ) )$
is of the form $X = P_{\overline{\eta}_{n+1}} Y$, with $Y \in 
 P H^{s-1} ( \overline{\eta}_{n+1} (\Om) )$. But then $Y = W \circ
 \overline{\eta}_{n+1}$, with $PW = W$, thus
 \begin{align}
 \begin{split}
 X & = P_{\overline{\eta}_{n+1}} Y
  = ( P ( Y \circ \overline{\eta}_{n+1}^{-1} ) ) \circ \overline{\eta}_{n+1}
= (P W ) \circ \overline{\eta}_{n+1}. 
 \end{split}
 \nonumber
 \end{align}
The right hand side of (\ref{z_equation_iteration}) depends 
linearly on $w_{n+1}$, and because composition on the right is a smooth
map (see, e.g.,  \cite{E1, EM, P}), we conclude that this ODE has a 
$H^{s-1}$ solution
$z_{n+1}$ for a small time interval $T$. Notice that $\dot{z}_{n+1}$
is also in $H^{s-1}$.

Letting $w_{n+1} = z_{n+1} \circ \overline{\eta}_{n+1}^{-1}$, 
it follows that $w_{n+1}$ is in $H^{s-1}$, that $P w_{n+1} = w_{n+1}$, and 
that $w_{n+1}$  satisfies
\begin{gather}
\frac{\partial w_{n+1}}{\partial t} 
+ P (\nabla_{\overline{u}_{n+1}}  w_{n+1} + \nabla_{w_{n+1}} 
\nabla h_{n+1} ) 
+ \nabla H_{n+1} \circ \overline{\eta}_{n+1}
= 0
\, \text{ in } \,
\underset{0 \leq t \leq T}{\bigcup} \{t\} \times \overline{\eta}_{n+1}(\Om).
\label{w_eq_iteration}
\end{gather}
 
We now show that $w_{n+1}$ is in $H^s$.
 In order to do so, suppose first
that 
 $\nabla h_{n+1}$, $\overline{u}_{n+1}$, 
$\overline{\eta}_{n+1}$, and $ u_0$ belong to $H^{N+1}$, where $N$ is some big number
larger than $s$. The above ODE argument then produces
$z_{n+1}$ and $w_{n+1}$ in $H^N$. We shall establish the following
a priori bound 
\begin{align}
\begin{split}
\p w_{n+1} \p_s & \leq C
( \p w_{n+1} \p_0 + 
(1 + \p \overline{\eta}_{n+1}^{-1} \p_s^s )
 \p u_0 \p_s 
 )
 \\
 &
 \times 
e^{C
(1 + \p \overline{\eta}_{n+1}^{-1} \p_s^s)
\int_0^t (1+ \p \overline{\eta}_{n+1} \p_s^s )(\p \overline{u}_{n+1} \p_s + \p \nabla h_{n+1} \p_s )}.
\end{split}
\label{a_priori_w}
\end{align}
Once (\ref{a_priori_w}) is established, one can take a sequence 
$\{ \nabla h_{n+1,j}, \overline{u}_{n+1,j},
\overline{\eta}_{n+1,j}, u_{0,j} \}_{j=1}^\infty$
of 
$H^{N+1}$
functions converging in $H^s$ to the original $\nabla h_{n+1}$, $\overline{u}_{n+1}$, 
$\overline{\eta}_{n+1}$, $u_0$, and then (\ref{a_priori_w}) implies 
that the corresponding $H^N$ solutions $w_{n+1,j}$ converge in $H^{s-1}$
to an element $w_{n+1}$ that is in fact in $H^s$; this $w_{n+1}$ will indeed be a solution of
(\ref{w_eq_iteration}) because we already know that is has a solution
in $H^{s-1}$.

Let $y = \curl w_{n+1}$. Taking the $\curl$ 
of (\ref{w_eq_iteration}) and using the fact that $\curl \nabla = 0$ gives
\begin{gather}
\frac{\partial y}{\partial t} 
+ \nabla_{\overline{u}_{n+1}}  y 
+
[ \curl, \nabla_{\overline{u}_{n+1}}] w_{n+1} 
+ [ \curl,  \nabla_{w_{n+1}} ] \nabla h_{n+1} = 0
\, \text{ in } \,
\underset{0 \leq t \leq T}{\bigcup} \{t\} \times \overline{\eta}_{n+1}(\Om).
\label{curl_w_eq_1}
\end{gather}
Here we used that fact that for any vector field $X$, $\curl P X = \curl X$ 
since $\curl Q = 0$. Composing (\ref{curl_w_eq_1}) with $\overline{\eta}_{n+1}$
 leads to
\begin{gather}
(y \circ \overline{\eta}_{n+1})\,\dot{}
= - 
( [ \curl, \nabla_{\overline{u}_{n+1}}] w_{n+1} ) \circ \overline{\eta}_{n+1}
-( [ \curl,  \nabla_{w_{n+1}} ] \nabla h_{n+1} ) \circ \overline{\eta}_{n+1}.
\nonumber
\end{gather}
Thus
\begin{align}
\begin{split}
y \circ \overline{\eta}_{n+1}
& = (y \circ \overline{\eta}_{n+1})(0) - 
\int_0^t ( [ \curl, \nabla_{\overline{u}_{n+1}}] w_{n+1} ) \circ \overline{\eta}_{n+1}
\\
& - \int_0^t 
( [ \curl,  \nabla_{w_{n+1}} ] \nabla h_{n+1} )
\circ \overline{\eta}_{n+1}.
\end{split}
\label{curl_w_eq_2}
\end{align}
$[ \curl, \nabla_{\overline{u}_{n+1}}]$ is a first order operator depending
on first derivatives of $\overline{u}_{n+1}$. Thus, using (\ref{Sobolev_composition}) we derive
\begin{gather}
\p ( [ \curl, \nabla_{\overline{u}_{n+1}}] w_{n+1} ) \circ \overline{\eta}_{n+1}
\p_{s-1}  
\leq C \p \overline{u}_{n+1} \p_s \p w_{n+1} \p_s(1 
+ \p \overline{\eta}_{n+1} \p_{s}^s ),
\label{comm_w_1}
\end{gather}
 Similarly,
$ [ \curl,  \nabla_{w_{n+1}} ]$ is a first order operator depending
on derivatives of $w_{n+1}$, hence
\begin{gather}
\p ( [ \curl,  \nabla_{w_{n+1}} ] \nabla h_{n+1} ) \circ \overline{\eta}_{n+1}
\p_{s-1}  
\leq C \p \nabla h_{n+1} \p_s \p w_{n+1} \p_s(1 
+ \p \overline{\eta}_{n+1} \p_{s}^s ).
\label{comm_w_2}
\end{gather}
Combining (\ref{curl_w_eq_2}), (\ref{comm_w_1}), and (\ref{comm_w_2})
produces
\begin{align}
\begin{split}
\p y \circ \overline{\eta}_{n+1} \p_{s-1} &\leq
C \p u_0 \p_s 
 + C \int_0^t
(1+ \p \overline{\eta}_{n+1} \p_s^s ) ( \p \overline{u}_{n+1} \p_s + 
\p \nabla h_{n+1} \p_s ) \p w_{n+1} \p_s,
\end{split}
\label{curl_w_eq_3}
\end{align}
where we used that $ (y \circ \overline{\eta}_{n+1})(0) = y(0)
= \curl u_0$. But
\begin{gather}
\p \curl w_{n+1} \p_{s-1} = 
\p y \p_{s-1} = \p y \circ \overline{\eta}_{n+1} \circ \overline{\eta}_{n+1}^{-1} 
\p_{s-1} \leq
C(1 + \p \overline{\eta}_{n+1}^{-1} \p_s^s ) \p y \circ \overline{\eta}_{n+1} \p_{s-1},
\nonumber
\end{gather}
which combined with (\ref{curl_w_eq_3}) gives
\begin{align}
\begin{split}
\p \curl w_{n+1} \p_{s-1} &\leq
C(1 + \p \overline{\eta}_{n+1}^{-1} \p_s^s )
 \p u_0 \p_s 
\\
& + C(1 + \p \overline{\eta}_{n+1}^{-1} \p_s^s )
\int_0^t
(1+ \p \overline{\eta}_{n+1} \p_s^s )
 ( \p \overline{u}_{n+1} \p_s + 
\p \nabla h_{n+1} \p_s ) \p w_{n+1} \p_s.
\end{split}
\label{curl_w_eq_4}
\end{align}
We now use  (\ref{div-curl-estimate})
to estimate $w_{n+1}$ in $H^s$, noting that the 
terms $\dive w_{n+1}$ and $\langle w_{n+1}, N \rangle$,
where $N$ is the normal to $\partial \overline{\eta}_{n+1}(\Om)$,
do not contribute because $P w_{n+1} = w_{n+1}$.
We also note that we are allowed to invoke (\ref{div-curl-estimate}) because
$\overline{\eta}_{n+1}$ is in $H^s$ and $s > \frac{3}{2} + 2$.
Thus, (\ref{div-curl-estimate}) and (\ref{curl_w_eq_4}) give
\begin{align}
\begin{split}
\p w_{n+1} \p_{s} &\leq
C \p w_{n+1} \p_0 + 
C(1 + \p \overline{\eta}_{n+1}^{-1} \p_s^s )
 \p u_0 \p_s 
\\
& + C(1 + \p \overline{\eta}_{n+1}^{-1} \p_s^s )
\int_0^t (1+ \p \overline{\eta}_{n+1} \p_s^s )
 ( \p \overline{u}_{n+1} \p_s + 
\p \nabla h_{n+1} \p_s ) \p w_{n+1} \p_s.
\end{split}
\nonumber
\end{align}
Iterating this inequality now produces (\ref{a_priori_w}).
Note that $w_{n+1} \in H^s$ also gives $z_{n+1}$ in $H^s$.
As before, one gets a bound on  
$\p z_{n+1} \p_s$ and 
$\p  \dot{z}_{n+1} \p_{s-1}$ in terms of $\Rin$. This finishes
step $(n5)$. \\

Now that we have $z_{n+1}$ (or $w_{n+1}$), $\nabla h_{n+1}$, $\nabla f_{n+1}$,
and $\beta_{n+1}$, we define
\begin{gather}
\eta_{n+1} = 
\id + \int_0^t (z_{n+1} + \nabla h_{n+1} \circ (\id + \nabla f_{n+1} ) 
\circ \beta_{n+1} ).
\label{def_eta_n_plus_1}
\end{gather}
Given $R_s$, if $T$ is small, we get $\p \eta_{n+1} - \id \p_s \leq R_s$.
$\eta_{n+1}$ will be in $\Emb^s(\Om)$ if it is sufficiently close to the identity,
and we choose $R_s$ accordingly.

Differentiating $\eta_{n+1}$ we get
\begin{align}
\begin{split}
\dot{\eta}_{n+1} & =  z_{n+1} + \nabla h_{n+1} \circ (\id + \nabla f_{n+1} ) 
\circ \beta_{n+1} \\
& = ( w_{n+1} + \nabla h_{n+1} ) \circ (\id + \nabla f_{n+1} ) 
\circ \beta_{n+1} ,
\end{split}
\nonumber
\end{align}
so we let 
\begin{gather}
\widehat{u}_{n+1} = 
w_{n+1} + \nabla h_{n+1}.
\nonumber
\end{gather}
Then $\widehat{u}_{n+1}$ is divergence-free 
and has the same regularity as $\widehat{u}_{n}$. From the previous bounds, we obtain
the desired estimates of 
$\widehat{u}_{n+1}$ and $\dot{\widehat{u}}_n$ in terms of $\Rin$, possibly 
after increasing $\Rin$. From the initial condition for the $f$-equation in step \emph{(n3)},
and (\ref{v_n_plus_1_t_zero}), it follows that $\dot{\eta}_{n+1}(0) = u_0$.

Differentiating $\eta_{n+1}$ in time twice and evaluating at $t=0$, we
can control $\ddot{\eta}_{n+1}(0)$ in terms of the other  quantities of the $n+1$st iteration
at time zero, which are inductively bounded by a function of $R_s$, $u_0$, $T$ and $\ddot{R}_s^0$
in view of (\ref{induction_bound_4}). Relabeling the constants and choosing 
$\ddot{R}_s^0$ and $\Rin$ appropriately, we conclude that the desired inductive
bounds hold for the  quantities of the $n+1$st iteration for the desired time interval.

Finally, a careful analysis of the above steps reveals that, after 
$T$ and the constants $R_s$, $\ddot{R}_s^0$ and $\Rin$ are suitably chosen in order to assure
that the  quantities of the $n+1$st iteration satisfy the inductive assumptions, these constants can 
be chosen uniformly, i.e., independent of $n$. In particular there exists
a $T>0$ that works for all $n$,  
and we get a well-defined
sequence $\{ \eta_n \}$ of embeddings, provided that
the procedure holds at $n=0$, i.e., to start the above iteration,  $\eta_0$ is needed.
Let $\zeta$ be a solution
to the Euler equations in the fixed domain $\Om$ with initial conditions
$\zeta(0) = \id$, $\dot{\zeta}(0) = Pu_0$, where $u_0$ is the given initial
condition for the free-boundary Euler equations (\ref{free_boundary_full}). 
For any $R_s$ that we choose, we can pick $T$ sufficiently small such that
$\p \zeta - \id \p_s \leq R_s$. Letting $\vartheta$
be given by $\dot{\zeta} = \vartheta \circ \zeta$, standard energy
estimates for the Euler equations produce bounds for $\p \vartheta \p_s$ and
$\p \dot{\vartheta} \p_{s-1}$ in terms of a constant that depends on $T$ 
and $u_0$. Hence, we can find a constant $C_0(T,u_0,R_s)$ 
depending on $T$, $u_0$, 
and $R_s$ such that
\begin{align}
\begin{split}
&\p \zeta - \id  \p_s \leq R_s, \\
&\p \vartheta \p_s \leq C(T,u_0,R_s), \, \p \dot{\zeta} \p_s \leq C(T,u_0,R_s), \\
&\p \dot{\vartheta} \p_{s-1} \leq C(T,u_0,R_s), \, 
\p \ddot{\zeta} \p_{s-1} \leq C(T,u_0,R_s).
\end{split}
\label{step_zero}
\end{align}
Set $\eta_0 = \beta_0 = \zeta$, $\widehat{u}_0 = \vartheta$,
$\nabla f_0 = z_0 = w_0 = \nabla h_0 = 0$.
Minor adjustments have to be made at the first step, since $f_0$ here 
is not obtained from theorem \ref{theo_f_eq}, and $\dot{\eta}_0(0) = P u_0$ rather
than $u_0$, but it is clear that this does not hinder the construction of the sequence
$\{ \eta_n \}$.

\subsection{Convergence of the approximating sequences}
Denote by $W^{k,\infty}([0,T], H^s(\Om))$ the usual Sobolev space
of $H^s(\Om)$-valued functions on $[0,T]$ whose derivatives up to order $k$
are essentially bounded with respect to the $H^s(\Om)$ topology.
$W^{k,2}([0,T], H^s(\Om))$ is similarly defined using the $L^2$ inner 
product respect to $t \in [0,T]$.

We start by establishing some further bounds. In  the arguments below,
the particular form of some of the expressions involved will be omitted for the
sake of simplicity, since such expressions are cumbersome. The
relevant information will be the derivative counting.

Differentiating (\ref{eq_f_bry}) in time, invoking (\ref{bilinear}), recalling that $s > \frac{3}{2} + 2$,
and using our bounds on $f$, $v$, and $\widetilde{q}_0$, we find
that $\{  \nabla \dddot{f}_n \} $ is bounded in $H^{s-3}(\Om)$ (with a bound depending on $\kappa$).

From step $(n4)$ of the inductive construction, the function $h_n(t)$ satisfies
\begin{gather}
\begin{cases}
\Delta h_{n+1} = 0, & \text{ in } \underset{0 \leq t \leq T}{\bigcup} \{t\} \times  (\id + \nabla f_{n+1})(\Om), \\
\frac{\partial h_{n+1}}{\partial \widetilde{N}_{n+1}} =
\langle (\nabla \dot{f}_{n+1} + D_{v_{n+1}} \nabla f_{n+1} + v_{n+1} ) \circ 
(\id + \nabla f_{n+1})^{-1},
\widetilde{N}_{n+1} \rangle & \text{ on } 
\underset{0 \leq t \leq T}{\bigcup} \{t\} \times \partial(\id + \nabla f_{n+1}) (\Om),
\end{cases}
\nonumber
\end{gather}
Differentiating twice, we see that we can bound $\p \nabla \ddot{h}_n \p_{s-3}$
in terms of: $\p \nabla \dddot{f} \p_{s-3},$ (which was just estimated) together with
a constant depending on the domain $(\id + \nabla f_{n+1}) (\Om)$, which was handled as in section \ref{section_suc_app};
and other quantities that have already been bounded. We obtain therefore 
an $H^{s-3}$ bound for the sequence 
$\{  \nabla \ddot{h}_n \} $. 

A similar argument using the equation for $z_n$ in step $(n5)$ of the inductive
construction shows that we can bound $\{ \ddot{z}_n \}$ in $H^{s-3}$.

We now establish the convergence. It will be implicit in the arguments that
we will be seeking convergence of some sub-sequence. \\

\noindent \emph{Convergence of $\{ \eta_n \}$:}
From (\ref{induction_bound_1}), (\ref{induction_bound_3}) and (\ref{induction_bound_5}), we find that
the sequence $\{\eta_n\}$ is bounded in\linebreak $W^{2,2}([0,T], H^{s-\frac{3}{2}}(\Om))$, and 
thus it has a subsequence, still denoted $\{\eta_n\}$, converging weakly to a limit
$\eta$. Also, $\eta_n$ and $\dot{\eta}_n$ are bounded in $L^\infty([0,T],H^s(\Om))$,
thus they have a weakly convergence subsequence. 
We conclude that $\eta \in W^{1,\infty}([0,T], H^s(\Om))$. In particular, 
$\eta \in C^0([0,T], H^s(\Om))$. A similar argument yields $\ddot{\eta} \in 
L^\infty([0,T],H^{s-\frac{3}{2}}(\Om))$, from which we find that
$\eta \in  W^{1,\infty}([0,T], H^s(\Om)) \cap W^{2,\infty}([0,T],H^{s-\frac{3}{2}}(\Om))$.

Furthermore, $\p \eta_n(t) - \eta_n(t^\prime) \p_s \leq \Rin |t - t^\prime|$ in view
of (\ref{induction_bound_3}), and $\{ \eta_n \}$ has compact closure in 
$H^{s-1}(\Om)$ because of the the compactness of the embedding 
$H^s \subset H^{s-1}$. Hence, by the Arzel\`a-Ascoli theorem, the convergence $\eta_n \rar \eta$
occurs in $C^0([0,T],H^{s-1}(\Om))$. Now, since $\{ \eta_n \}$ is bounded in $H^s$,
interpolating between $H^s$ and $H^{s-1}$ shows that 
 $\eta_n \rar \eta$
in $C^0([0,T],H^{s-\de}(\Om))$, where $\de>0$ is some fixed small number.
A similar argument using (\ref{induction_bound_5})
gives that $\dot{\eta}_n \rar \dot{\eta}$
in $C^0([0,T],H^{s-2}(\Om))$. After interpolation, we have in fact
$\dot{\eta}_n \rar \dot{\eta}$
in $C^0([0,T],H^{s-\de}(\Om))$. Therefore, $\eta_n \rar \eta$ in 
$C^1([0,T],H^{s-\de}(\Om)))$.

Next, from the definition of $\eta_n$, the inductive bounds and the bounds  established above
on $\ddot{z}_n$, $\nabla \ddot{h}_n$, and $\nabla \ddot{f}$, we find that 
$\{ \dddot{\eta}_n \}$ is bounded in $H^{s-3}$. Thus, invoking once more 
the Arzel\`a-Ascoli theorem, we find that $\ddot{\eta}_n$ converges 
in $C^0([0,T], H^{s-3-\de}(\Om))$, where  $\de>0$ is a fixed small number.
By interpolation and (\ref{induction_bound_5}), $\ddot{\eta}_n$ converges also in 
in $C^0([0,T], H^{s-\frac{3}{2}-\de}(\Om))$.

Summarizing:

\noindent $\bullet$ $\eta_n \rar \eta$ in 
$C^1([0,T],H^{s-\de}(\Om)) \cap C^2([0,T], H^{s-\frac{3}{2}-\de}(\Om))$, and
\begin{gather}
\eta \in  W^{1,\infty}([0,T], H^s(\Om)) \cap W^{2,\infty}([0,T],H^{s-\frac{3}{2}}(\Om)).
\nonumber
\end{gather}

\noindent \emph{Convergence of $\{ \beta_n \}$:}
We can repeat the same argument with the sequence $\{\beta_n \}$ and conclude that it has a limit
in $ W^{1,\infty}([0,T], H^s(\Om)) \cap W^{2,\infty}([0,T],H^{s-\frac{3}{2}}(\Om))$.
Because $\cD_\mu^s(\Om)$ is a closed submanifold of $H^s(\Om)$, we have in fact
$\beta \in   W^{1,\infty}([0,T], \cD_\mu^s(\Om)) \cap W^{2,\infty}([0,T],H^{s-\frac{3}{2}}(\Om)).
$
Because of (\ref{ind_bounds_1}) and (\ref{ind_bounds_2}), we
can as before use the Arzel\`a-Ascoli theorem and interpolation inequalities
to conclude that 

\noindent $\bullet$ $\beta_n \rar \beta$ in
$C^1([0,T],H^{s-\de}(\Om))$ and
\begin{gather}
\beta \in   W^{1,\infty}([0,T], \cD_\mu^s(\Om)) \cap W^{2,\infty}([0,T],H^{s-\frac{3}{2}}(\Om)).
\nonumber
\end{gather}

\noindent \emph{Convergence of $\{ \nabla f_n \}$:}
Recalling that bounds on $\p f_n \p_{s,\partial}$ translate into
bounds on $\p f_n \p_{s+\frac{1}{2}}$ (see estimate (\ref{elliptic_estimate_f_bry})),
we apply  an analogous argument to the sequence $\{ \nabla f_n \}.$ 
Using the bounds and regularity given by theorem \ref{theo_f_eq},
the sequence is bounded in $W^{2,2}([0,T], H^{s-\frac{3}{2}}(\Om)).$
Also $\{ \nabla f_n \}$, $\{ \nabla \dot{f}_n \}$, and
$\{ \nabla \ddot{f}_n \}$ are bounded in $L^\infty([0,T], H^{s+\frac{3}{2}}(\Om))$,
$L^\infty([0,T], H^{s}(\Om))$, and $L^\infty([0,T], H^{s-\frac{3}{2}}(\Om))$, 
respectively. Therefore, $\nabla f_n$ converges weakly in 
$W^{2,2}([0,T], H^{s-\frac{3}{2}}(\Om))$ to a limit $\nabla f$ which is
 in
$L^{\infty}([0,T], H^{s+\frac{3}{2}}(\Om)) \cap  W^{1,\infty}([0,T], H^s(\Om)) \cap W^{2,\infty}([0,T],H^{s-\frac{3}{2}}(\Om)).
$
Furthermore,  in light of  
the previously obtained bound on $\nabla \dddot{f}_n$, we have $\p \nabla \ddot{f}(t) - \nabla \ddot{f}(t^\prime) \p_{s-3} \leq C|t-t^\prime|$. Hence, as before, combining 
the  Arzel\`a-Ascoli theorem with interpolation inequalities gives \\

\noindent $\bullet$ $\nabla f_n \rar \nabla f$ in 
$C^0([0,T],H^{s+\frac{3}{2} - \de} (\Om)) \cap C^1([0,T], H^{s-\de}(\Om)) \cap
C^2([0,T], H^{s-\frac{3}{2} - \de}(\Om))$ and 
\begin{gather}
\nabla f \in L^{\infty}([0,T], H^{s+\frac{3}{2}}(\Om)) \cap  W^{1,\infty}([0,T], H^s(\Om)) \cap W^{2,\infty}([0,T],H^{s-\frac{3}{2}}(\Om)).
\nonumber
\end{gather}

The other quantities appearing in (\ref{def_eta_n_plus_1}) are handled in a similar fashion.
We obtain: \\

\noindent $\bullet$ $\widehat{u}_n \circ (\id +\nabla f_n )\circ \beta_n \rar 
\widehat{u} \circ (\id + \nabla f) \circ \beta$ in $C^0([0,T], H^{s-\de}(\Om))$ and
\begin{gather}
\widehat{u} \circ (\id + \nabla f) \circ \beta \in L^\infty([0,T],H^s(\Om)) \cap
W^{1,\infty}([0,T],H^{s-\frac{3}{2}}(\Om)).
\nonumber
\end{gather}
\noindent $\bullet$ $\widetilde{q}_{0,n} \rar \widetilde{q}_0$ in $C^0([0,T], H^{s+1-\de}(\Om))$
and
\begin{gather}
\widetilde{q}_0 \in L^\infty([0,T],H^{s+1}(\Om)) \cap
W^{1,\infty}([0,T],H^{s-\frac{1}{2}}(\Om)).
\nonumber
\end{gather}
\noindent $\bullet$ $\nabla h_n \circ (\id +\nabla f_n )\circ \beta_n \rar 
\nabla h \circ (\id + \nabla f) \circ \beta$ in $C^0([0,T], H^{s-\de}(\Om))$ and
\begin{gather}
\nabla h \circ (\id + \nabla f) \circ \beta \in L^\infty([0,T],H^s(\Om)) \cap
W^{1,\infty}([0,T],H^{s-\frac{3}{2}}(\Om)).
\nonumber
\end{gather}
\noindent $\bullet$ $w_n \circ (\id +\nabla f_n )\circ \beta_n \rar 
\nabla w \circ (\id + \nabla f) \circ \beta$ in $C^0([0,T], H^{s-\de}(\Om))$ and
\begin{gather}
w \circ (\id + \nabla f) \circ \beta \in L^\infty([0,T],H^s(\Om)) \cap
W^{1,\infty}([0,T],H^{s-1}(\Om)).
\nonumber
\end{gather}
Notice that all the limit quantities satisfy the same bounds as the corresponding
sequences. In particular, $\nabla f$ satisfies the bounds given by theorem \ref{theo_f_eq}.

\subsection{Solution\label{solution_large_kappa}}
With the above information, we can pass to the limit in (\ref{def_eta_n_plus_1}) obtaining
\begin{gather}
\eta = \id + \int_0^t( w + \nabla h) \circ (\id + \nabla f)\circ  \beta.
\nonumber
\end{gather}
$\eta$ is volume-preserving and its velocity $u$ given by $\dot{\eta} = u \circ \eta$
agrees with $\widehat{u}$. Also, 
$\eta$ necessarily has the form $\eta = (\id + \nabla f) \circ \beta$
and $f$ and $\beta$ have the above regularity properties. 
In particular $\eta \in \ccE^s_\mu(\Om)$. 
Moreover, in light of the 
way $w$ and $\nabla h$ were constructed (see section \ref{overview}),
and the previously established convergences, $u$ satisfies
\begin{gather}
P(\frac{\partial u}{\partial t} + \nabla_u u ) = 0.
\label{P_Euler}
\end{gather}
We also know that $\nabla f$ satisfies (\ref{system_f_v_f_dot_dot}), with
$p$ satisfying  
$p \circ (\id +\nabla f)  = p_0 \circ (\id +\nabla f)  + \cA_H \circ (\id +\nabla f)$,
where $p_0 \circ (\id +\nabla f)  = \widetilde{q}_0$ 
(compare with (\ref{split_pressure}) and (\ref{free_boundary})), so in particular 
$p = \cA$ on $\partial \eta(\Om)$. From the above convergence, we know that 
$p_0$ is in $H^{s+1}(\Om(t))$ and $\cA$, being third order in $\left. f \right|_{\partial \Om}$,
is in $H^{s-1}(\partial \Om(t))$, so that $\cA_H \in H^{s-\frac{1}{2}}( \Om(t))$,
and hence $p \in H^{s-\frac{1}{2}}( \Om(t))$.

It remains to show that 
(\ref{free_boundary_full}) is satisfied, which is not immediately obvious since
we did not solve (\ref{u_Eulerian}) in the iteration, but rather
$P$ of that equation, i.e., (\ref{P_Euler}). We will now show that
$Q$ of (\ref{u_Eulerian}) follows from the equation for $f$ that we solved.

Equation (\ref{P_Euler}) gives
$\frac{\partial u}{\partial t} + \nabla_u u = Q( \frac{\partial u}{\partial t} + \nabla_u u )$, so there exists a function $\chi$ such that 
\begin{gather}
\frac{\partial u}{\partial t} + \nabla_u u = \nabla \chi.
\nonumber
\end{gather}
We need to show that $\chi = -p$, where $p$ is as in (\ref{free_boundary}).
As $\ddot{\eta} \circ \eta^{-1} = \frac{\partial u}{\partial t} + \nabla_u u$,
we have 
\begin{gather}
\ddot{\eta} = \nabla \chi \circ \eta.
\label{eta_dot_dot_chi}
\end{gather}
On the other hand (\ref{eta_decom_exp}) also holds, and because it is a fixed
point of the above iteration, $f$ in this case does satisfy 
(\ref{system_f_v_f_dot_dot}). Therefore, differentiating (\ref{eta_decom_exp})
in time twice, decomposing according to (\ref{decomp}) and  using 
(\ref{system_f_v_f_dot_dot}) and (\ref{eta_dot_dot_chi}) gives
(compare with (\ref{equation_ddot_nabla_f})):
\begin{gather}
Q(\id - L L_1^{-1} P )(\nabla \chi \circ \widetilde{\eta} )
= 
Q(\id - L L_1^{-1} P )(- \nabla p \circ \widetilde{\eta} ),
\nonumber
\end{gather}
or
\begin{gather}
Q(\id - D^2f L_1^{-1} P )(\nabla \chi \circ \widetilde{\eta} )
= 
Q(\id - D^2 f L_1^{-1} P )(- \nabla p \circ \widetilde{\eta} ),
\nonumber
\end{gather}
since $QP=0$. We can write this as
\begin{gather}
\left( Q_{\widetilde{\eta}^{-1}} - (D^2 f L_1^{-1} P)_{\widetilde{\eta}^{-1}} \right)
(\nabla \chi ) 
=
\left( Q_{\widetilde{\eta}^{-1}} - (D^2f L_1^{-1} P)_{\widetilde{\eta}^{-1}} \right)
(-\nabla p ). 
\nonumber
\end{gather}
Since $Q$ is the identity on its image, for $f$ small
the operator
\begin{gather}
Q_{\widetilde{\eta}^{-1}} - (D^2f L_1^{-1} P)_{\widetilde{\eta}^{-1}}
\nonumber
\end{gather}
is invertible on $Q( H^s(\widetilde{\eta} (\Om) ))$, and therefore
$\nabla \chi = -\nabla p$, as desired.

From the regularity of $f$ and we have
$\cA_H \circ (\id+\nabla f) \circ \beta \in L^\infty( [0,T], H^{s-\frac{1}{2}}(\Om))$,
and hence $p \in L^\infty( [0,T], H^{s-\frac{1}{2}} (\Om(t) ) )$. As 
$\left. p \right|_{\partial \Om(t)} = \cA$, 
we conclude that there exists 
\begin{gather}
\eta \in  W^{1,\infty}([0,T], H^s(\Om)) \cap W^{2,\infty}([0,T],H^{s-\frac{3}{2}}(\Om))
\nonumber
\end{gather}
satisfying (\ref{free_boundary_full}), as desired. Uniqueness of $\eta$ now follows from the uniqueness
of the decomposition $\eta = (\id + \nabla f)\circ \beta$ given by the exponential 
map near the identity. This implies uniqueness of $p$ in view of
(\ref{equation_p_0}), (\ref{bry_p_0}), (\ref{harmonic_ext_def}), and
(\ref{equation_harm_ext}).

To finish the existence part of  theorem \ref{main_theorem}, we point out that
 $\eta \in  W^{1,\infty}([0,T], H^s(\Om))$ implies 
$\eta \in C^0([0,T], H^s(\Om))$, and thus $\eta$ has the stated regularity.

\subsection{Proof of theorem \ref{main_theorem}: existence\label{existence_proof}}
Define a new time variable by $t = a \tau$, where $a>0$ is a constant that will be chosen. 
Define $\eta_a$ by
$\eta_a(\tau) = \eta(t)$, i.e., $\eta_a(\tau) = \eta(a\tau)$. Then $\ddot{\eta}_a(\tau) = a^2 \ddot{\eta}(t) \equiv a^2 \ddot{\eta}(a\tau)$.
Using the equation for $\eta$, i.e., (\ref{basic_fluid_motion_full}), we have 
$\ddot{\eta}_a(\tau) = -a^2 \nabla p(t) \circ \eta(t) \equiv -a^2 \nabla p(a \tau) \circ \eta(a\tau)$. So if we define $p_a(\tau) = p(t)$, 
i.e., $p_a(\tau) = p(a\tau)$, we obtain $\ddot{\eta}_a(\tau) = -a^2 \nabla p_a(\tau) \circ \eta_a(\tau)$.
Then letting $\pi_a(\tau) = a^2 p_a(\tau)$, we finally obtain
\begin{gather}
\ddot{\eta}_a(\tau) = - \nabla \pi_a(\tau) \circ \eta_a(\tau).
\label{eq_eta_new}
\end{gather}
Multiplying  (\ref{bry_p_full}) by $a^2$ gives $a^2 p(t) \equiv a^2 p(a\tau) \equiv a^2 p_a(\tau) \equiv \pi_a(\tau)
= a^2 \kappa \cA(t) \equiv a^2 \kappa \cA(a\tau)$. Thus, if we define $\cA_a(\tau) = \cA(t) \equiv \cA(a\tau)$ 
and $\kappa_a = a^2 \kappa$, we have
\begin{gather}
\pi_a(\tau) = \kappa_a \cA_a(\tau) \, \text{ on } \, \eta_a(\tau)(\partial \Om).
\label{eq_p_new}
\end{gather}
Equations (\ref{eq_eta_new}) and (\ref{eq_p_new}) are of the form (\ref{free_boundary_full}) (with $u_a(\tau)$ defined
accordingly) with a coefficient of surface tension given by $\kappa_a$. If $\kappa > 0$ is fixed, not necessarily large,
we can then choose $a^2$ large enough so that $\kappa_a$ is sufficiently large as to apply 
the result of section \ref{solution_large_kappa}, provided the other assumptions
can also be accommodated. This is discussed below.
We therefore obtain a solution $(\eta_a,\pi_a)$. Reverting back to the original variables, this yields a solution to the original 
problem for $\eta$ and $p$ with a given $\kappa >0$.

The result in section \ref{solution_large_kappa}
assumes that $\partial \Om$ has constant mean curvature. We shall 
show that if we are 
interested only in part (1) of theorem \ref{main_theorem},
 this assumption can be removed as well.

First, without such an assumption, we do not necessarily have  
$\cA_{\partial \Om} > 0$ so the proof of lemma \ref{lemma_elliptic_pseudo_2}
has to be altererd. Before we used $\cA_{\partial \Om} > 0$
to show the positivity and invertibility of  
$-\overline{\Delta} \partial_\nu - \frac{1}{2} \cA_{\partial \Om} \overline{\Delta}$.
We then used this result in proposition \ref{prop_first_order} to construct an evolution
operator associated with $A_\kappa(t)$ (see equation (\ref{first_order_simple})).
In the present case, we consider the operator  $-\overline{\Delta} \partial_\nu$
instead of $-\overline{\Delta} \partial_\nu - \frac{1}{2} \cA_{\partial \Om} \overline{\Delta}$.
As discussed in lemma lemma \ref{lemma_elliptic_pseudo_2},  $-\overline{\Delta} \partial_\nu$
is positive and invertible, so we still obtain the operator $\cS$ (see
(\ref{cS_defi})). This of course gives and extra term 
in the operator $A_\kappa(t)$, namely,  
$\frac{\sqrt{\kappa}}{2} \cA_{\partial \Om} \overline{\Delta} \cS^{-1}$. But this will be a 
bounded operator and therefore we still obtain an evolution operator 
from theorem \ref{theorem_evolution_operator} (see the last statement of theorem
\ref{theorem_generator_semi_group}). 
 
Second, when the mean curvature of $\Om$,
$\cA_{\partial \Om}$, is not constant, then  the equation for $f$ will contain the additional term
$\kappa \cA_{\partial \Om}$ (see equations (\ref{eq_f_bry}), (\ref{F_tilde_operator}) 
and remark \ref{remark_constant_mean_curvature}). This extra term
is simply an extra inhomogeneous term that can be absorbed into $\cG.$ 
 (see equation (\ref{first_order_simple})).

The result in section \ref{solution_large_kappa} also 
assumes $\nabla \dot{f}(0)$ to be small, i.e., the gradient part of the initial 
velocity, $Q u_0$, has to be small (we do not have to worry about $\nabla f(0)$ being small since $\nabla f(0) = 
0$). We now show how this assumption can be removed.

The assumption that $\nabla \dot{f}(0)$ is small is used in step $(n3)$ 
of section \ref{section_suc_app}
to guarantee that $\dot{f}(0) = \left. \Delta_\nu^{-1} \dive u_0 \right|_{\partial \Om}$ is small.
That $\dot{f}(0)$ is small is used in theorem \ref{theo_f_eq} (see
(\ref{f_1_hypothesis})) in order to obtain estimate (\ref{estimate_using_dot_f_small}).
However (\ref{estimate_using_dot_f_small}) still holds if $\dot{f}(0)$
is not small (we explain below why the ensuing argument was not used in the proof of
theorem \ref{theo_f_eq}). I.e., in theorem \ref{theo_f_eq}, assume that instead 
of (\ref{f_1_hypothesis}) we have
\begin{gather}
\p f_1 \p_{s+\frac{1}{2},\partial} \leq K_3.
\label{new_f_1_hypothesis}
\end{gather}
In what follows, we continue to assume that $\kappa$ is large since,
as showed above, the problem for arbitrary $\kappa > 0$ can be reduced to that of large 
$\kappa$ via a rescaling.

As in the proof of theorem \ref{theo_f_eq}, we invoke (\ref{z_U_cG}).
The estimate of the term $\int_0^t \cU(t,\tau) \cG(\tau)$ does not rely on 
(\ref{f_1_hypothesis}), so this term yields
$\frac{C(K_0) T}{\sqrt{\kappa}} \sup_{0\leq \tau \leq T} \p \cG( \tau ) \p_{s+\frac{1}{2},\partial}$
as before. 

For the term $\cU(t,0) z(0)$, first notice that (\ref{estimate_using_dot_f_small})
corresponds to the first component of $z$, i.e., $z_1$. Write
\begin{gather}
\cU(t,0) z(0) = 
\left(
\begin{matrix}
\cU_{11}(t,0) & \cU_{12}(t,0) \\
\cU_{21}(t,0) & \cU_{22}(t,0)
\end{matrix}
\right)
\,
\left(
\begin{matrix}
0 \\
f_1
\end{matrix}
\right),
\nonumber
\end{gather}
where $z(0) = (0,f_1)$. The first component of the above is 
$\cU_{12}(t,0) f_1$. Recall that $\cU(0,0) = I$, so $\cU_{12}(0,0)=0$, and also $\cU(t,\tau)$ is strongly continuous
into $Y = H^{s+\frac{1}{2}}_0(\partial \Om)$ (see theorem \ref{theorem_evolution_operator}
and section \ref{section_f_boundary}). Thus, with $f_1$ and $\kappa$ given, we can choose
$T$ (and hence $t$) so small that
\begin{gather}
\p \cU_{12}(t,0) f_1 \p_{s+\frac{1}{2},\partial } \leq  \frac{C(K_0)  K_3}{\kappa}.
\nonumber
\end{gather}
Therefore, estimate (\ref{estimate_using_dot_f_small}) still holds without the assumption
that $\nabla \dot{f}(0)$ is small.

The above argument was not used in the proof of theorem 
\ref{theo_f_eq} because it produces a time interval
$[0,T_\kappa)$ that shrinks to zero as
$\kappa \rar \infty$, so the corresponding existence result and estimates
would not apply to the limit
$\kappa \rar\infty.$

The other part in the proof of theorem \ref{theo_f_eq} where (\ref{f_1_hypothesis})
has been employed was (\ref{estimate_h_dot_cM}). It is clear, however, that the argument
following (\ref{estimate_h_dot_cM}) still holds if (\ref{new_f_1_hypothesis})
replaces (\ref{f_1_hypothesis}). Indeed, under (\ref{new_f_1_hypothesis}),
(\ref{estimate_h_dot_cM}) becomes
\begin{align}
\begin{split}
\p \dot{h} \p_{s+\frac{1}{2},\partial} & \leq C(K_0) 
 \p f_1 \p_{s+\frac{1}{2},\partial}
+ C(K_0) T \sup_{0\leq \tau \leq T} \p \cG( \tau ) \p_{s+\frac{1}{2},\partial} \\
& \leq C(K_0) K_3
+C(K_0) T \sup_{0\leq \tau \leq T} \p \cG( \tau ) \p_{s+\frac{1}{2},\partial},
\end{split}
\label{new_estimate_h_dot_cM}
\end{align}
and we can still choose $\ell$ large enough so that the right hand side of
(\ref{new_estimate_h_dot_cM}) is less than $\ell$.

Finally, (\ref{estimate_h_dot_cM}) was also invoked when we derived estimate
(\ref{estimate_difference_As_3}). But again, 
the $\frac{1}{\kappa}$ factor is not needed here since we only need the 
right hand side of (\ref{estimate_h_dot_cM}) to be bounded in order to obtain
(\ref{estimate_difference_As_3}), which is the case in light of 
(\ref{new_estimate_h_dot_cM}) (see the paragraph after (\ref{estimate_difference_As_3})).

Inspection in the proof leading to the existence part
in section \ref{solution_large_kappa}
shows that the remaining
arguments are the same without the assumption that $Q u_0$ is small. 
This establishes the proof of theorem \ref{main_theorem}.

\begin{remark}
Notice that these arguments are consistent, in the following sense. We obtain a solution that exists for 
$0 \leq \tau \leq \cT$, or, in the $t$ variable, $0 \leq \frac{1}{a} t \leq \cT$, i.e, $0 \leq  t \leq a \cT$.  So, if we take the limit
$a \rar 0$, so that $\kappa_a \rar 0$, the interval of existence shrinks to zero, as it should since the problem is not well
posed when $\kappa_a = 0$ and the Taylor sign condition, which we do not assume, does not hold
 (it turns out that $\cT$ also depends on $a$, so this idea of consistency with $a \rar 0$
is more complicated than just stated, but on a heuristic level we see that we obtain what is expected).
\end{remark}

\section{Proof of theorem \ref{main_theorem}: convergence \label{section_convergence}}

Here we establish the convergence part of theorem \ref{main_theorem}, and thus
we assume the corresponding hypotheses and notations throughout.
Some of  arguments below resemble those of 
theorem 5.1 in \cite{ED}
and theorem 5.5 in \cite{E2}.
From now on  it is convenient to re-instate the subscript $\kappa$.

Let $[0,T_\kappa)$ be the maximal interval of existence for the solution
$(\eta_\kappa, p_\kappa)$ found above. Let $\cT_\kappa = \min\{ T, T_\kappa \}$,
where we recall that $[0,T]$ is an interval on which the solution to (\ref{Euler})
is defined. We henceforth consider the quantities $\eta_\kappa$ and $\zeta$
on $[0,\cT_\kappa)$. Assume also that $T$ is chosen such that
(\ref{step_zero}) holds.

As in the calculations leading to (\ref{equation_ddot_nabla_f}), we differentiate
$\eta_\kappa$ twice in time to obtain
\begin{gather}
\ddot{\eta}_\kappa - \ddot{\be}_\kappa
= D^2 f_\kappa \circ \beta_\kappa \ddot{\be}_\kappa
+ (\nabla \ddot{f}_\kappa + 2 D_{v_\kappa} \nabla \dot{f}_\kappa
+ D^2_{v_\kappa v_\kappa} \nabla f_\kappa )\circ \be_\kappa,
\label{eta_minus_beta_dot_dot_1}
\end{gather}
where we used that $\dot{\be}_\kappa = v_\kappa \circ \beta_\kappa$, so 
that
\begin{gather}
\ddot{\be}_\kappa = ( \frac{\partial v_\kappa}{\partial t} 
+ \nabla_{v_\kappa} v_{\kappa} )\circ \be_\kappa.
\label{beta_dot_dot_Euler_type}
\end{gather}
Integrating (\ref{eta_minus_beta_dot_dot_1}):
\begin{gather}
\dot{\eta}_\kappa - \dot{\be}_\kappa
=u_{0\kappa} - Pu_{0\kappa} 
+ \int_0^t (\nabla \ddot{f}_\kappa + 2 D_{v_\kappa} \nabla \dot{f}_\kappa
+ D^2_{v_\kappa v_\kappa} \nabla f_\kappa )\circ \be_\kappa
 +\int_0^t D^2 f_\kappa\circ \beta_\kappa   \ddot{\be}_\kappa,
\label{eta_minus_beta_dot_1}
\end{gather}
where we used that $\dot{\eta}_\kappa(0) = u_0$ and 
$\dot{\beta}_\kappa(0) = P u_0$. Write (\ref{eta_minus_beta_dot_1})
as 
\begin{gather}
\dot{\eta}_\kappa - \dot{\be}_\kappa
= Q u_{0\kappa} + R_\kappa,
\label{eta_minus_beta_dot_2}
\end{gather}
where
\begin{gather}
R_\kappa = \int_0^t r_\kappa,
\label{def_R_kappa}
\end{gather}
with
\begin{gather}
r_\kappa = 
 (\nabla \ddot{f}_\kappa + 2 D_{v_\kappa} \nabla \dot{f}_\kappa
+ D^2_{v_\kappa v_\kappa} \nabla f_\kappa )\circ \be_\kappa
 + D^2 f_\kappa \circ \beta_\kappa  \ddot{\be}_\kappa.
\label{def_r_kappa}
\end{gather}
On the other hand, (\ref{dot_eta_f_beta}) gives
\begin{gather}
\dot{\eta}_\kappa - \dot{\beta}_\kappa = (\nabla \dot{f}_\kappa
+ D_{v_\kappa} \nabla f_\kappa ) \circ \beta_{\kappa},
\nonumber
\end{gather}
so that the estimates of section \ref{section_suc_app} give
\begin{align}
\begin{split}
\p \dot{\eta}_\kappa - \dot{\beta}_\kappa \p_s & \leq
\p \nabla \dot{f}_\kappa \circ \beta_{\kappa} \p_s 
+ \p ( D_{v_\kappa} \nabla f_\kappa ) \circ \beta_{\kappa} \p_s
\\
&
\leq C \p \nabla \dot{f}_\kappa \p_s (1+ \p \beta_{\kappa} \p_s^s )
+
C \p v_\kappa \p_s \p \nabla f_\kappa \p_{s+1} 
 (1+ \p \beta_{\kappa} \p_s^s  )
 \\
 &
\leq \frac{C(\Rin)}{\sqrt{\kappa}}  + \frac{C(\Rin)}{\kappa} \leq 
\frac{C(\Rin)}{\sqrt{\kappa}},
\end{split}
\label{estimate_eta_minus_beta_dot}
\end{align}
where $C(\Rin)$ is a constant depending on $\Rin$,  and $\Rin$ is
as in section \ref{section_suc_app}.
Combining (\ref{estimate_eta_minus_beta_dot}) with 
(\ref{eta_minus_beta_dot_2}) and our assumptions on $Qu_{0\kappa}$
leads to
\begin{gather}
\p R_\kappa \p_s \leq \frac{C(\Rin)}{\sqrt{\kappa}}.
\label{estimate_R_kappa}
\end{gather}
\begin{remark}
It is important to notice that (\ref{estimate_R_kappa})  follows from
the specific relation between $R_\kappa$ and $\dot{\eta}_\kappa - \dot{\be}_\kappa$, i.e., equation (\ref{eta_minus_beta_dot_2}).
In fact,  we could not try to estimate $R_\kappa$ term by term from
(\ref{def_R_kappa}) and (\ref{def_r_kappa}) since some of such terms, 
e.g. $\nabla \ddot{f}_\kappa$, may 
not even be in $H^s$.
\end{remark}
Equations (\ref{eta_minus_beta_dot_2}), (\ref{def_R_kappa}) and (\ref{def_r_kappa})
also give
\begin{gather}
\ddot{\eta}_\kappa - \ddot{\be}_\kappa = r_\kappa.
\label{eta_minus_beta_dot_dot_2}
\end{gather}
Recall that we denote by $\zeta$ the solution
to  (\ref{Euler}). From (\ref{eta_minus_beta_dot_2}) we have
\begin{gather}
\ddot{\zeta} - \ddot{\eta}_\kappa =
 \ddot{\zeta} - \ddot{\be}_\kappa - r_\kappa,
\nonumber
\end{gather}
so that
\begin{gather}
\dot{\zeta} - \dot{\eta}_\kappa =
\vartheta_0 -  u_{0\kappa}
+ 
\int_0^t( \ddot{\zeta} - \ddot{\be}_\kappa ) - R_\kappa,
\label{zeta_minus_eta_dot_with_beta_dot_dot}
\end{gather}
after recalling (\ref{def_R_kappa}).

It is well-known (see, e.g., \cite{EM}) that (\ref{Euler}) can be written
\begin{gather}
\ddot{\zeta} = Q( \nabla_\vartheta \vartheta ) \circ \zeta,
\nonumber
\end{gather}
and, since $\vartheta = \dot{\zeta} \circ \zeta^{-1}$, 
it follows that \begin{gather}
\ddot{\zeta} = (Q ( \nabla_{\dot{\zeta} \circ \zeta^{-1}} 
 \dot{\zeta} \circ \zeta^{-1} ) )  \circ \zeta.
\label{Euler_Lagrangian}
\end{gather}
On the other hand, from (\ref{beta_dot_dot_Euler_type}) we find,
using that $Q v_\kappa = 0$, that
\begin{gather}
Q (\ddot{\be}_\kappa \circ \be_\kappa^{-1} ) = Q(\nabla_{v_\kappa}
v_\kappa ),
\nonumber
\end{gather}
and thus
\begin{align}
\begin{split}
 \frac{\partial v_\kappa}{\partial t} 
+ \nabla_{v_\kappa} v_{\kappa} = 
\ddot{\be}_\kappa \circ \be_\kappa^{-1} = 
Q( \ddot{\be}_\kappa \circ \be_\kappa^{-1}  )
+ P( \ddot{\be}_\kappa \circ \be_\kappa^{-1}  )
=  Q(\nabla_{v_\kappa}
v_\kappa ) + P( \ddot{\be}_\kappa \circ \be_\kappa^{-1}  ).
\end{split}
\nonumber
\end{align}
Composing with $\be_\kappa$ and recalling (\ref{beta_dot_dot_Euler_type})
once more:
\begin{align}
\begin{split}
\ddot{\be}_\kappa 
= (Q(\nabla_{v_\kappa} v_\kappa )) \circ \be_\kappa
 + (P( \ddot{\be}_\kappa \circ \be_\kappa^{-1}  ) ) \circ \be_\kappa,
\end{split}
\nonumber
\end{align}
so (using notation \ref{notation_sub})
\begin{align}
\begin{split}
\ddot{\be}_\kappa 
= (Q(\nabla_{\dot{\be}_\kappa \circ \be_\kappa^{-1} } \dot{\be}_\kappa \circ \be_\kappa^{-1} )) \circ \be_\kappa
 + P_{\be_\kappa} \ddot{\be}_\kappa.
\end{split}
\label{beta_dot_dot_Euler_type_Lagrangian}
\end{align}
Using (\ref{Euler_Lagrangian}) and (\ref{beta_dot_dot_Euler_type_Lagrangian})
into (\ref{zeta_minus_eta_dot_with_beta_dot_dot}) produces
\begin{align}
\begin{split}
\dot{\zeta} - \dot{\eta}_\kappa
& = \vartheta_0 - u_{0\kappa} - R_\kappa 
- \int_0^t  P_{\beta_\kappa} \ddot{\be}_\kappa 
\\
& 
+ 
\int_0^t 
 \Big( (Q ( \nabla_{\dot{\zeta} \circ \zeta^{-1}} 
 \dot{\zeta} \circ \zeta^{-1} ) )  \circ \zeta
 -   (Q(\nabla_{\dot{\be}_\kappa \circ \be_\kappa^{-1} } \dot{\be}_\kappa \circ \be_\kappa^{-1} )) \circ \be_\kappa
 \Big ).
\end{split}
\label{difference_dot_flows}
\end{align}
The term $Q(\nabla_{\dot{\be}_\kappa \circ \be_\kappa^{-1} } \dot{\be}_\kappa \circ \be_\kappa^{-1} )) \circ \be_\kappa$ is in $H^s(\Om)$ because $\be_\kappa \in \cD_\mu^s(\Om)$ and so is
 $(Q ( \nabla_{\dot{\zeta} \circ \zeta^{-1}} 
\dot{\zeta} \circ \zeta^{-1} ) )  \circ \zeta$. Integration of 
(\ref{beta_dot_dot_Euler_type_Lagrangian}) shows that 
$\int_0^t  P_{\beta_\kappa} \ddot{\be}_\kappa$ is in $H^s(\Om)$ in that $\dot{\be}_\kappa \in H^s(\Om)$.
We can,  therefore, estimate (\ref{difference_dot_flows}) in $H^s$ and we proceed to do so.

Denote by $T\cD_\mu^s(\Om)$ the tangent bundle of $\cD_\mu^s(\Om)$ and let $Z$ be the map
\begin{align}
\begin{split}
& Z: T\cD_\mu^s(\Om) \rar H^s(\Om, \RR^3), \\
& Z(\xi,X) = ( Q(\nabla_{X\circ \xi^{-1}} X\circ \xi^{-1} ) ) \circ \xi.
\end{split}
\nonumber
\end{align} 
In \cite{EM} it is shown that $Z$ is a smooth map. Since the image of $(\zeta,\dot{\zeta})$ is compact, 
$Z$ is uniformly Lipschitz in a neighborhood 
of $(\zeta, \dot{\zeta})$. Thus
\begin{gather}
\p Z(\zeta, \dot{\zeta} ) - Z(\be_\kappa, \dot{\be}_\kappa ) \p_s
\leq C ( \p \zeta - \be_\kappa \p_s + \p \dot{\zeta} -\dot{\be}_\kappa \p_s
).
\label{Lipschitz_zeta_beta}
\end{gather}
Combining (\ref{estimate_R_kappa}), (\ref{difference_dot_flows}), (\ref{Lipschitz_zeta_beta}), and
our assumptions on $u_{0\kappa}$,
it follows that
\begin{align}
\begin{split}
\p \dot{\zeta} - \dot{\eta}_\kappa \p_s
& \leq \frac{C(\Rin)}{\sqrt{\kappa}}(1+t)
+ 
\p \int_0^t  P_{\beta_\kappa} \ddot{\be}_\kappa 
\p_s
+ 
C(1+t) \int_0^t \p \dot{\zeta} - \dot{\eta}_\kappa \p_s,
\end{split}
\label{estimate_to_iterate}
\end{align}
where we used the fact that the term $\p \zeta - \be_\kappa \p_s$ in 
(\ref{Lipschitz_zeta_beta}) can be estimated in terms of
$\p \dot{\zeta}(0) - \dot{\be}_\kappa(0) \p_s + \p \dot{\zeta} - \dot{\be}_\kappa \p_s$ because of the fundamental theorem of calculus. We also used the fact that 
$\p \dot{\zeta} - \dot{\be}_\kappa \p_s \leq \p \dot{\zeta} - \dot{\eta}_\kappa \p_s 
+ \frac{C}{\sqrt{\kappa}}$. Thus, iterating (\ref{estimate_to_iterate}),
 \begin{align}
\begin{split}
\p \dot{\zeta} - \dot{\eta}_\kappa \p_s
& \leq  \Big( \frac{C(\Rin)}{\sqrt{\kappa}}(1+t)
+ 
\p \int_0^t  P_{\beta_\kappa} \ddot{\be}_\kappa  \p_s
\Big ) e^{C(t+t^2)}.
\end{split}
\label{basic_estimate_zeta_eta_dot}
\end{align}
Integrating (\ref{beta_dot_dot_Euler_type_Lagrangian}) and using the estimates
of section \ref{section_existence} yields
\begin{gather}
\p \int_0^t  P_{\beta_\kappa} \ddot{\be}_\kappa \p_s
\leq C(\Rin).
\label{bound_int_P_beta_beta_dot_dot}
\end{gather}
By the definition of $\cT_\kappa$, $\p \dot{\zeta} \p_s$ remains uniformly 
bounded on $[0,\cT_\kappa)$, and therefore the same holds
for  $\p \dot{\eta}_\kappa \p_s$ in light of (\ref{basic_estimate_zeta_eta_dot})
and (\ref{bound_int_P_beta_beta_dot_dot}). We conclude that $[0,\cT_\kappa)$
is not the maximal interval of existence of the solution of (\ref{free_boundary_full}),
and therefore $T_\kappa > T$. Moreover, this conclusion holds for all $\kappa$
sufficiently large since  (\ref{basic_estimate_zeta_eta_dot})
and (\ref{bound_int_P_beta_beta_dot_dot}) hold for all $\kappa$ sufficiently large.

Next we proceed to show convergence. First we estimate the term
$P_{\beta_\kappa} \ddot{\be}_\kappa $ in $H^{s-\frac{3}{2}}$.
Using (\ref{beta_dot_dot_Euler_type}) and (\ref{system_f_v_v_dot})
we get
\begin{align}
\begin{split}
 P( \ddot{\be}_\kappa \circ \beta_\kappa^{-1}) &+ L_{1\kappa}^{-1}P (
  2 D_{v_\kappa} \nabla\dot{f}_\kappa + D^2_{v_\kappa v_\kappa} \nabla f_\kappa ) \\
  &
  + L_{1\kappa}^{-1} P( L_\kappa Q( \nabla_{v_\kappa} v_\kappa) ) 
 = -L_{1\kappa}^{-1} P  (\nabla p_\kappa\circ (\id + \nabla f_\kappa) ). 
 \end{split}
\label{p_beta_ddot_eq}
\end{align}
All terms in (\ref{p_beta_ddot_eq}) are in
$H^{s-\frac{3}{2}}(\Om)$, and we have
\begin{gather}
\p  L_{1\kappa}^{-1}P (
  2 D_{v_\kappa} \nabla\dot{f}_\kappa + D^2_{v_\kappa v_\kappa} \nabla f_\kappa )  \p_{s-\frac{3}{2}}
  \leq \frac{C(\Rin)}{\sqrt{\kappa}},
\nonumber
\end{gather}
and
\begin{gather}
\p L_{1\kappa}^{-1} P( L_\kappa Q( \nabla_{v_\kappa} v_\kappa) )  \p_{s-\frac{3}{2}} \leq \frac{C(\Rin)}{\kappa},
\nonumber
\end{gather}
after using $PQ = 0$.

Writing $q_\kappa = p_\kappa \circ (\id + \nabla f_\kappa)$, we have
\begin{gather}
\nabla q_\kappa = \nabla p_\kappa \circ (\id + 
\nabla f_\kappa) (\id + D^2 f_\kappa),
\nonumber
\end{gather}
so
\begin{gather}
\nabla p_\kappa \circ (\id + \nabla f_\kappa) = 
\nabla q_\kappa + O(D^2 f_\kappa) \nabla q_\kappa.
\nonumber
\end{gather}
Thus,
\begin{gather}
P(  \nabla p_\kappa \circ (\id + \nabla f_\kappa) ) = P (  O(D^2f_\kappa) 
\nabla q_\kappa ),
\nonumber
\end{gather}
so that 
\begin{gather}
\p P(  \nabla p_\kappa \circ (\id + \nabla f_\kappa) ) \p_{s-\frac{3}{2}} \leq 
\frac{C(\Rin)}{\kappa}.
\nonumber
\end{gather}
These estimates and  (\ref{p_beta_ddot_eq}) imply
\begin{gather}
\p   P_{\beta_\kappa} \ddot{\be}_\kappa \p_{s-\frac{3}{2}}
\leq \frac{C(\Rin)}{\sqrt{\kappa}}.
\label{estimate_P_lower_norm}
\end{gather}
Thus (\ref{basic_estimate_zeta_eta_dot}) and (\ref{estimate_P_lower_norm})
imply that 
$\p \dot{\zeta} - \dot{\eta}_\kappa \p_{s-\frac{3}{2}} \rar 0$
as $\kappa \rar\infty$. The 
interpolation inequality can now be invoked to conclude that 
$\p \dot{\zeta} - \dot{\eta}_\kappa \p_{s-\de} \rar 0$
for any fixed $\de > 0$. We thus also conclude that 
that $\p \int_0^t  P_{\beta_\kappa} \ddot{\be}_\kappa \p_{s-\de} \rar 0$.

Let $J_\si$ be a standard Friedrichs mollifier 
in $\Om$. Recall that if $h$ is defined in $\Om$, $J_\si h$ is given by
 (see, e.g., \cite{TaylorPDE3})
\begin{gather}
J_\si h = R \widetilde{J}_\si E h,
\label{mollifier_bounded_domain}
\end{gather}
where $E: H^r(\Om) \rar H^r(\RR^n)$ and $R: H^r(\RR^n) \rar H^r(\Om)$ are, respectively,
the extension and restriction operators, and $\widetilde{J}_\si$ is the Friedrichs mollifier
in $\RR^n$. Notice that, here in contrast to (\ref{restriction}), one does not lose a
half-derivative upon restriction since $\Om$ is of the same dimension as $\RR^n$. 
Denoting by $\p \cdot \p_{s,\RR^n}$ the Sobolev norm in $\RR^n$, 
standard properties of $\widetilde{J}_\si$ give
\begin{gather}
\p \widetilde{J}_\si h \p_{s,\RR^n} \leq \frac{C}{\si^\frac{3}{2}} \p h \p_{s-\frac{3}{2},\RR^n}
\label{standard_prop_mol}
\end{gather}
if $h$ is in $H^{s-\frac{3}{2}}(\RR^n)$. From (\ref{mollifier_bounded_domain}) and (\ref{standard_prop_mol}),
\begin{align}
\begin{split}
\p J_\si \int_0^t P_{\beta_\kappa} \ddot{\be}_\kappa \p_s
& = \p R \widetilde{J}_\si E \int_0^t P_{\beta_\kappa} \ddot{\be}_\kappa \p_s
\leq C \p \widetilde{J}_\si E \int_0^t P_{\beta_\kappa} \ddot{\be}_\kappa \p_{s,\RR^n} \\
& \leq \frac{C}{\si^\frac{3}{2}} \p E \int_0^t P_{\beta_\kappa} \ddot{\be}_\kappa \p_{s-\frac{3}{2},\RR^n} \\
& \leq \frac{C}{\si^\frac{3}{2}} \p  \int_0^t P_{\beta_\kappa} \ddot{\be}_\kappa \p_{s-\frac{3}{2}}
\leq \frac{C T}{\si^\frac{3}{2}} \p  P_{\beta_\kappa} \ddot{\be}_\kappa \p_{s-\frac{3}{2}}.
\end{split}
\nonumber
\end{align}
Thus, in light of (\ref{estimate_P_lower_norm}) 
\begin{gather}
 \p  J_\si \int_0^t P_{\beta_\kappa} \ddot{\be}_\kappa  \p_s \leq
\frac{C T}{\si^\frac{3}{2}} 
\p P_{\beta_\kappa} \ddot{\be}_\kappa  \p_{s-\frac{3}{2}}
\leq \frac{C(\Rin)}{\si^\frac{3}{2} \sqrt{\kappa}},
\label{estimate_mollifier}
\end{gather} 
where $T$ was absorbed into $C(\Rin)$.
Consider
\begin{gather}
\int_0^t  P_{\beta_\kappa} \ddot{\be}_\kappa = 
\int_0^t  P_{\beta_\kappa} \ddot{\be}_\kappa
- J_\si \int_0^t  P_{\beta_\kappa} \ddot{\be}_\kappa 
+    J_\si \int_0^t P_{\beta_\kappa} \ddot{\be}_\kappa .
\nonumber
\end{gather}
and choose a sequence  of  $\si$'s depending on $\kappa$ by
$\si = \si_\kappa = \frac{1}{\kappa^\frac{1}{6}}$. Then $\si_\kappa \rar 0$
as $\kappa \rar \infty$ and 
 (\ref{estimate_mollifier}) gives
\begin{gather}
 \p  J_{\si_\kappa} \int_0^t P_{\beta_\kappa} \ddot{\be}_\kappa \p_s 
 \leq \frac{C(\Rin)}{\kappa^\frac{1}{4}}.
\nonumber
 \end{gather} 
But 
\begin{gather}
\lim_{\si \rar 0} \p \int_0^t  P_{\beta_\kappa} \ddot{\be}_\kappa
- J_\si \int_0^t  P_{\beta_\kappa} \ddot{\be}_\kappa 
\p_s = 0,
\nonumber
\end{gather}
and thus we conclude that
\begin{gather}
\lim_{\kappa \rar \infty} \p \int_0^t  P_{\beta_\kappa} \ddot{\be}_\kappa 
\p_s = 0.
\nonumber
\end{gather}
Combining this with (\ref{basic_estimate_zeta_eta_dot}) gives
$\p \dot{\zeta} - \dot{\eta}_\kappa \p_s \rar 0$ as $\kappa \rar \infty$.
This immediately produces
$\p \zeta - \eta_\kappa \p_s \rar 0$ as well since
$\zeta(t) = \zeta(0) + \int_0^t \dot{\zeta}$, and similarly for $\eta_\kappa$.
From the regularity of $\eta_\kappa$ the convergence is as stated
in theorem \ref{main_theorem}, finishing the proof.

\subsection{Proof of corollary \ref{corollary_convergence}} From our estimates, we immediately
have
\begin{gather}
\p \nabla f_\kappa \p_{s+\frac{3}{2}} \leq \frac{C}{\kappa},
\nonumber
\end{gather}
and
\begin{gather}
\p \nabla \dot{f}_\kappa \p_s \leq \frac{C}{\sqrt{\kappa}},
\nonumber
\end{gather}
so that $\nabla f_\kappa \rar 0$ in $H^{s+\frac{3}{2}}$ and 
$\nabla \dot{f}_\kappa \rar 0$ in $H^s$, as stated. Combined with  decomposition 
(\ref{decomp_eta}), this gives
\begin{gather}
\p \eta_\kappa - \be_\kappa \p_s \rar 0,
\nonumber
\end{gather}
so that $\be_\kappa \rar \zeta$ in $H^s(\Om)$ in view of the convergence part of theorem
\ref{main_theorem}. Similarly, one gets the convergence 
$\dot{\be}_\kappa \rar \dot{\zeta}$ in $H^s(\Om)$ from (\ref{dot_eta_f_beta}).


\newpage
\begin{thebibliography}{ZZZZ}
\bibitem{Adams} Adams, R. A.;
Fournier, J. J. F.  \emph{Sobolev spaces.} Volume 140, Second Edition (Pure and Applied Mathematics). Academic Press; 2 edition (2003).

\bibitem{Ale} Aleksandrov, A. D. \emph{Uniqueness theorem for surfaces in the large. V}.
Amer. Math. Soc. Transl. (2) 21 1962 412-416.

\bibitem{Amborse}
Ambrose, D. M. \emph{Well-posedness of vortex sheets with surface tension}, SIAM J. Math. Anal. 35 (2003), no. 1, 211-244.

\bibitem{AmbroseMasmoudi} Ambrose, D. M.; Masmoudi, N. \emph{The zero surface tension limit of two-dimensional water
waves}, Comm. Pure Appl. Math., 58 (2005), 1287-1315.

\bibitem{BB} Bourguignon, J. P.; Brezis, H. \emph{Remarks on the Euler equation},
Journal of Functional Analysis, Vol.15, 1974, pp. 341-363.

\bibitem{Cas1} Castro, A.; C\'ordoba, D.; Fefferman, C.; Gancedo, F.;
G\'omez-Serrano, J.; \emph{Finite time singularities for the free boundary incompressible Euler equations.} Ann. of Math., 178 (2013), 1061-1134.

\bibitem{Cas2} Castro, A.; C\'ordoba, D.; Fefferman, C.; Gancedo, F.;
 L\'opez-Fern\'andez, M. \emph{Rayleigh-Taylor breakdown for the Muskat problem with applications to water waves.}  Ann. of Math., 175 (2012), 909-948.

\bibitem{Craig} Craig, W. \emph{An existence theory for water waves and the Boussinesq and Korteweg-de Vries scaling limits}, 
Comm. Partial Differential Equations, 10 (1985), no. 8, 787-1003

\bibitem{ChLin} Christodoulou, D.; Lindblad, H. \emph{On the motion of the free surface of a liquid},
 Comm. Pure Appl. Math., {\bf 53} (2000), 1536-1602.

\bibitem{CS} Coutand, D.;  Shkoller, S. \emph{Well-posedness 
of the free-surface incompressible Euler equations with or without surface tension}, J. Amer. Math. Soc. 20 (2007), no. 3, 829-930. 

\bibitem{CSB} Coutand, D.; Shkoller, S. \emph{A simple proof of well-posedness for the free-surface incompressible Euler equations},
 Discrete Contin. Dyn. Syst. Ser. S 3 (2010), no. 3, 429-449.

\bibitem{CS2} Coutand, D.; Shkoller, S. 
\emph{Well-Posedness in Smooth Function Spaces for the Moving-Boundary Three-Dimensional Compressible Euler Equations in Physical Vacuum}, 
Arch. Ration. Mech. Anal. 206 (2012), no. 2, 515-616.

\bibitem{CS3} Coutand, D.; Shkoller, S. 
\emph{ Well-posedness in smooth function spaces for moving-boundary 1-D compressible Euler equations in physical vacuum}, 
Comm. Pure Appl. Math. 64 (2011), no. 3, 328–366. 

\bibitem{CS4} Coutand, D.; Shkoller, S. 
\emph{On the finite-time splash and splat singularities for the 3-D free-surface Euler equations.} Commun. Math. Phys., 325 (2014), 143-183.

\bibitem{CSH} Coutand, D.; Hole, J.; Shkoller, S.
\emph{Well-posedness of the free-boundary compressible 3-D Euler equations with surface tension and the zero surface tension limit},
arXiv:1208.2726 [math.AP]

\bibitem{CSL} Coutand, D.; Lindblad, H.; Shkoller, S. 
\emph{A priori estimates for the free-boundary 3D compressible Euler equations in physical vacuum}, Comm. Math. Phys. 296 (2010), no. 2, 559-587.

\bibitem{Hitch} Di Nezza, E.; Palatucci, G.; Valdinoci, E.
\emph{Hitchhiker's guide to the fractional Sobolev spaces}.
arXiv:1104.4345 [math.FA]

\bibitem{Dis_linear} Disconzi, M. M. \emph{On a linear problem arising in dynamic 
boundaries.} Evolution Equations and Control 
Theory, Vol 3., Number 4, p. 627-644 (2014).

\bibitem{DE2d} Disconzi, M. M.; Ebin, D. G. \emph{On the limit of large surface tension for a fluid motion with free boundary}.  Communications in Partial Differential Equations, 39: 740-779 (2014).

\bibitem{E_manifold}  Ebin, D. G. \emph{The manifold of Riemannian metrics},
 1970 Global Analysis (Proc. Sympos. Pure Math., Vol. XV, Berkeley, Calif., 1968) pp. 11–40, Amer. Math. Soc., Providence, R.I.

\bibitem{E0} Ebin, D. G. \emph{The equations of motion of a perfect fluid with free boundary are not well posed}, Comm. in Partial Diff. Eq., \textbf{12} (10), 1175-1201 (1987).

\bibitem{E1} Ebin, D. G. \emph{Espace des metrique riemanniennes et mouvement des fluids via les varietes d'applications},
 Ecole Polytechnique, Paris, 1972.

\bibitem{E2} Ebin, D. G. \emph{The motion of slightly compressible fluids viewed as a motion with strong constraining force},
 Annals of Math., vol 105, Number 1, 1977, pp 141-200.



\bibitem{E4} Ebin, D. G. \emph{Motion of slightly compressible fluids in a bounded domain I}, 
Comm. Pure Appl. Math. 35 (1982), no. 4, 451-485.

\bibitem{ED} Ebin, D. G.; Disconzi, M. M. \emph{Motion of slightly compressible fluids in a bounded domain II}, arXiv: 1309.0477 [math.AP] (2013). 49 pages.

\bibitem{EM} Ebin, D. G.; Marsden, J. \emph{Groups of diffeomorphisms and the motion of an incompressible fluid},
 Annals of Math., Vol. 92, 1970, pp. 102-163.
 
\bibitem{GT} Gilbarg, D.; Trudinger, N. S. \emph{Elliptic partial differential
equations of second order.} Springer (2001).

\bibitem{Han} Han, Q.;  Hong, J. \emph{Isometric Embedding of Riemannian
Manifolds in Euclidean Spaces.} American Mathematical Society
 (Mathematical Surveys and Monographs) (2006).

\bibitem{Hille} Hille, E.; Phillips, R. S. \emph{Functional analysis and
semigroups.} American Mathematical Society (1957).

\bibitem{IT} Ifrim, M.; Tataru, D. \emph{Two dimensional water waves in holomorphic coordinates II: global solutions.} arXiv:1404.7583 [math.AP].

\bibitem{IonFef}
Ionescu, A.; Fefferman, C;  Lie, V. \emph{On the absence of  “splash” singularities in the case of two-fluid interfaces.} arXiv:1312.2917 (2013).

\bibitem{IonPus} Ionescu, A. D.; Pusateri, F. \emph{Global regularity for 2d water waves with surface tension.} arXiv:1408.4428 [math.AP] (2014).

\bibitem{K_1} Kato, T. \emph{Linear evolution equations of hyperbolic type.}
J. Faculty of Science, University of Tokyo, Sec. I, Vol. XVII, Parts
1 and 2 (1970), 241-258.

\bibitem{K}  Kato, T. \emph{The Cauchy problem for quasi-linear symmetric hyperbolic systems.} Arch. Rational Mech. Anal. 58 (1975), no. 3, 181-205.

\bibitem{Kato_hyp_2} Kato, T. \emph{Linear evolution equations of hyperbolic type, II.}
J. Math. Soc. Japan, Vol 25, No.4 (1973).

\bibitem{K_Linear} Kato, T. \emph{Perturbation theory for linear operators.}
Springer.

\bibitem{KPW} 
K\"ohne, M.; Pr\"uss, J.; Wilke, M. \emph{Qualitative behaviour of solutions for the two-phase Navier-Stokes equations with surface tension.}
Math. Ann. 356 (2013), no. 2, 737-792.

\bibitem{L} Lang, S. \emph{Differentiable manifolds}. Addison-Wesley Reading, Mass. (1972).

\bibitem{Lannes} Lannes, D. \emph{Well-posedness of the water-waves equations}, J. Amer. Math. Soc., 18, (2005) 605-654.

\bibitem{Lin} Lindblad, H. \emph{Well-posedness for the motion of an incompressible liquid with free surface boundary},
 Annals of Mathematics, {\bf 162} (2005), 109-194.

\bibitem{Lin2}  Lindblad, H. \emph{Well-posedness for the linearized motion of an incompressible liquid with free
surface boundary}, Comm. Pure Appl. Math., 56, (2003), 153-197.

\bibitem{LindNor} Lindblad, H.; Nordgren, K. \emph{A priori estimates for the motion of a self-gravitating 
incompressible liquid with free surface boundary}, J. Hyperbolic Differ. Equ. 6 (2009), no. 2, 407-432.


\bibitem{MogSol} Mogilevskii, I. S.;  Solonnikov, V. A. \emph{On the solvability of an evolution free boundary problem for the Navier-Stokes equations in H\"older spaces of functions.} Mathematical
problems relating to the Navier-Stokes equation, 105-181, Ser. Adv. Math.
Appl. Sci., 11, World Sci. Publ., River Edge, NJ, 1992.

\bibitem{Nalimov} Nalimov, V. I. \emph{The Cauchy-Poisson Problem} (in Russian), Dynamika Splosh. Sredy, 18 (1974),104-210.

\bibitem{Munkres} Munkres, J. R. \emph{Topology.} Pearson, 2nd ed. (2000).

\bibitem{P} Palais, R. S. \emph{Seminar on the Atiyah-Singer index theorem}. Ann. of Math. Studies No. 57, Princeton (1965).

\bibitem{PS}  Pr\"uss, J.; Simonett, G. \emph{On the two-phase Navier-Stokes equations with surface tension.} Interfaces Free Bound. 12 (2010), no. 3, 311-345.

\bibitem{Schwartz} Schwartz, J. T. \emph{Non-linear functional analysis.}
Gordon and Breach. (1969).

\bibitem{Sch} Schweizer, B. \emph{On the three-dimensional Euler equations with a free boundary subject to surface tension}. 
Ann. I. H. Poincar\'e -- AN 22 (2005) 753-781.

\bibitem{Se1} Secchi, P. \emph{On the uniqueness of motion of viscous gaseous stars}. Math. Methods Appl. Sci. {\bf 13}, 391 (1990).

\bibitem{Se2} Secchi, P.  \emph{On the motion of gaseous stars in the presence of radiation}. Commun. Part. Diff. Eqs. {\bf 15}, 185 (1990).

\bibitem{Se3} Secchi, P.  \emph{On the evolution equations of viscous gaseous stars}. Ann. Scuola Norm. Sup. Pisa {\bf 36} (1991), 295-318.

\bibitem{ShatahZeng} Shatah, J.; Zeng, C.; \emph{Geometry and a priori estimates for free boundary problems of the Euler's equation},
Communications on Pure and Applied Mathematics Volume 61, Issue 5, pages 698-744, May 2008.

\bibitem{ShatahZeng2} Shatah, J.; Zeng, C.; \emph{Local well-posedness for fluid interface problems,} Arch. Ration. Mech. Anal., Vol. 199, No. 2, 653-705, 2011.

\bibitem{Sol1}  Solonnikov, V. A. \emph{Solvability of the problem of evolution of an isolated amount of a
viscous incompressible capillary fluid.} (Russian) Mathematical questions in the theory
of wave propagation, 14. Zap. Nauchn. Sem. Leningrad. Otdel. Mat. Inst. Steklov.
(LOMI) 140 (1984), 179-186. Translated in J. Soviet Math. 37 (1987).

\bibitem{Sol2}  Solonnikov, V. A. \emph{Unsteady flow of a finite mass of a fluid bounded by a free surface.}
(Russian. English summary) Zap. Nauchn. Sem. Leningrad. Otdel. Mat. Inst. Steklov.
(LOMI) 152 (1986), 137-157. Translation in J. Soviet Math. 40 (1988), no. 5, 672-
686.

\bibitem{Sol3}  Solonnikov, V. A. \emph{Unsteady motions of a finite isolated mass of a self-gravitating
fluid.} (Russian) Algebra i Analiz 1 (1989), no. 1, 207–249. Translation in Leningrad
Math. J. 1 (1990), no. 1, 227-276.

\bibitem{Sol4}  Solonnikov, V. A. \emph{Solvability of a problem on the evolution of a viscous incompressible
fluid, bounded by a free surface, on a finite time interval.} (Russian) Algebra
i Analiz 3 (1991), no. 1, 222-257. Translation in St. Petersburg Math. J. 3 (1992),
no. 1, 189-220.

\bibitem{Sol5}  Solonnikov, V. A. \emph{On the quasistationary approximation in the problem of motion
of a capillary drop.} Topics in Nonlinear Analysis. The Herbert Amann Anniversary
Volume, (J. Escher, G. Simonett, eds.) Birkh¨auser, Basel, 1999.

\bibitem{Sol6}  Solonnikov, V. A. \emph{$L^q$-estimates for a solution to the problem about the evolution
of an isolated amount of a fluid.} J. Math. Sci. (N. Y.) 117 (2003), no. 3, 4237-4259.

\bibitem{Sol7}  Solonnikov, V. A. \emph{Lectures on evolution free boundary problems: classical solutions.}
Mathematical aspects of evolving interfaces (Funchal, 2000), 123-175, Lecture Notes
in Math., 1812, Springer, Berlin, 2003.

\bibitem{Spivak} Spivak, M. \emph{A Comprehensive Introduction to Differential Geometry.} 3rd Edition. Publish or Perish.

\bibitem{TaylorPDE1} M. E. Taylor. 
\emph{Partial Differential Equations I: Basic Theory}
(Applied Mathematical Sciences). Springer.

\bibitem{TaylorPDE3} M. E. Taylor. 
\emph{Partial Differential Equations III: Nonlinear equation}
(Applied Mathematical Sciences). Springer.

\bibitem{TayPSO} M. E. Taylor. \emph{Pseudodifferential Operators.}
Princeton University Press.

\bibitem{Wu} S. Wu. \emph{ Well-posedness in Sobolev spaces of the full water wave problem in 3-D},
  J. Amer. Math. Soc., 12 (1999), 445-495.

\bibitem{WuWaterWavesGlobal} S. Wu. \emph{Global wellposedness of the 3-D full water wave problem},
Inventiones mathematicae, 141 (1), 2011, 125-220.

\bibitem{Yosihara} Yosihara, H. \emph{Gravity Waves on the Free Surface of an Incompressible Perfect Fluid}, Publ. RIMS Kyoto Univ., 18 (1982), 49-96.

\end{thebibliography}
\end{document}